\documentclass[11pt,a4paper]{amsart}
\usepackage[margin=1in]{geometry}
\usepackage[utf8]{inputenc}
\usepackage{amsmath}
\usepackage{amsfonts}
\usepackage{xypic}
\usepackage{amsthm}
\usepackage{mathrsfs}
\usepackage{amssymb}
\usepackage[all]{xy}
\usepackage{color}

\usepackage{hyperref}
\usepackage{graphicx}
\usepackage{wrapfig}
\usepackage{float}
\usepackage[cal=boondox,scr=boondoxo]{mathalfa} 
\graphicspath{ {dibujos/} }
\usepackage{authblk}

\newcommand{\PPp}{\mathbcal{P}}
\newcommand{\UUu}{\mathbcal{U}}
\newcommand{\Jjac}{\mathbcal{J}\!\! \mathbcal{a}\!\! \mathbcal{c}}
\numberwithin{equation}{section}
\newtheorem{theorem}{Theorem}[section]
\newtheorem*{theorem*}{Theorem}
\newtheorem{lemma}[theorem]{Lemma}
\newtheorem{proposition}[theorem]{Proposition}
\newtheorem{corollary}[theorem]{Corollary}

\newtheorem{assumption}{Assumption}

\theoremstyle{remark}
\newtheorem{remark}[theorem]{Remark}

\newtheorem{innercustomgeneric}{\customgenericname}
\providecommand{\customgenericname}{}
\newcommand{\newcustomtheorem}[2]{%
  \newenvironment{#1}[1]
  {%
   \renewcommand\customgenericname{#2}%
   \renewcommand\theinnercustomgeneric{##1}%
   \innercustomgeneric
  }
  {\endinnercustomgeneric}
}

\newcustomtheorem{customthm}{Theorem}
\newcustomtheorem{customprop}{Proposition}
\newcustomtheorem{customlemma}{Lemma}

\newtheoremstyle{TheoremNum}
        {\topsep}{\topsep}              %%% space between body and thm
        {\itshape}                      %%% Thm body font
        {}                              %%% Indent amount (empty = no indent)
        {\bfseries}                     %%% Thm head font
        {.}                             %%% Punctuation after thm head
        { }                             %%% Space after thm head
        {\thmname{#1}\thmnote{ \bfseries #3}}%%% Thm head spec
    \theoremstyle{TheoremNum}

\newcommand{\lie}{\mathfrak}

\DeclareMathOperator{\Jac}{Jac}

\newcommand{\BBB}{\mathrm{(BBB)}}
\newcommand{\BAA}{\mathrm{(BAA)}}

\newcommand{\Spec}{\mathrm{Spec}}

\newcommand{\Hilb}{\mathrm{Hilb}}

\newcommand{\supp}{\mathrm{supp}}
\newcommand{\End}{\mathrm{End}}

\newcommand{\Tor}{\mathrm{Tor}}
\newcommand{\Sym}{\mathrm{Sym}}
\newcommand{\sym}{\mathfrak{S}}

\newcommand{\ad}{\mathrm{ad}}

\newcommand{\id}{\mathbf{1}}
\newcommand{\rk}{\mathrm{rk}}

\newcommand{\GL}{\mathrm{GL}}
\newcommand{\SL}{\mathrm{SL}}

\newcommand{\U}{\mathrm{U}}
\newcommand{\sing}{\mathrm{sing}}

\newcommand{\nod}{\mathrm{nod}}
\newcommand{\sm}{\mathrm{sm}}
\newcommand{\red}{\mathrm{red}}
\newcommand{\Ord}{\mathrm{Ord}}
\newcommand{\st}{\mathrm{st}}
\newcommand{\sst}{\mathrm{sst}}
\newcommand{\pst}{\mathrm{pst}}

\newcommand{\Aa}{\mathcal{A}}

\newcommand{\Dd}{\mathcal{D}}
\newcommand{\Ee}{\mathcal{E}}
\newcommand{\Ff}{\mathcal{F}}
\newcommand{\Gg}{\mathcal{G}}
\newcommand{\Ii}{\mathcal{I}}

\newcommand{\Jj}{\mathcal{J}}

\newcommand{\Ll}{\mathcal{L}}

\newcommand{\Oo}{\mathcal{O}}
\newcommand{\Pp}{\mathcal{P}}

\newcommand{\Ss}{\mathcal{S}}

\newcommand{\Uu}{\mathcal{U}}

\newcommand{\Zz}{\mathcal{Z}}

\newcommand{\Car}{\mathrm{Car}}
\newcommand{\Bor}{\mathrm{Bor}}
\newcommand{\Par}{\mathrm{Par}}
\newcommand{\Uni}{\mathrm{Uni}}

\newcommand{\CCar}{\mathbf{Car}}
\newcommand{\UUni}{\mathbf{Uni}}

\newcommand{\B}{\mathrm{B}}
\newcommand{\C}{\mathrm{C}}

\renewcommand{\H}{\mathrm{H}}
\newcommand{\Hr}{\mathrm{H}_{\ol{r}}}
\renewcommand{\L}{\mathrm{L}}
\newcommand{\M}{\mathrm{M}}
\newcommand{\Mr}{\mathrm{M}_{\ol{r}}}

\newcommand{\N}{\mathrm{N}}
\renewcommand{\P}{\mathrm{P}}

\newcommand{\V}{\mathrm{V}}
\newcommand{\W}{\mathrm{W}}

\newcommand{\Bbor}{\mathbcal{B\! o\! r}}
\newcommand{\MMm}{\mathbcal{M}}

\newcommand{\ol}[1]{\overline{#1}}

\newcommand{\mc}[1]{\mathcal{#1}}

\newcommand{\wt}[1]{\widetilde{#1}}

\newcommand{\spec}{\ol{X}}

\newcommand{\imaginary}{\mathrm{i}}

\newcommand{\Ccc}{\mathbf{Car}}
\newcommand{\Fff}{\mathscr{F}}

\newcommand{\Lll}{\mathscr{L}}

\newcommand{\CC}{\mathbb{C}}

\newcommand{\EE}{\mathbb{E}}

\newcommand{\PP}{\mathbb{P}}

\newcommand{\plonge}{\hookrightarrow}

\newcommand{\quotient}[2]{{\raisebox{.2em}{\thinspace $#1$}\left / \raisebox{-.15em}{ $#2$}\right.}}

\newcommand\Quotient[2]{
        \mathchoice
            {% \displaystyle
                \text{\raise1ex\hbox{\thinspace $#1$}\Big{/} \lower1ex\hbox{$#2$} \thinspace}%
            }
            {% \textstyle
                #1\,/\,#2
            }
            {% \scriptstyle
                #1\,/\,#2
            }
            {% \scriptscriptstyle  
                #1\,/\,#2
            }
    }

\newcommand\GIT[2]{
        \mathchoice
            {% \displaystyle
                \text{\raise1ex\hbox{\thinspace $#1$}\Big{/}\!\!\!\!\Big{/} \lower1ex\hbox{$#2$} \thinspace}%
            }
            {% \textstyle
                #1\,/\,#2
            }
            {% \scriptstyle
                #1\,/\,#2
            }
            {% \scriptscriptstyle  
                #1\,/\,#2
     a       }
    }

\newcommand{\map}[5]{\begin{array}{ccc}   #1  & \stackrel{#5}{\longrightarrow} &  #2  \\  #3 & \longmapsto & #4  \end{array}}

\newcommand{\morph}[6]{\begin{array}{cccc} #6: & #1  & \stackrel{#5}{\longrightarrow} &  #2  \\ & #3 & \longmapsto & #4  \end{array}}

\title[Branes via Borel and other parabolic subgroups]{\bf Branes on the singular locus of the Hitchin system via Borel and other parabolic subgroups}

\author[E. Franco]{Emilio Franco}
\address{E. Franco,
\newline\indent Centro de An\'alise Matem\'atica, Geometria e Sistemas Din\^{a}micos, 
\newline\indent Instituto Superior T\'ecnico, Universidade de Lisboa, 
\newline\indent Av. Rovisco Pais s/n, 1049-001 Lisboa, Portugal}
\email{emilio.franco@tecnico.ulisboa.pt}

\author{Ana Pe\'on-Nieto}
\address{Ana Pe\'on-Nieto \newline\indent Laboratoire de Math\'ematiques J.A. Dieudonn\'e \newline\indent UMR  7351 CNRS \newline\indent Universit\'e de Nice Sophia-Antipolis \newline\indent 06108 Nice Cedex 02, France}
\email{ana.peon-nieto@unice.fr}

\thanks{E. Franco is currently supported by FCT (Portugal) in the framework of the Investigador FCT program. He has previously been supported by project PTDC/MAT- GEO/2823/2014 funded by FCT with Portuguese national funds and FAPESP postdoctoral grant number 2012/16356-6 and BEPE-2015/06696-2 (Brazil). 
\\
A. Pe\'on-Nieto is currently supported by the scheme H2020-MSCA-IF-2019,  Agreement n.	897722 (GoH). She was formerly funded through a  Beatriu de Pin\'os grant n. 2018 BP 332 (H2020-MSCA-COFUND-2017 Agreement n. 801370), a postdoctoral grant associated to the project FP7 - PEOPLE - 2013 - CIG - GEOMODULI
 number: 618471, a postdoctoral contract of the Heidelberg Institute for Theoretical Studies, a MATCH postdoctoral fellowship and the European Research Council under ERC-Consolidator grant 614733.}

\date{\today}



\begin{document}

\begin{abstract}
We study mirror symmetry on the singular locus of the Hitchin system at two levels. Firstly, by covering it by (supports of) $\BBB$-branes, corresponding to Higgs bundles reducing their structure group to the Levi subgroup of some parabolic subgroup $\P$, whose conjectural dual $\BAA$-branes we describe. Heuristically speaking, the latter are given by Higgs bundles reducing their structure group to the unipotent radical of $\P$. Secondly, when $\P$ is a Borel subgroup, we are able to construct a family of hyperholomorphic bundles on the $\BBB$-brane, and study the variation of the dual under this choice. We give evidence of both families of branes being dual under mirror symmetry via an integral functor induced by Fourier--Mukai in the moduli stack of Higgs bundles.
%
%The singular locus of the moduli space of Higgs bundles is covered by a family of hyperholomorphic subvarieties parametrized by parabolic subgroups $\P$ of the general linear group $\GL(n,\CC)$. Each hyperholomorphic subvariety is given by those Higgs bundles that reduce its structure group to the Levi subgroup of $\P$. 
%
%Associated to the Borel subgroup $\B$ one has the Cartan locus given by those Higgs bundles that decompose into direct sum of line Higgs bundles. For each topologically trivial line bundle $\Ll \to X$ we construct a hyperholomorphic line bundle over the Cartan locus. In the physicist's language, this constitutes a $\BBB$-brane that we denote by $\CCar(\Ll)$. From the unipotent radical subgroup of the Borel we are able to construct a complex lagrangian subvariety whose support depends on $\Ll$. Equipped with the trivial flat bundle this becomes a $\BAA$-brane that we denote by $\UUni(\Ll)$. We give evidence of both branes being dual under mirror symmetry, in the sense that an ad-hoc Fourier--Mukai integral functor relates the restriction of the hyperholomorphic bundle of the $\BBB$-brane to a generic Hitchin fibre, with the support of the $\BAA$-brane. 
%
%Finally we study the analogous constructions of $\BBB$-branes and $\BAA$-branes associated to an arbitrary parabolic subgroup $\P$.
\end{abstract}

\maketitle

\tableofcontents

\section{Introduction}

\subsection{Brief description}

In this paper we study the action of mirror symmetry on the singular locus of the moduli space $\M_n$ of Higgs bundles. We proceed first by describing hyperholomorphic subvarieties covering $\M_n^\sing$, those become $\BBB$-branes after specifying a hyperholomorphic bundle on them. Then, we construct complex Lagrangian subvarieties, supporting $\BAA$-branes after being equipped with a flat bundle, and we conjecture that behind these constructions stands a pair of mirror dual branes. Each of the previous pairs of branes is naturally associated to a parabolic subgroup of $\GL(n,\CC)$. When this parabolic is the Borel subgroup we find ourselves over the locus of totally reducible spectral curves. A more complete analysis is possible in this case and we are able to construct families of flat (hence hyperholomorphic) bundles giving rise to $\BBB$-branes. These $\BBB$-branes only intersect Hitchin fibres associated to coarse compactified Jacobians where no Fourier--Mukai transform has been defined. We then consider the Fourier--Mukai transform between the associated stacks and prove that it restricts to a transform whose source is the support of the $\BBB$-branes associated to the Borel subgroup. Our biggest contribution is the description of the behaviour of these $\BBB$-branes under such a transform, showing that it returns a sheaf supported on the complex Lagrangian subvarieties we have previously described.

\subsection{Mathematical background and motivation}

Hitchin introduced in \cite{hitchin-self} Higgs bundles over a smooth projective curve $X$ and soon it was noted that their moduli space $\M_n$ carries a very interesting geometry \cite{hitchin-self, simpson1, simpson2, Nitin}. In particular $\M_n$ can be endowed with a hyperk\"ahler structure $(g, \Gamma_1, \Gamma_2, \Gamma_3)$ \cite{hitchin-self, simpson1, simpson2, donaldson, corlette} and fibres over a vector space $h: \M_n \to \H$ with Lagrangian tori as generic fibres \cite{hitchin_duke}. A natural generalization is to consider Higgs bundles for complex reductive Lie groups other than $\GL(n,\CC)$. After the work of \cite{HT, donagi&gaitsgory, donagi&pantev}, the moduli spaces of Higgs bundles for two Langlands dual groups equipped with the afore mentioned fibrations become \emph{SYZ mirror partners} (as defined by \cite{HT} based on work by \cite{SYZ}) and mirror symmetry is expected to be implemented by a Fourier-Mukai transform relative to the fibres of the Hitchin fibration. In this paper we focus in the case of $\GL(n,\CC)$, which is Langlands self-dual. 

Branes in the Higgs moduli space were introduced in \cite{kapustin&witten} and have since attracted great attention. A {\it $\BBB$-brane} in $\M_n$ is given by a pair $(\N, \Fff, \nabla_\Fff)$, where $\N \subset\M_n$ is a hyperholomorphic subvariety and $(\Fff, \nabla_\Fff)$ a hyperholomorphic sheaf on $\N$. This means that the connection $\nabla_{\Fff}$ on the sheaf $\Fff$ is of type $(1,1)$ with respect to all three complex strutures $\Gamma_1$, $\Gamma_2$, $\Gamma_3$. Additionally, a {\it $\BAA$-brane} is a pair $(S, W, \nabla_W)$ where $S \subset \M_n$ is a subvariety which is complex Lagrangian with respect to the holomorphic symplectic form in complex structure $\Gamma_1$, and $(W, \nabla_W)$ is a flat bundle over $S$. It is conjectured in \cite{kapustin&witten} that mirror symmetry interchanges $\BBB$-branes with $\BAA$-branes. This context has motivated many authors to construct $\BBB$ and $\BAA$-branes \cite{hitchin_char, BS2, BGP, heller&schaposnik, biswas&calvo&franco&garciaP,hitchin_spinors, gaiotto, garciaprada-ramanan, franco&jardim, BS3, B, BBS}. Papers such as \cite{hitchin_spinors, gaiotto, franco&jardim} go a step further by giving evidence of the duality between certain $\BBB$ and $\BAA$-branes, however focusing on the smooth locus of the Hitchin system.

Mirror symmetry is more obscure over singular Hitchin fibres, since it involves autoduality of compactified Jacobians of singular curves. Such autoduality was stated via Fourier--Mukai equivalences by Arinkin \cite{arinkin1, arinkin2} in the case of integral curves, and by Melo, Rapagnetta and Viviani \cite{melo1, melo2} in the case of fine compactified Jacobians. Kass \cite{kass} extended the autoduality to the case of coarse compactified Jacobians, which is the one that concerns us, although his construction does not provide a Fourier--Mukai transform. 

Our main motivation is to extend the study of mirror symmetry for branes to the locus of singular Hitchin fibres. This has been addressed also in some papers that appeared after the first preprint of the present one. In \cite{TorsionBundles}, written by the authors along with Gothen and Oliveira, some pair of $\BBB$ and $\BAA$-branes are considered, noting that the $\BBB$-branes play a crucial role in topological mirror symmetry \cite{HT}. These branes are dense over Hitchin fibres associated to integral curves so Arinkin's Fourier--Mukai transform \cite{arinkin1, arinkin2} is enough to study, in this case, the behaviour of these branes under mirror symmetry. Branco \cite{B} studies the intersection of certain branes with the locus of Hitchin fibres associated to non-reduced curves. In this case, mirror symmetry is discussed in geometrical terms, by dualizing a certain abelian variety inside the non-reduced Hitchin fibres. It is noteworthy to mention the work of Hausel, Mellit and Pei \cite{HMP}, who showed that the pair of branes described by Hitchin in \cite{hitchin_char} satisfy an agreement of certain topological invariants. This gives strong evidence for the duality of these branes, as proposed in \cite{hitchin_char}, where such duality was only checked over the locus of smooth Hitchin fibres.

\subsection{Our work}

%Summarising, in this paper we address two main points: on one hand, the manner mirror symmetry operate on the singular locus. On the other hand, the it does vary when varying the hyperholomorphic bundles. 

We start by constructing a family of $\BBB$-branes and complex Lagrangian subvarieties (support of $\BAA$-branes) indexed by a topologically trivial line bundle $\Ll \to X$. Both lie over the locus of singular Hitchin fibres given by totally reducible spectral curves and both constructions involve the Borel subgroup $\B < \GL(n,\CC)$. 

We shall consider $\Car$, the locus of Higgs bundles whose structure group reduces to the Cartan subgroup $\C<\B$, as the support of our $\BBB$-brane. It is well known that this subvariety is naturally hyperholomorphic (being given by reduction of the structure group to a reductive subgroup), the novel point of this piece of work is the construction of different flat (hence hyperholomorphic) bundles, constructed from a chosen line bundle $\Ll\to X$. Our $\BBB$-brane $\CCar (\Ll)$ consists of $\Car$ equipped with this bundle. %This allows the study of mirror symmetry based on the hyperholomorphic bundle rather than just the support. 
The image of $\Car$ under the Hitchin fibration $h(\Car)$ is the locus totally reducible spectral curves $\ol{X}_b$, making Schaub's spectral correspondence \cite{Schaub} explicit over this subset of the singular locus.

We define as well a complex Lagrangian subvariety $\Uni(\Ll)$ consisting of Higgs bundles whose structure group reduces to $\B$, and whose associated graded bundle is constant and depends on $\Ll$. Thus, this complex Lagrangian subvariety depends on $\Ll\to X$, and, heuristically speaking, parametrizes Higgs bundles that reduce their structure group to the unipotent radical of $\B$.  After specifying a flat bundle over $\Uni(\Ll)$, we shall obtain a $\BAA$-brane.

To study the behaviour of $\CCar(\Ll)$ and $\Uni(\Ll)$ under mirror symmetry one would like to transform $\CCar(\Ll)$ under a Fourier--Mukai transform. These branes are supported on $h(\Car)$, included in the locus of (singular) reducible curves. Then, $\Car$ and $\Uni(\Ll)$ only intersect Hitchin fibers $h^{-1}(b) \cong \overline{\Jac}(\ol{X}_b)$ that are coarse compactified Jacobians, not fine, and therefore a full Fourier--Mukai transform is not known to exist, not even after restricting ourselves to the open subset of the Cartan locus whose associated spectral curves are nodal. Nevertheless, it is possible to construct a Poincar\'e sheaf over the moduli stack of torsion-free sheaves over reducible nodal curves although it is yet not known whether the associated integral functor is a derived equivalence or not. The restriction of this stacky Poincar\'e sheaf to the support of the stacky version of $\CCar(\Ll)$ and the Jacobian can be lifted to a sheaf on the corresponding schemes. We then define the associated integral functor 
\[
\Phi^\Car : D^b \left( \Car \cap \ol{\Jac}(\ol{X}_b) \right ) \longrightarrow
D^b \left( \Jac(\ol{X}_b) \right ).
\]

Our main result (Corollary \ref{co support of the dual brane in the Hitchin fibre}) consists on checking that this functor relates the generic loci of both branes. 

\begin{theorem*} There is an equality
\[
\supp \left ( \Phi^\Car \left ( \CCar(\Ll) |_{\ol{\Jac}(\ol{X}_b)}  \right ) \right ) = \Uni(\Ll) \cap \Jac(\ol{X}_b).
\]
\end{theorem*}

We finish by discussing how this construction can be generalized to a large class of branes in the moduli space $\M_n$ of rank $n$ Higgs bundles covering the whole singular locus. In the $\BBB$-case, the support of these branes correspond to the image of $\M_{r_1} \times \dots \times \M_{r_s}$, or equivalently, the locus of those Higgs bundles reducing its structure group to the Levi subgroup $\GL(r_1, \CC) \times \dots \times \GL(r_s, \CC)$. We observe that these subvarieties cover the singular locus of $\M_n$. The $\BAA$-brane is given by a complex Lagrangian subvariety constructed in a similar way as before, but substituting the Borel subgroup with the parabolic subgroup associated to the partition $n = r_1 + \dots + r_s$. As in the case of the Borel group, we are able to identify the spectral correspondence over the nodal locus.

A word should be said about the possible applications of the present piece of work. The branes hereby described are used in a crucial way in \cite{TorsionBundles} to prove that certain branes are of type $\BAA$. On the other hand, the analysis of spectral data corresponding to reducible spectral curves furnishes a useful tool to study the geometry of these loci.

\subsection{Structure of the paper}

The greater completeness of the analysis for the Borel case is the first reason for the choice of the structure of the paper, presenting first this case, then the case of a general parabolic subgroup. The second reason for this choice is of a more prosaic nature and is linked to the complications in the geometry of these singular loci. Indeed, the singular locus consist of several submanifolds which are nested into one another. The smallest, contained in all the others, is precisely the locus of singular points over totally reducible spectral curves. Thus a good understanding of the singular locus requires as a first step a good understanding of the singular locus over totally reducible spectral curves.

This paper is organized as follows. Section \ref{sc Higgs moduli spaces} gives the necessary background on Higgs moduli spaces and the Hitchin system. In Subsection \ref{sc FM} we address the construction of the Poincar\'e sheaf over the moduli stack of torsion-free rank $1$ sheaves on nodal reduced curves. This construction is a natural generalization of that of \cite{arinkin2} and makes part of unpublished work of Arinkin and Pantev \cite{pantev}. The detailed description of this construction is included in Section \ref{sc FM} for the sake of completeness of our paper.

In Section \ref{sc totally reducible} we study the locus of singular Hitchin fibres associated to totally reducible spectral curves. We prove that the preimage of this locus under $h$ coincides with the locus of Higgs bundles whose structure group reduces to the Borel subgroup (Proposition \ref{pr Bor 1}) and describe the associated spectral data (Propositions \ref{pr description line bundles on V^nod} and \ref{pr Bor 2}).

We provide the construction of the $\BBB$-brane $\CCar(\Ll)$ in Section \ref{sect Cartan}. We consider the Cartan locus, $\Car$, given by those Higgs bundles whose structure group reduces to the Cartan subgroup $\C \cong \left(\CC^\times\right)^n<\GL(n,\CC)$. The Cartan locus is given by the image of $c:\Sym^n(\M_1) \hookrightarrow \M_n$, where $\M_1$ is the rank one Higgs moduli space. Also, we prove that the choice of a topologically trivial line $\Ll$ bundle on $X$ yields a hyperholomorphic bundle on $\Car$. This produces the $\BBB$-brane $\CCar(\Ll)$ (cf. Proposition \ref{prop Car}). Finally, we analyze the restriction of the brane $\CCar(\Ll)$ to a generic Hitcin fibre (Proposition \ref{pr intersection of the BBB with the Hitchin fibre}), which is crucial to study the behaviour of $\CCar(\Ll)$ under mirror symmetry. 

Section \ref{sc BAA brane} addresses the construction and description of the complex Lagrangian subvariety $\Uni(\Ll)$, supporting a $\BAA$-brane. $\Uni(\Ll)$ is defined as the subvariety of the locus of all the Higgs bundles reducing to the Borel subgroup $\B$ whose underlying vector bundle project to a certain $\C$-bundle determined by $\Ll$. Then, we prove that $\Uni(\Ll)$ is isotropic by gauge considerations, closed and half-dimensional, hence Lagrangian (Theorem \ref{tm uni Lagrangian}). We finish this section by studying the spectral data of the points of $\Uni(\Ll)$ in Proposition \ref{pr filtration for L-normalized}. 

We have at this point a description of the generic restriction of $\CCar(\Ll)$ and $\Uni(\Ll)$ to a generic Hitchin fibre. In this case, the generic Hitchin fibres are isomorphic to the coarse compactified Jacobian of reduced but reducible curves. We study in Section \ref{sc duality} the transformation of the first under a Fourier--Mukai integral functor. To deal with the lack of a Poincar\'e sheaf over coarse compactified Jacobians, we consider the Poincar\'e sheaf over the associated moduli stack that we reviewed in Section \ref{sc FM} and observe in Proposition \ref{pr meaning of Pp^Car} that its restriction to $\Car$ and the Jacobian provides a sheaf $\Pp^\Car$. It is then natural to study the behaviour of $\CCar(\Ll)$ under the Fourier--Mukai integral functor constructed with $\Pp^\Car$, which we do. We obtain that the generic restriction of $\CCar(\Ll)$ to a Hitchin fibre is sent to a sheaf over $\Uni(\Ll)$ (Corollary \ref{co support of the dual brane in the Hitchin fibre}). This lead us to conjecture that the $\BBB$-brane $\CCar(\Ll)$ is dual under mirror symmetry to a $\BAA$-brane supported on $\Uni(\Ll)$.  

In Section \ref{sc singular locus branes} we adapt the above results to arbitrary parabolic subgroups. Given a partition $n=r_1+\dots+r_s$ we consider the associated parabolic subgroup $\P_{\ol{r}}<\GL(n,\CC)$ with Levi subgroup $\L_{\ol{r}}<\P_{\ol{r}}$. In Section \ref{sc Levi BBB} we consider the subvariety $\Mr$ of $\M_n$, consisting  of Higgs bundles whose structure group reduces to $\L_{\ol{r}}$, and describe the intersection with generic Hitchin fibers (Proposition \ref{prop spectral data levi}). The variety $\Mr$ is a complex subscheme for $\Gamma_1$, $\Gamma_2$ and $\Gamma_3$, hence the support of a $\BBB$-brane. By varying the partition $\ol{r}$, we produce families of branes covering the strictly semistable locus of $\M_n$. On the other hand, in Section \ref{sect BAA para} we consider $\Uni^{\ol{r}}(E_1,\dots,E_s)$, consisting of Higgs bundles with structure group reducing to $\P_{\ol{r}}$ and fixed associated graded bundle $\bigoplus_{i=1}^s E_i$. We prove that under the right conditions on $\ol{E}$, this is a Lagrangian submanifold (Theorem \ref{thm Unir lagrangian}), and so a choice of flat bundle on it produces a $\BAA$-brane. The imposed hypotheses are related to the existence of a hyperholomorphic bundle on the hypothetical dual $\Mr$ (see Remark \ref{rk assumption reasonable}). A look at the spectral data of both $\Mr$ and $\Uni^{\ol{r}}(\ol{E})$, as well as the comparison with the case $\P_{(1,\dots,1)}$, indicates the existence of a duality. 

\

\noindent{\it Acknowledgements.}
We would like to thank P. Gothen, M. Jardim,  A. Oliveira and C. Pauly for their kind support and inspiring conversations. Many thanks to J. Heinloth for reading a preliminary version of this paper and pointing out some mistakes. We are indebted to A. Wienhard, whose support and hospitality made this project possible.

\section{Preliminaries}\label{sect preliminaries}

%\subsection{Non-abelian Hodge theory}
%\label{sc Mm hyperkahler}

\subsection{Higgs bundles and their moduli}
\label{sc Higgs moduli spaces}

Let $X$ be a smooth projective curve over $\CC$. A Higgs bundle over $X$ is a pair $(E, \varphi)$ given by a holomorphic vector bundle $E$ over $X$ and a Higgs field $\varphi \in H^0(X, \End(E) \otimes K)$, which is a holomorphic section of the endomorphisms bundle twisted by the canonical bundle $K$ of $X$
 \cite{hitchin-self, Si, simpson1, simpson2}. 

A Higgs bundle $(E,\Phi)$ of trivial degree is stable ({\it resp.} semistable) if every $\Phi$-invariant subbundle $F \subset E$ has negative ({\it resp.} non-positive) degree, and it is polystable if it is semistable and decomposes as a direct sum of stable Higgs bundles. The moduli space of rank $n$ and degree $0$ semistable Higgs bundles on $X$ was constructed in \cite{hitchin-self, simpson1, simpson2, Nitin}. We review this construction in the following paragraphs. 

Fix a topological bundle $\EE$ of degree $0$ on $X$ and consider the space $\Aa$ of holomorphic structures 
on $\EE$. This is an affine space modelled on $\Omega^{0,1}(X,\ad(\EE))$ whose cotangent bundle is
$$
T^* \! \Aa=\Aa\times \Omega^{0}(X,\ad(\EE)\otimes K),
$$
where we have identified $\ad(\EE)$ and its dual by means of the Killing form (rather, a non degenerate extension of it to the center, to which we will henceforth refer as Killing form). Given a Hermitian metric $h$ on $\EE$ let us denote its Chern connection by $\nabla_h$. We consider the following conditions for pairs:
\begin{enumerate}
  \item \label{it Hitchin eq} exists a Hermitian metric $h$ such that $\nabla_h^2+[\varphi,\varphi^{*_h}]=0$, 
  \item \label{it holomorphicity} $\ol{\partial}_A(\varphi)=0,$
  \item \label{it dual holomorphicity} $\partial_{A,h}(\varphi^{*,h})=0.$
\end{enumerate}
%\noindent
Observe that condition \eqref{it holomorphicity} implies that the pair determines a Higgs bundle and in that case \eqref{it dual holomorphicity} is automatically satisfied for any choice of metric $h$. We shall denote by $(T^*\Aa)_H$ the subset of solutions to \eqref{it holomorphicity} (and, therefore, to \eqref{it dual holomorphicity}). Condition \eqref{it Hitchin eq} is known as the Hitchin equation and it follows from \cite{hitchin-self, simpson1, simpson2} that a Higgs bundle is polystable if and only if \eqref{it Hitchin eq} holds, so we will write $(T^*\!\Aa)^\pst_H$ for the locus of pairs satisfying simultaneously \eqref{it Hitchin eq} and \eqref{it holomorphicity} (hence \eqref{it dual holomorphicity} as well). Note that we have $(T^*\! \Aa)^\st_H \subset (T^*\! \Aa)^\pst_H \subset (T^*\! \Aa)_H^\sst$, where $\st$ and $\sst$ stand for stable and semistable Higgs bundles. These loci are all preserved by the action of the complex gauge group,
$$
\mc{G}=\Omega^0(X,\mathrm{Aut}(\EE)),
$$
and $(T^*\! \Aa)_H^\sst$ and $(T^* \! \Aa)_H^{\pst}$ classify semistable and closed orbits, respectively. The moduli space of semistable Higgs bundles over $X$ of rank $n$ and trivial degree is identified with  
\begin{equation} \label{eq def M_n}
\M_n\cong (T^* \! \Aa)_H/\!\!/\mc{G}=(T^* \! \Aa)^{\pst}_H/\mc{G},
\end{equation}
where the double quotient denotes the GIT quotient. This is a quasi-projective variety of dimension 
\begin{equation} \label{eq dim M_n}
\dim \M_n = 2n^2(g-1) + 2,
\end{equation}
whose points represent isomorphism classes of polystable Higgs bundles and the smooth locus is given by the locus of stable Higgs bundles \cite{simpson2}. The geometry of $\M_n$ is surprisingly rich. In particular, it can be equipped with a hyperk\"ahler structure and becomes an integrable system by means of the Hitchin fibration.

We shall first study the hyperk\"ahler structure of $\M_n$. Let us fix a particular Hermitian metric $h_0$ on the topological bundle $\EE$, this choice determines a Hermitian metric $\eta$ on $T^*\Aa$. Let 
$$
\mc{G}_0=\Omega^0(X,\mathrm{Aut}(\EE, h_0)),
$$
be the {\it unitary gauge group} of automorphisms of $\EE$ preserving the metric $h_0$. We can see that $\eta$ is preserved by $\mc{G}_0$. Also, one can naturally define three complex structures $\widetilde{\Gamma}_1$, $\widetilde{\Gamma}_2$ and $\widetilde{\Gamma}_3$ on $T^* \! \Aa$ satisfying the quaternionic relations, together with a hyperk\"ahler metric preserved by $\mc{G}_0$. This action defines a moment map $\mu_i$ associated to each of the complex structures $\widetilde{\Gamma}_i$, and one can see that $\eta$ is hyperk\"ahler  with respect to them. One can see that the vanishing of $\mu_1$ coincides with equation \eqref{it Hitchin eq}, the vanishing of $\mu_2$ with \eqref{it holomorphicity} and the vanishing of $\mu_3$ with \eqref{it dual holomorphicity}. Therefore, the moduli space of Higgs bundles is identified with the hyperholomorphic quotient,
$$
\M_n \cong \Quotient{\mu_1^{-1}(0) \cap \mu_2^{-1}(0) \cap \mu_3^{-1}(0)}{\Gg_0},
$$
as it follows from \cite{hitchin-self, simpson1, simpson2}. The complex structures $\widetilde{\Gamma}_i$ descend to complex structures $\Gamma_i$ in the quotient and so does the hyperk\"ahler metric $\eta$, defining a hyperk\"ahler structure on $\M_n$. Observe that natural the complex structure in $\M_n$ obtained by the identification \eqref{eq def M_n} coincides with $\Gamma_1$. Additionally, \cite{donaldson, corlette} proved that the moduli space of rank $n$ flat connections on the $C^\infty$ vector bundle $\EE$ over $X$ of degree $0$ is isomorphic to the above hyperk\"ahler quotient equipped with the complex structure $\Gamma_2$.

The hyperk\"ahler structure defined on $\M_n$ induces a holomorphic $2$-form $\Omega_1 = \omega_2 + \mathrm{i} \omega_3$ on $\M_n$, where $\omega_2$ and $\omega_3$ are the K\"ahler forms associated to $\Gamma_2$ and $\Gamma_3$. We next give the expression of $\Omega_1$ by means of the gauge theoretic construction of $\M_n$. 
Let $(\partial_A,\varphi)\in (T^* \! \Aa)^{\st}_H$, and consider two tangent vectors 
$$
(\dot{A}_i,\dot{\varphi}_i)\in T_{(\partial_A,\varphi)}T^* \! \Aa \qquad i=1,2
$$ 
we have
\begin{equation}\label{eq Omega1}
\Omega_1\left((\dot{A}_1,\dot{\varphi}_1),(\dot{A}_2,\dot{\varphi}_2)\right)=\int_X \dot{A}_1\dot{\wedge}\dot{\varphi}_2-\dot{A}_2\dot{\wedge}\dot{\varphi}_1.
\end{equation}
where to define the wedge product $\dot{\wedge}$, we identity $\Omega^{0,1}(X,\ad(\EE))\cong (\Omega^{0}(X,\ad(\EE)) \otimes \Omega^{0,1}_X)$
and $\Omega^{0}(\ad(\EE)\otimes K)\cong (\Omega^{0}(\ad(\EE) \otimes \Omega^{1,0}_X)$,
and for $Z_i\otimes \omega_i$, $i=1,2$, $Z_i\in\Omega^0(X,\ad(\EE))$, $\omega_i\in \Omega^1(X)$, we set 
$$
(Z_1\otimes \omega_1)\dot{\wedge}(Z_2\otimes \omega_2)=\langle Z_1,Z_2\rangle\otimes\omega_1\wedge\omega_2
$$
with $\langle \ ,\ \rangle$ being the Killing form.

%\subsection{The Hitchin fibration}
%\label{sc Hitchin map}

We recall now the Hitchin fibration and spectral construction given in \cite{hitchin_duke, BNR}. Let $(q_1, \dots, q_n)$ be a base of the algebra
$\CC[\mathfrak{gl}(n,\CC)]^{\GL(n,\CC)}$ of regular functions on $\mathfrak{gl}(n,\CC)$  invariant under the adjoint action of $\GL(n,\CC)$. We choose them so that $\deg(q_i) = i$. The Hitchin map is defined by
$$
\morph{\M_n}{\H := \bigoplus_{i = 1}^n H^0(X,K^i)}{(E,\varphi)}{\left ( q_1(\varphi), \dots, q_n(\varphi) \right ).}{}{h}
$$
It is a surjective proper morphism \cite{hitchin_duke,Nitin} endowing the moduli space with the structure of an algebraically completely integrable system. In particular, its generic fibers are abelian varieties and every fiber is a compactified Jacobian \cite{simpson2,Schaub}. To describe these, consider the total space $|K|$ of the canonical bundle and the obvious algebraic surjection $\pi: |K| \to X$. We note that the pullback bundle $\pi^*K \to |K|$ admits a tautological section $\lambda$. Given an element $b \in \H$, with $b = (b_1, \dots, b_n)$, we construct the {\it spectral curve} $\ol{X}_b \subset |K|$ by considering the vanishing locus of the section of $\pi^*K^n$
\begin{equation} \label{eq equation spectral curve}
\lambda^n + \pi^*b_1 \lambda^{n-1} + \dots + \pi^*b_{n-1} \lambda + \pi^*b_n.
\end{equation}
The restriction of $\pi: |K| \to X$ to $X_b$ is a ramified degree $n$ cover that which by abuse of notation we also denote by
$$
\pi : \ol{X}_b \longrightarrow X.
$$
Since the canonical divisor of the symplectic surface $|K|$ is zero and $\ol{X}_b$ belongs to the linear system $|nK|$, one can compute the arithmetic genus of $\ol{X}_b$,
\begin{equation} \label{eq genus of the spectral curve}
g \left (\ol{X}_b \right) = 1 + n^2(g-1).
\end{equation}
By Riemann-Roch, the rank $n$ bundle $\pi_*\Oo_{\ol{X}_b}$ is has degree
$$
\deg(\pi_*\Oo_{\ol{X}_b}) = - (n^2 - n)(g-1).
$$
Given a torsion-free rank one sheaf $\Ff$ over $\ol{X}_b$ of degree $\delta$, where 
\begin{equation}\label{eq delta}
\delta := n(n - 1)(g-1),
\end{equation}
we have that $E_\Ff := \pi_*\Ff$ is a vector bundle on $X$ of rank $n$ and degree $0$. Since $\pi$ is an affine morphism, the natural $\Oo_{|K|}$-module structure on $\Ff$, given by understanding $\Ff$ as a sheaf supported on $|K|$, corresponds to a $\pi_*\Oo_{|K|} = \Sym^\bullet(K^*)$-module structure on $E_\Ff$. Such structure on $E_\Ff$ is equivalent to a Higgs field
\begin{equation} \label{eq Higgs field}
\varphi_\Ff : E_\Ff \longrightarrow E_\Ff \otimes K.
\end{equation}
As expected, one has that
$$
h \left ( (E_\Ff, \varphi_\Ff) \right ) = b.
$$

A stability notion may be defined for a torsion-free sheaf $\Ff$ of rank one on the curve $\ol{X}_b$. If $\ol{X}_b$ is reduced and irreducible (integral) then $\Ff$ is automatically stable. For reduced but reducible curves, \cite[Th\'eor\`eme 3.1]{Schaub} gives an easy characterization of semistability, modulo some corrections pointed out in \cite[Remark 4.2]{CL} and \cite[Section 2.4]{dC}. A torsion-free rank one sheaf $\Ff$ on $\ol{X}_b$ of degree $\delta$ is stable (resp. semi-stable) if and only if for every closed sub-scheme $Z \subset \ol{X}_b$ pure of dimension one has that
\begin{equation} \label{eq Schaub}
\deg_Z \Ff_Z \, > \, (n_Z^2-n_Z )(g - 1) \quad \textnormal{(resp. $\geq$),} 
\end{equation}
where $\Ff_Z := \Ff|_Z/\Tor(\Ff|_Z)$ and $n_Z = \rk(\pi_* \mathcal{O}_Z)$. One can easily check that every line bundle is stable so the Jacobian $\Jac^\delta(\ol{X}_b)$ is contained inside the moduli space of semistable torsion free rank $1$ degree $\delta$ sheaves on $\ol{X}_b$. Furthermore, the former is projective (see \cite{simpson1}) what explains that we refer to it as the {\it compactified Jacobian} and denote it by $\ol{\Jac}^{\, \delta}(\ol{X}_b)$.

The previous construction provides a one-to-one correspondence between rank $1$ torsion-free sheaves over a certain spectral curve and Higgs bundles over the corresponding point of the Hitchin base. Furthermore, stability is preserved under such correspondence. 

\begin{theorem}[\cite{simpson2, Schaub}]\label{thm spectral corresp}
\label{tm hitchin fibre = Pic}
A torsion-free rank one sheaf $\Ff$ on the spectral curve $\ol{X}_b$ is stable (resp. semistable, polystable) if and only if the corresponding Higgs bundle $(E_\Ff, \varphi_\Ff)$ on $X$ is stable (resp. semistable, polystable). Hence, the Hitchin fibre over $b \in \H$ is isomorphic to the moduli space of semistable torsion-free rank one sheaves of degree $\delta = (n^2 - n)(g-1)$ over $\ol{X}_b$, 
$$
h^{-1}(b) \cong \overline{\Jac}^{\,\, \delta}\left( \ol{X}_b \right).
$$
\end{theorem}

For the case of trivial degree, one can construct a section of the Hitchin fibration, named \emph{Hitchin section}, associated to any line bundle $\Jj \in \Jac^{\, \delta/n}(X)$. This section is constructed by assigning to each $b\in B$ the Higgs bundle whose spectral data is the line bundle $\pi^* \Jj$ over the spectral curve $\ol{X}_b$. In other words, we have a morphism
\begin{equation}\label{eq Hitchin section}
\morph{\H_n}{\M_n}{b}{(E_{(\Jj,b)}:= \pi_{*}\pi^{\ast}\Jj,\varphi_{(\Jj,b)}),}{}{\Sigma_{\Jj}}
\end{equation}
where $\varphi_{(\Jj,b)} = \varphi_{E_{(\Jj,b)}}$ as defined in \eqref{eq Higgs field}. One can check that the push-forward of the trivial sheaf of any spectral curve is $\bigoplus_{i = 0}^{n-1} K^{-i}$, applying the projection formula one has
\begin{equation} \label{eq description of $ of Hitchin section}
E_{(\Jj,b)} \cong \Jj \otimes \pi_* \Oo_{\ol{X}_b} \cong \Jj \otimes \left ( \bigoplus_{i = 0}^{n-1} K^{-i} \right ).
\end{equation}
for all $b \in \H_n$.

When studying mirror symmetry beyond the generic locus, one is quickly brought to considering the moduli stack of Higgs bundles. We thus  finish this section with some elements about the geometry of the moduli stack $\MMm_n$ of Higgs bundles of rank $n$ and trivial degree over the smooth projective curve $X$, and its relation with the moduli space $\M_n$.

Let us recall that the stack $\MMm_n$ contains an open set $\MMm_n^{\sst}$ of semistable objects.

\begin{theorem}{\cite{AHH}}\label{thm:stack_Higgs}
The moduli space $\M_n$ is a good moduli space  for $\MMm_n^{\sst}$ in the sense of \cite{Alper}. That is,  there exists
a quasi-compact morphism 
$$
\Psi: \MMm_n^{\sst}\longrightarrow\M_n
$$
such that the pushforward functor is exact and induces an isomorphism of sheaves $\Psi_*\Oo_{\MMm}\cong\Oo_{\M}$.
\end{theorem}
The notion of a good moduli space recovers the usual properties of good  quotients of finite dimensional varieties by group actions \cite{SeshadriQuot, Newstead}. In particular, $\Psi$ is surjective and universally closed, and  $\M_n$ has the quotient-topology. 

The proof of Theorem \ref{thm:stack_Higgs} combines a number of results: Alper proves that the stack of bundles has a good moduli space \cite[Theorem 13.6]{Alper}. In \cite[\S 1.F]{Heinloth}, Heinloth explained how the classical stability notion for bundles can be seen in terms of $\Theta$-stability (notion developed also independently by Halpern-Leistner \cite{HalpernInst}). As explained in \cite[\S 6]{AHH}, one may deduce a similar result for Higgs bundles, so $\MMm_n^{\sst}$ are Hilbert-Mumford semistable objects for a suitable line bundle.  Theorem C in {\it loc.cit.} implies the  existence of a good moduli space for $\MMm_n^{\sst}$.

\subsection{Arinkin's Poincar\'e sheaf and Fourier--Mukai transform}
\label{sc FM}

Arinkin constructed a Poincar\'e sheaf \cite{arinkin2} on the compactified Jacobian of an integral curve with planar singularities, yielding a Fourier--Mukai transform between these spaces and their duals. This was generalized by Melo, Rapagnetta and Viviani \cite{melo1,melo2} to any fine compactified Jacobian of a reduced curve. The universal sheaf for the fine compactified Jacobian is a crucial piece in Arinkin's construction and, because of this, no Poincar\'e sheaf has been constructed for coarse compactified Jacobians which is the situation that concern us in this paper. Nevertheless, Arinkin's methods adapt naturally to moduli stacks as we will review in this section. The construction of a Poincar\'e sheaf over the moduli stack of torsion-free rank $1$ sheaves over a reducible planar curve makes part of unpublished work by Arinkin and Pantev \cite{pantev} where they conjecture that the associated Fourier--Mukai transform gives rise to self-duality of the moduli stack. A sketch of the construction appears in the preprints \cite{maoli1} and \cite{maoli2}. 

Here we restrict to the case of nodal curves. We do so because for these curves the construction of the Poincar\'e sheaf is considerably simpler than in the case of an arbitrary reducible curve (see \cite[Section 4.3]{arinkin2}).

Let $\overline{X}$ be a connected reduced curve with at most nodal singularities and pick an ample line bundle $\Oo_{\overline{X}}(1)$ on it. Let $\overline{\Jjac}^{\, \delta}(\ol{X})$ be the moduli stack of rank $1$ torsion-free sheaves over $\overline{X}$ and denote by $\UUu \to \ol{X} \times \overline{\Jjac}^{\, \delta}(\ol{X})$ the associated universal sheaf. Denote also by $\Jjac^\delta(\overline{X})$ the substack of those sheaves that are invertible ({\it i.e.} line bundles), %and by $\Jac^\delta(\overline{X})$ and recall that it is a fine moduli space (see for instance \cite[Theorem 2, Section 8.2]{neron}) 
and by $\UUu^0 \to \overline{X} \times \Jjac^\delta(\overline{X})$ the restriction of the universal bundle to it. 

Recall that the Hilbert scheme is a fine moduli space represented by a universal subscheme $\Zz_N \subset \overline{X} \times \Hilb^N(\overline{X})$. Write $\Ii_Z$ for the ideal sheaf associated to the zero dimensional subscheme $Z \subset \ol{X}$ and $\Ii_{\Zz_N} \to \overline{X} \times \Hilb^N(\overline{X})$ for the ideal sheaf associated to the universal subscheme. Since $\ol{X}$ is a nodal curve, we have that $\Ii_Z^\vee$ is a torsion-free sheaf. One can use the universal subscheme $\Zz_m := \Zz_{N_m}$ to construct the associated {\it Abel-Jacobi map} 
\[
\morph{\Hilb^{N_m}(\overline{X})}{\overline{\Jjac}^{\delta}(\overline{X})}{Z}{\Ii_Z^\vee \otimes \Oo_{\overline{X}}(-m),}{}{\alpha_m}
\]
where $N_m = m \deg \Oo_{\overline{X}}(1) + \delta$. Note that $\alpha_m$ is given by 
\begin{equation} \label{eq definition of Uu_m}
\Ii_{\Zz_m}^\vee \otimes q_m^*\Oo_{\overline{X}}(-m) \to \overline{X} \times \Hilb^{N_m}(\overline{X}), 
\end{equation}
where $q_m$ denotes the projection $\overline{X} \times \Hilb^{N_m}(\overline{X}) \to \overline{X}$. Denote by $\Hilb^{N_m}(\overline{X})'$ the open subset of $\Hilb^{N_m}(\overline{X})$ given by those zero dimensional subschemes $Z \subset S$ that can be embedded in a smooth curve. Define $W_m$ to be the open subset of $\Hilb^{N_m}(\overline{X})'$ given by those subschemes $Z$ whose ideal sheaf $\Ii_Z$ satisfies the condition $H^1(\overline{X}, \Ii_Z^\vee) = 0$. For any positive integer $r$, we set $W^r := \bigsqcup_{m=r}^{\infty} W_m$ and $\alpha^r := \prod_{m=r}^{\infty} \alpha_{m}|_{W_m}$.

The following is well known although it appears in the literature \cite{altman&kleiman, arinkin2, melo0} in different forms than how we present it here. 

\begin{proposition} \label{pr abel-jacobi atlas}
Let $\overline{X}$ be a connected reduced curve with at most nodal singularities. For any $r$, the Abel--Jacobi map induces a smooth atlas 
\[
\alpha^r : W^r \to \overline{\Jjac}^{\delta}(\overline{X})
\]
for the Artin stack $\overline{\Jjac}^{\delta}(\overline{X})$. Using this atlas, the universal sheaf is $\{ \Uu_m \to \overline{X} \times W_m \}_{m = r}^\infty$ where the $\Uu_m$ are given by restricting the sheaves \eqref{eq definition of Uu_m} to $\overline{X} \times W_m$. 
\end{proposition}

%\begin{proof}
%This is well known although it appears in the literature \cite{altman&kleiman, arinkin2, melo0} in different forms than how we present it here.
%\end{proof}

Now we construct the Poincaré bundle over the product $\overline{\Jjac}^\delta(\overline{X}) \times \Jjac^\delta(\overline{X})$. Given a flat morphism $f : Y \to S$ whose geometric fibres are curves, for any $S$-flat sheaf $\Ee$ on $Y$, we can construct the determinant of cohomology $\Dd_f(\Ee)$ (see for instance (see \cite{knudsen&mumford} and \cite[Section 6.1]{esteves})), which is an invertible sheaf on $S$ constructed locally as the determinant of complexes of free sheaves locally quasi-isomorphic to $Rf_*\Ee$. Consider the triple product $\ol{X} \times \overline{\Jjac}^{\, \delta}(\ol{X}) \times \Jjac^\delta(\ol{X})$ and denote by $f_{ij}$ the projection to the product of the $i$-th and $j$-th factors. We define the Poincar\'e bundle $\PPp \to \overline{\Jjac}^{\, \delta}(\ol{X}) \times \Jjac^\delta(\ol{X})$ as the invertible sheaf
\begin{equation} \label{eq definition Poincare bundle}
\PPp = \Dd_{f_{23}} \left ( f_{12}^*\UUu \otimes f_{13}^* \UUu^0  \right ) \otimes \Dd_{f_{23}} \left ( f_{13}^* \UUu^0  \right )^{-1} \otimes \Dd_{f_{23}} \left ( f_{12}^* \UUu  \right )^{-1}.
\end{equation}
Given a degree $\delta$ line bundle $J$ over $\ol{X}$, denote by $\PPp_J := \PPp|_{\overline{\Jjac}^{\, \delta}(\ol{X}) \times \{ J \}}$ the restriction of $\PPp$ to the slice corresponding to $J$. In fact, if we consider the obvious projections $f_1 : \ol{X} \times \overline{\Jjac}^{\, \delta}(\ol{X}) \to \ol{X}$ and $f_2 : \ol{X} \times \overline{\Jjac}^{\, \delta}(\ol{X}) \to \overline{\Jjac}^{\, \delta}(\ol{X})$, one has (see \cite[Lemma 5.1]{melo1} for instance) that 
\begin{equation} \label{eq description of Pp_J}
\PPp_J = \Dd_{f_2} (\UUu \otimes f_1^*J) \otimes \Dd_{f_2}(f_1^*J)^{-1} \otimes \Dd_{f_2}(\UUu)^{-1}.
\end{equation}

\begin{remark}
If $\overline{X}$ is a smooth irreducible curve, rank $1$ torsion free sheaves over it are simple line bundles so 
\[
\overline{\Jjac}^{\delta}(\overline{X}) \cong \Jjac^{\delta}(\overline{X}) \cong \left [ \quotient{\Jac^{\, \delta}(\overline{X})}{\CC^*}\right ],
\]
and $\PPp$ pulls-back to a bundle $\Pp \to \Jac^\delta(\overline{X}) \times \Jac^\delta(\overline{X})$ under the projection $\Jac^{\, \delta}(\overline{X}) \to [\Jac^{\, \delta}(\overline{X})/\CC^*]$. The integral functor associated to $\Pp$ is a derived equivalence of categories \cite{mukai}, the Fourier--Mukai transform. 
\end{remark}

One can reverse the roles of $\Jjac^\delta(\overline{X})$ and $\overline{\Jjac}^\delta(\overline{X})$ in \eqref{eq definition Poincare bundle} to obtain a Poincar\'e bundle over $\Jjac^\delta(\overline{X}) \times \overline{\Jjac}^\delta(\overline{X})$ which coincides with the one defined in \eqref{eq definition Poincare bundle} over $\Jjac^\delta(\overline{X}) \times \Jjac^\delta(\overline{X})$. We then see that the Poincar\'e bundle extends naturally to a bundle over
\[
\left ( \overline{\Jjac}^\delta(\overline{X}) \times \overline{\Jjac}^\delta(\overline{X}) \right )^\sharp := \left ( \overline{\Jjac}^\delta(\overline{X}) \times \Jjac^\delta(\overline{X}) \right ) \cup \left ( \Jjac^\delta(\overline{X}) \times \overline{\Jjac}^\delta(\overline{X}) \right )
\]
that we denote by $\PPp^\sharp$. Following \cite{arinkin2}, it is possible to extend $\PPp^\sharp$ even further to a Cohen--Macaulay sheaf over  $\overline{\Jjac}^{\, \delta}(\ol{X}) \times \overline{\Jjac}^{\, \delta}(\ol{X})$, as we will see below.
%satisfying 
%\[
%\overline{\PPp} = \imath_*\PPp^\sharp,
%\]
%where $\imath$ denotes the inclusion of $\left ( \overline{\Jjac}^\delta(\overline{X}) \times \overline{\Jjac}^\delta(\overline{X}) \right )^\sharp$ into $\overline{\Jjac}^\delta(\overline{X}) \times \overline{\Jjac}^\delta(\overline{X})$.
%We call such $\overline{\PPp}$ the Poincar\'e sheaf.

First we need some definitions. Consider the projection to the Hilbert scheme of its associated universal scheme $h_m : \Zz_{m} \to \Hilb^{N_m}(\overline{X})$, the coherent sheaf of algebras $\Aa_m := h_{m,*}\Oo_{\Zz_{m}}$ over $\Hilb^{N_m}(\overline{X})$ and denote by $\Aa_m^*$ the subsheaf of invertible elements. Consider $p_1$ to be the projection of $\Hilb^{N_m}(\overline{X}) \times \overline{\Jjac}^{\, \delta}(\overline{X})$ to the first factor and take the pull-back $p_1^{-1}\Aa_m^*$. Given a sheaf, we use the subindex ${p_1^{-1}(\Aa_m^*)}$ to denote the maximal quotient of the sheaf where $p_1^{-1}(\Aa_m^*)$ acts via the norm character.

Consider also the triple product $\overline{X} \times W_m \times \overline{\Jjac}^{\, \delta}(\overline{X})$ and denote by $g_{ij}$ the projections to the $i$-th and $j$-th factors.  Following \cite{arinkin2}, we define the sheaf over $W_m \times \overline{\Jjac}^{\, \delta}(\overline{X})$
\begin{equation} \label{def Poincare sheaf}
\overline{\PPp}_m := \left ( \bigwedge^{N_m} g_{23,*} ( g_{12}^*\Oo_{\Zz_m} \otimes g_{13}^*\UUu ) \right)_{p_1^{-1}(\Aa_m^*)} \otimes  \left ( \bigwedge^{N_m} g_{23,*} (g_{12}^* \Oo_{\Zz_m}) \right)^{-1}.
\end{equation}
The following is an inmediate adaptation of \cite{arinkin2}.

\begin{proposition} \label{pr Pp is CM}
The sheaves $\overline{\PPp}_m \to W_m \times \overline{\Jjac}^{\, \delta}(\overline{X})$ are Cohen--Macaulay and flat over $\overline{\Jjac}^{\, \delta}(\overline{X})$ for all positive integer $m$.
\end{proposition}

\begin{proof}
Up to a base change, the construction of \eqref{def Poincare sheaf} coincides with Arinkin's definition of the sheaf $Q'$ after making the substitution of the fine compactified Jacobian (of an integral curve) and its universal sheaf by the moduli stack of torsion free sheaves (on a nodal cuve) and its associated universal sheaf. After the same substitution, one can also adapt Arinkin's construction of another sheaf $Q$ which he shows to be isomorphic to $Q'$ in \cite[Proposition 4.5]{arinkin2}. The proof of \cite[Proposition 4.5]{arinkin2} relies entirely on a result \cite[Lemma 3.6]{arinkin2} concerning isospectral Hilbert schemes of surfaces, so \cite[Proposition 4.5]{arinkin2} extends to our case and both constructions coincide here as well. Using the construction of $\overline{\PPp}_m$ associated to $Q$ and \cite[Lemma 2.1 and Proposition 4.2]{arinkin2}, we have that $\overline{\PPp}_m$ is a Cohen--Macaulay sheaf, flat over $\overline{\Jjac}^{\, \delta}(\overline{X})$. Note that \cite[Lemma 2.1]{arinkin2} is a statement for Cohen--Macaulay sheaves in general and \cite[Proposition 4.2]{arinkin2} works for any reduced curve and any rank $1$ torsion free sheaf on it, so both are valid in our case. 
\end{proof}

This construction recovers the Poincar\'e bundle.

\begin{proposition} \label{pr restriction of Poincare is Poincare}
$\PPp$ and $\overline{\PPp}_m|_{\W_m \times \Jjac^{\, \delta}(\overline{X})}$ are isomorphic up to the twisting by a line bundle over $\Jjac^{\, \delta}(\overline{X})$. 
\end{proposition}

\begin{proof}
Since the $\Uu_m$ are defined as (the restriction to $W_m \times \Hilb^{N_m}(\overline{X})$ of) \eqref{eq definition of Uu_m}, in terms of the Abel--Jacobi atlas from Proposition \ref{pr abel-jacobi atlas}, $\PPp$ reads
\[
\PPp \cong \Dd_{g_{23}}(g_{12}^* \Ii^\vee_{\Zz_m} \otimes g_{12}^*q_m^*\Oo_{\overline{X}}(-m) \otimes g_{13}^*\UUu_0) \otimes \Dd_{g_{23}}(g_{13}^*\UUu_0)^{-1} \otimes \Dd_{g_{23}}(g_{12}^* \Ii^\vee_{\Zz_m} \otimes g_{12}^*q_m^*\Oo_{\overline{X}}(-m))^{-1}.
\]

We recall that $W_m$ is a subset of those subschemes $Z$ such that the first cohomology space of its ideal sheaf is trivial, $H^1(\overline{X},\Ii_Z) = 0$. It then follows that $R^1g_{23,*} (g_{12}^* \Oo_{\Zz_m})$ vanishes and $R^0g_{23,*} (g_{12}^* \Oo_{\Zz_m})$ is locally free of rank $N_m$. Under these conditions, the second term in the tensorization of the right-hand side of \eqref{def Poincare sheaf} equals the determinant in cohomology, 
\[
\bigwedge^{N_m} g_{23,*} (g_{12}^* \Oo_{\Zz_m}) \cong \det R^0 g_{23,*} (g_{12}^* \Oo_{\Zz_m}) \cong \Dd_{g_{23}}(g_{12}^* \Oo_{\Zz_m}).
\]
Also, $g_{13}^*\UUu$ is a line bundle over $W_m \times \Jjac^{\, \delta}(\overline{X})$. This implies, for large $m$, that $R^1g_{23,*} (g_{12}^* \Oo_{\Zz_m} \otimes g_{13}^*\UUu)$ vanishes and $R^0g_{23,*} (g_{12}^* \Oo_{\Zz_m} \otimes g_{13}^*\UUu)$ is locally free of rank $N_m$. Then, 
\[
\bigwedge^{N_m} g_{23,*} (g_{12}^* \Oo_{\Zz_m} \otimes g_{13}^*\UUu_0) \cong \det R^0 g_{23,*} (g_{12}^* \Oo_{\Zz_m} \otimes g_{13}^*\UUu) \cong \Dd_{g_{23}}(g_{12}^* \Oo_{\Zz_m} \otimes g_{13}^*\UUu_0)
\]
is a line bundle on which $p_1^{-1}(\Aa_m^*)$ acts via the norm character. Therefore, we have seen that
\[
\overline{\PPp}_m|_{\W_m \times \Jjac^{\, \delta}(\overline{X})} \cong \Dd_{g_{23}}(g_{12}^* \Oo_{\Zz_m} \otimes g_{13}^*\UUu_0) \otimes \Dd_{g_{23}}(g_{12}^* \Oo_{\Zz_m})^{-1}.
\]

From the short exact sequence 
\[
0 \to g_{12}^*\Oo_{\overline{X} \times \Hilb^{N_m}(\overline{X})} \to g_{12}^*\Ii^\vee_{\Zz_m} \to g_{12}^*\Oo_{\Zz_m} \to 0,
\]
and the additivity property of the determinant in cohomology, one can deduce
\[
\Dd_{g_{23}}(g_{12}^* \Oo_{\Zz_m} \otimes g_{13}^*\UUu_0) \cong \Dd_{g_{23}}(g_{12}^* \Ii^\vee_{\Zz_m} \otimes g_{13}^*\UUu_0) \otimes \Dd_{g_{23}}(g_{13}^*\UUu_0)^{-1}
\]
and 
\[
\Dd_{g_{23}}(g_{12}^* \Oo_{\Zz_m}) \cong \Dd_{g_{23}}(g_{12}^* \Ii^\vee_{\Zz_m}).
\]
Therefore, 
\[
\overline{\PPp}_m|_{\W_m \times \Jjac^{\, \delta}(\overline{X})} \cong \Dd_{g_{23}}(g_{12}^* \Ii^\vee_{\Zz_m} \otimes g_{13}^*\UUu_0) \otimes \Dd_{g_{23}}(g_{13}^*\UUu_0)^{-1} \otimes \Dd_{g_{23}}(g_{12}^* \Ii^\vee_{\Zz_m})^{-1}.
\]

Thanks to this description of $\overline{\PPp}_m|_{\W_m \times \Jjac^{\, \delta}(\overline{X})}$ and the description of $\PPp$ given at the beginning of the proof, the result follows from \cite[Claim after (4.18)]{melo2}.
\end{proof}

The following theorem was explained to us by T.Pantev, who proved it in collaboration with D. Arinkin. Since the proof is not published, we include one here.
\begin{theorem}[D. Arinkin and T. Pantev]
Let $\overline{X}$ be a connected reduced curve with at most nodal singularities. For $r$ large enough, the $\{ \overline{\PPp}_m \to W_m \times \overline{\Jjac}^{\, \delta}(\overline{X}) \}_{m = r}^\infty$ descend to a Cohen--Macaulay sheaf $\overline{\PPp}$ over $\overline{\Jjac}^{\, \delta}(\overline{X}) \times \overline{\Jjac}^{\, \delta}(\overline{X})$, that extends $\PPp$ up to a twist. 
\end{theorem}

\begin{proof}
Thanks to Proposition \ref{pr restriction of Poincare is Poincare} one has that the set of restrictions $\{ \overline{\PPp}_m|_{\W_m \times \Jjac^{\, \delta}(\overline{X})} \}_{m = r}^\infty$ descend to a bundle over the product of stacks $\overline{\Jjac}^{\, \delta}(\overline{X}) \times \Jjac^{\, \delta}(\overline{X})$. Let $W_m^\ell$ denote that subset of $W_m \subset \Hilb^{N_m}(\overline{X})$ given by those subschemes whose ideal sheaf is invertible. One can proceed analogously as we did in the proof of Proposition \ref{pr restriction of Poincare is Poincare} and show that the restriction $\{ \overline{\PPp}_m|_{W_m^\ell \times \Jjac^{\, \delta}(\overline{X})} \}_{m = r}^\infty$ descend to a bundle over the product of stacks $\Jjac^{\, \delta}(\overline{X}) \times \overline{\Jjac}^{\, \delta}(\overline{X})$. Therefore, the restriction of the $\overline{\PPp}_m$ to $\left ( W_m \times \overline{\Jjac}^{\, \delta} (\overline{X}) \right)^\sharp :=  \left ( W_m \times \Jjac^{\, \delta}(\overline{X})  \right ) \cup \left ( W_m^\ell \times \overline{\Jjac}^{\, \delta}(\overline{X}) \right )$ descend to a bundle over $\left ( \overline{\Jjac}^{\, \delta}(\overline{X}) \times \overline{\Jjac}^{\, \delta}(\overline{X}) \right )^\sharp$ that we denote $\overline{\PPp}^\sharp_m$.

We now recall that $i: \left ( \overline{\Jjac}^{\, \delta}(\overline{X}) \times \overline{\Jjac}^{\, \delta}(\overline{X}) \right )^\sharp \hookrightarrow \overline{\Jjac}^{\, \delta}(\overline{X}) \times \overline{\Jjac}^{\, \delta}(\overline{X})$ has codimension at least $2$. Thanks to Proposition \ref{pr Pp is CM}, we have that $\overline{\PPp}$ is Cohen --Macaulay. Then, it follows that 
\begin{equation} \label{eq olPp as the pf of olPp}
\overline{\PPp}_m \cong i^*\overline{\PPp}^\sharp_m                                                   \end{equation}
so the collection $\{ \overline{\PPp}_m^\sharp \}_{m = r}^\infty$ descend to a bundle on $\left ( \overline{\Jjac}^{\, \delta}(\overline{X}) \times \overline{\Jjac}^{\, \delta}(\overline{X}) \right )^\sharp$. Thanks to \eqref{eq olPp as the pf of olPp}, one has that
\[
\overline{\PPp} \cong i^*\overline{\PPp}^\sharp.
\]
Therefore, $\{\overline{\PPp}_m \}_{m = \ell}^\infty$ descend to a sheaf over $\overline{\Jjac}^{\, \delta}(\overline{X}) \times \overline{\Jjac}^{\, \delta}(\overline{X})$. The rest of the proof is straigth-forward.
\end{proof}

When our curve $\overline{X}$ is irreducible any rank $1$ torsion free sheaf is stable and simple. Therefore, the moduli stack of torsion free sheaves on a curve is the quotient stack associated to the fine compactified Jacobian $\overline{\Jac}^{\, \delta} (\overline{X})$ quotiented by the trivial action of $\CC^*$,
\[
\overline{\Jjac}^{\, \delta} (\overline{X}) \cong \left [ \quotient{\overline{\Jac}^{\, \delta}(\overline{X})}{\CC^*} \right ].
\]
Let us denote by $\overline{\Pp} \to \overline{\Jac}^{\, \delta}(\overline{X}) \times \overline{\Jac}^{\, \delta}(\overline{X})$ the pull-back of the Poincar\'e sheaf $\PPp$ under the obvious projection $\overline{\Jac}^{\, \delta}(\overline{X}) \to [\overline{\Jac}^{\, \delta}(\overline{X})/\CC^*]$, and one can consider the integral functor given by it, 
\begin{equation} \label{eq FM with olPp}
\morph{D^b \left ( \overline{\Jac}^{\, \delta}(\ol{X}) \right )}{D^b \left ( \overline{\Jac}^\delta(\ol{X}) \right )}{\Ee^\bullet}{R \pi_{2,*}(\pi_1^*\Ee^\bullet \otimes \overline{\Pp}).}{}{\overline{\Phi}}
\end{equation}
The Poincar\'e sheaf $\overline{\Pp}$ was first obtained by \cite{EGK} for compactified Jacobians of irreducible nodal curves. Arinkin \cite{arinkin2} extended this construction to any irreducible reduced planar curve, showing also that \eqref{eq FM with olPp} is a derived equivalence. Although his result does not extend to the context under consideration, we include it for the sake of completeness:

\begin{theorem}[\cite{arinkin2}] \label{tm arinkin}
Let $\overline{X}$ be an irreducible reduced planar curve, the Fourier--Mukai integrable functor $\overline{\Phi}$ provides an equivalence of categories. 
\end{theorem}

The integral functor associated to $\overline{\PPp}$ is an eigenfunctor of the derived category of sheaves over the moduli stack of torsion free rank $1$ sheaves over a reducible planar curve. It is being studied by Arinkin and Pantev \cite{pantev} whether this provides an equivalence or not.

\section{Totally reducible spectral curves}

\label{sc totally reducible}

\subsection{The locus of totally reducible spectral curves and the Borel subgroup}

We start by studying the Hitchin fibres associated to spectral curves that are totally reducible. 

Recall from Section \ref{sc Higgs moduli spaces} that, for any $b \in \H$, the associated spectral curve $\ol{X}_b$ is the $n:1$ cover of the base curve $X$ given by the vanishing of the section \eqref{eq equation spectral curve}. If $\ol{X}_b$ is totally reducible, then one can rewrite the section \eqref{eq equation spectral curve} as
\begin{equation} \label{eq equation totally reducible spectral curve}
\prod_{i = 1}^n (\lambda - \pi^*\alpha_i),
\end{equation}
where $\alpha_i \in H^0(X,K)$. In view of this, consider the symmetric product
\begin{equation} \label{eq definition of V_n}
\V := \Sym^n \left ( H^0(X,K) \right ).
\end{equation}
Hence
\begin{equation}\label{eq dim Vn}
\dim \V = n g.
\end{equation}
There is an injection into the Hitchin base
\begin{equation} \label{eq embedding of V into H}
\begin{array}{ccc} \V  & \hookrightarrow &  \H  
\\  
(\alpha_1, \dots, \alpha_n)_{\sym} & \longmapsto & ( q_1(\alpha_1, \dots, \alpha_n), \dots , q_n(\alpha_1, \dots, \alpha_n) ). \end{array}
\end{equation}
In the above: $(\alpha_1, \dots, \alpha_n)_{\sym}$ denotes the orbit of $(\alpha_1, \dots, \alpha_n)$ under the $n$-th symmetric group ${\sym}$, and $q_i(\alpha_1, \dots, \alpha_n)$ is the evaluation of $q_i$ on the diagonal Higgs field with entries $\alpha_i$. Note that the $q_i$ being invariant under the adjoint action, this depends only on the orbit $(\alpha_1, \dots, \alpha_n)_{\sym}$.

Seen inside the Hitchin base, $\V$ describes the locus of totally reducible spectral curves.

\begin{lemma}\label{lemma spectral curve}
$\V$ parametrizes all spectral curves that are totally reducible. Let $v \in \V$ be given by $v = (\alpha_1, \stackrel{m_1}{\dots}, \alpha_1, \dots, \alpha_{\ell},\stackrel{m_{\ell}}{\dots}, \alpha_{\ell})_{\sym}$, where $\sum_{i=1}^{\ell} m_i = n$ and $\alpha_i \neq \alpha_j$ if $i \neq j$. Then, its corresponding spectral curve is
\begin{equation} \label{eq description of X_v}
\ol{X}_v = \bigcup_{i = 1}^{\ell} X_i^{m_i},
\end{equation}
where each $X_i^{m_i}$ is a curve of multiplicity $m_i$ whose reduced subscheme is $X_i :=\alpha_i(X)$, isomorphic to $X$. 
\end{lemma}

\begin{proof}
This follows easily from \eqref{eq equation totally reducible spectral curve}.
\end{proof}

Fix a Borel subgroup $\B < \GL_n(\CC)$ containing $\C$, so that $\B=\C\ltimes \U$ where $\U=[\B,\B]$ is the unipotent radical of $\B$. Let us consider the subvariety given by those Higgs bundles whose structure group reduces to $\B$,
\[
\Bor : = \left\{(E,\varphi)\in\M_n\ \left| \begin{array}{l}
\exists \, \sigma \in H^0(X,E/\B),\\
\varphi\in H^0(X,E_\sigma(\mathfrak{b})\otimes K).
\end{array}\right. \right\},
\]
where $E_\sigma :=\sigma^*E$ is the principal $\B$-bundle on $X$ associated to the section $\sigma \in H^0(X,E/\B)$.

We can see that $\Bor$ coincides with the preimage under the Hitchin map of the locus of totally reduced spectral curves.

\begin{proposition}\label{pr Bor 1}
One has the following,
\begin{equation}\label{eq Bor}
\M_n\times_{\H} \V = \Bor.
\end{equation}
\end{proposition}

\begin{proof}
We first see that $\Bor \subset \M_n \times_{\H} \V$. This is a consequence of the following fact: given the Jordan--Chevalley decomposition of $x=x_s+x_n\in\mathfrak{gl}_n(\CC)$ into a semisimple $x_s$ and a nilpotent piece $x_n$, the invariant polynomials $q_i$ defining the Hitchin fibration evaluate  independently of the nilpotent part, namely $q_i(x)=q_i(x_s)$.

For the other inclusion one has to prove that any Higgs bundle $(E, \varphi) \in \M_n\times_{\H}\V$ admits a full flag decomposition.

Denote by $\Ff$ the torsion-free sheaf over the spectral curve $\ol{X}_v$ associated to $(E,\varphi)$ under the spectral correspondence. Recall that $\ol{X}_v$ is described in \eqref{eq description of X_v} and, using this notation, define
 \begin{equation}\label{eq Yi Zi}
Y_i := \bigcup_{j = 1}^i X_j^{m_j},\qquad
Z_i := \bigcup_{k = i+1}^\ell X_k^{m_k}.
\end{equation}
We consider the restriction of $\Ff$ to $\Ff|_{Z_i}$ and denote its kernel by $\Ff_i$, 
\begin{equation}\label{eq Li}
0 \longrightarrow \Ff_i \longrightarrow \Ff \longrightarrow \Ff|_{Z_i} \longrightarrow 0.
\end{equation}
Since $\Ff_i$ is a subsheaf of $\Ff$, it gives the Higgs subbundle $(E_i, \varphi_i) \subset (E, \varphi)$ under the spectral correspondence. Since $\Ff_{i-1}$ is a subsheaf of $\Ff_i$ we have that $(E_{i-1}, \varphi_{i-1}) \subset (E_i, \varphi_i)$ so we obtain a filtration
\begin{equation} \label{eq filtration from ordering}
0 \subset (E_1, \varphi_1) \subset \dots \subset (E_\ell, \varphi_\ell) = (E,\varphi).
\end{equation}

Note that a full flag filtration for each of the $(F_i, \phi_i):=(E_i, \varphi_i)/(E_{i-1}, \varphi_{i-1})$ will induce a full flag filtration of $(E,\varphi)$.

Note that the eigenvalues of $\phi_i$ are all equal to $\alpha_i$. Set $F_{i,1} = \ker(\phi_i - \alpha_i \otimes \id_{F_i})$ and let $\phi_{i,1}$ be the restriction to $F_{i,1}$. Set $(F'_i, \phi'_i) = (F_i,\phi_i)/(F_{i,1}, \phi_{i,1})$ and take $F'_{i,2} = \ker(\phi'_i - \alpha_i \otimes \id_{F'_i})$ and $\phi'_{i,2} = \phi'_{i}|_{F'_{i,2}}$. Note that $(F'_{i,2},\phi'_{i,2}) \subset (F'_i, \phi'_i)$ lifts to a subbundle $(F_{i,2}, \phi_{i,2})$ of $(F_i, \phi_i)$ which contains $(F_{i,1}, \phi_{i,1})$. Repeating this procedure one gets a filtration
\[
0 \subset (F_{i,1}, \phi_{i,1}) \subset \dots \subset (F_{i,s}, \phi_{i,s}) = (F_{i}, \phi_{i}),
\]
where each quotient $(F_{i,j}, \phi_{i,j})/(F_{i,j-1}, \phi_{i,j-1})$ is isomorphic to a Higgs bundle of the form $(G_{i,j}, \alpha \otimes \id_{G_{i,j}})$.

Given an ample line bundle $\Oo_X(1)$, one has that, for sufficiently high $N>0$, that $\Oo_X(-N)$ is a subbundle of $G_{i,j}$, and the same is valid for the quotient $G_{i,j}/\Oo_X(-N)$. Hence, one can always construct a full flag filtration for each of the $G_{i,j}$. This provides a full flag filtration for all the $(F_i, \phi_i)$, hence a full flag filtration for $(E,\varphi)$. 
\end{proof}

\begin{remark} \label{rm Bor stack}
Note that Proposition \ref{pr Bor 1} generalizes to the corresponding moduli stacks as stability plays no role on its proof. 
\end{remark}

\begin{remark} \label{rm reduction = full flag filtration not canonical}
The full flag filtration of the Higgs bundle $(E,\varphi)$ determines the reduction to the Borel subgroup $\sigma \in H^0(X,E/\B$ with $\varphi \in H^0(X,E_\sigma(\lie{b}\otimes K))$. Note that, in general, one can not give a canonical such a full-flag filtration. 
\end{remark}

In the remaining of the section we will focus on an open subset of $\V$. Denote the big diagonal of $\V$ by
$$
\Delta := \{ (\alpha_1, \dots, \alpha_n)_{\sym} \in \V \textnormal{ such that } \alpha_i = \alpha_j \textnormal{ for some } i,j  \}
$$
and its complement in $\V$ by
$$
\V^\red := \V \setminus \Delta.
$$
Let us provide a description of the spectral curves parametrized by $\V^\red$.

\begin{lemma}\label{lemma spectral curve red}
$\V^\red$ is a dense open subset of $\V$ parametrizing reduced, totally reducible, and nodal spectral curves. Furthermore, for any $v \in \V^\red$ given by $(\alpha_1, \dots, \alpha_n)_{\sym_n}$, the spectral curve $\ol{X}_v$ is reduced and has the following decompositon into irreducible components,
\begin{equation} \label{eq nodal spectral curve}
\ol{X}_v = \bigcup_{i = 1}^{n} X_i,
\end{equation}
with $X_i = \alpha_i(X) \cong X$. It is a singular curve with singularity divisor of length $|D| = (n^2 - n)(g-1) = \delta$. Its normalization, $\ol{X}_v$, is isomorphic to
\begin{equation} \label{eq description of normalization}
\wt{X}_v \cong \bigsqcup_{i=1}^n X_i \cong \bigsqcup_{i=1}^n X,
\end{equation}
and the normalization morphism,
\begin{equation}\label{eq normalization}
\nu: \wt{X}_v \to\ol{X}_v,
\end{equation}
is the identity restricted to each of the $X_i$.
\end{lemma}

\begin{proof}
$\Delta$ is a closed subset of $\V$ of codimension $1$, hence $\V^\red$ is open and dense. When $v \in \V \setminus \Delta$, \eqref{eq description of X_v} implies that $\ol{X}_v$ is the union of $n$ different reduced and irreducible curves $X_i$ all isomorphic to $X$. It then follows that $\ol{X}_v$ is reduced and its normalization is as described in \eqref{eq description of normalization}. The description of the normalization mmorphism follows form the description of the spectral curve given in \eqref{eq nodal spectral curve}. The length of $D$ can be obtained after an easy computation using Riemann-Roch.
\end{proof}

For any two $\alpha_i$ and $\alpha_j$ with $i \neq j$, denote the divisor $D_{ij} = \alpha_i(X) \cap \alpha_j(X)$. Consider also the following subset of $\V^\red$,
$$
\V^\nod := \left\{ 
	   \begin{array}{l}
		 \textnormal{$(\alpha_1, \dots, \alpha_n)_{\sym} \in \V^\red$ such that for every $i<j<k$} 
		 \\
		 \hspace{1cm} \textnormal{(a) there is no multiple point on $D_{ij}$, and}
		 \\
		 \hspace{1cm} \textnormal{(b) $D_{ij} \cap D_{ik}$ is empty.}
	   \end{array}
	     \right\}.
$$

\begin{lemma} \label{lm spectral curve nod}
$\V^\nod$ is a dense open subset of $\V$ parametrizing reduced, totally reducible, and nodal spectral curves. For any $v \in \V^\nod$ given by $(\alpha_1, \dots, \alpha_n)_{\sym_n}$, the singularity divisor $D$ of the spectral curve $\ol{X}_v$ is
$$
D := \bigcup_{i,j} D_{ij}
$$
and consists only of simple points.
\end{lemma}

\begin{proof}
Since conditions (a) and (b) are open and generic, $\V^\nod$ is a dense open subset of $\V^\red$. It then follows from Lemma \ref{lemma spectral curve red} that $\V^\nod$ is dense within $\V$ too and the first statement follows.

Recall the description of $\ol{X}_v$ given in Lemma \ref{lemma spectral curve red}. Take two irreducible components of $\ol{X}_v$, $X_i$ and $X_j$, intersecting each other at $D_{ij}$. Note that $D$ coincides with the set of intersection points and recall that we have imposed the condition $D_{ij} \cap D_{ik} = \emptyset$ if $j\neq k$ in the definition of $\V^\nod$, so $D$ is the union of the $D_{ij}$.
\end{proof}

Using the notation of Lemma \ref{lemma spectral curve red}, consider the following morphisms,
\begin{equation}\label{eq not curves}
\xymatrix{
\wt{X}_v\ar[dddd]_p\ar[ddr]_{\nu} & & &
\\
 & & &
\\
& \spec_v\ar[ddl]_\pi & & X_j  \ar@{_(->}[ll]_{\iota_j}  \ar@{_(->}[uulll]_{\delta_j}
\\
 & & &
\\
X.\ar[urrur]_{\alpha_j}^{\cong} & & & 
}
\end{equation}

We have seen in Remark \ref{rm reduction = full flag filtration not canonical} that the reduction to the Borel subgroup can not be defined canonically for an arbitrary Higgs bundle in $\Bor$. However, for those Higgs bundles lying over $v \in \V^\nod$, one can fix such a reduction after choosing an ordering for the components of $v$.

\begin{proposition} \label{pr description line bundles on V^nod}
Let $v=(\alpha_1,\dots,\alpha_n)_{\sym_n}\in \V^{\nod}$ and let $(E,\varphi)\in h^{-1}(v)$. For any ordering $J=(\alpha_{j_1},\dots, \alpha_{j_n})$ of the set $\{\alpha_1,\dots,\alpha_n\}$, one can chose canonically a filtration
$$
(E_J)_\bullet \, : \, 0 \subsetneq (E_1,\varphi_1) \subsetneq \dots \subsetneq (E_n,\varphi_n) = (E,\varphi),
$$
such that the Higgs field induced by $\varphi$ on $E_i/E_{i-1}$ is $\alpha_{j_i}$. Furthermore, if the associated spectral datum associated to $(E,\varphi)$ is a line bundle over the spectral curve, $L \in \Jac^{\delta}(\ol{X}_v)$, then% lies in $\Jac^\delta(\ol{X}_v)$
$$
E_i/E_{i-1} \cong (\alpha_{j_i}^*\iota_{j_i}^*L)\otimes K^{i-n}
%(E_{j_i},\varphi_{j_i})/(E_{j_{i-1}},\varphi_{j_{i-1}}) = (\alpha_{j_i}^*\iota_{j_i}^*L\otimes K^{i-n},\alpha_{j_i}),
$$
%where the notation is as in \eqref{eq not curves}. 
\end{proposition}

\begin{proof}
Using the ordering $J$ set $Y_{i}=\bigcup_{k=1}^iX_{j_k}$, $Z_{i}=\bigcup_{k=i+1}^nX_{j_k}$ as in \eqref{eq Yi Zi}. After the choice of $J$, the filtration for the spectral data given in \eqref{eq Li} is canonical and so is the filtration \eqref{eq filtration from ordering} of $(E,\varphi)$. Since $v \in \V^\nod$, \eqref{eq filtration from ordering} is a full flag filtration what proves the first statement.

For the second statement recall that the filtration of $L$ is defined by the subsheaves $L_{i}=L\otimes\mc{I}_{\spec,Z_i}$ where $\mc{I}_{\spec,Z_i}$ denotes the ideal defining the subscheme $Z_i\subset \spec$. Now, $\mc{I}_{\spec,Z_i}\cong \mc{O}_{Y_i}\otimes \mc{I}_{Y_i,Z_i\cap Y_i}$, thus
$$
L_i\cong L|_{Y_i}\otimes \mc{I}_{Y_i,Z_i\cap Y_i}.
$$
Note that
$$
0\longrightarrow \quotient{L_i}{L_{i-1}}\longrightarrow L|_{Z_{i-1}}\longrightarrow L|_{Z_i}\longrightarrow 0
$$
is exact, so that 
\begin{align*}
L_i/L_{i-1} & \cong L|_{Z_i}\otimes\mc{I}_{Z_{i-1},Z_i}
\\
& \cong L|_{Z_i}\otimes \mc{O}_{X_i}\otimes\mc{I}_{X_i,Z_i\cap X_i}
\\
& \cong L|_{X_i}(-\sum_{k=i+1}^n D_{ik}).    
\end{align*}
Now, the pushforward 
of 
$$
0\longrightarrow {L_{i-1}}\longrightarrow{L_i}\longrightarrow \quotient{L_i}{L_{i-1}}\longrightarrow 0
$$
gives under the spectral correspondence
$$
\quotient{(E_i, \varphi_i)}{(E_{i-1}, \varphi_{i-1})} \cong \left ( \alpha_{j_i}^*\iota_{j_i}^*L(-\sum_{k=i+1}^nD_{ik}), \alpha_{j_i} \right ),
$$
where we abuse notation by identifying the divisor $D_{jk}$ and its image under $\pi$. Naturally, $K \cong \Oo_X(D_{jk})$, which yields the result.
\end{proof}

\subsection{Totally reducible nodal spectral curves and their desingularization}

We study in this section the relation between the Hitchin fibres associated to totally reduced spectral curves with only nodal singularities and their partial and complete desingularizations.

We first recall some well known facts about rank one torsion free sheaves on a reduced connected nodal curve $\ol{X}$ with  divisor of singularities $D$. We start by studying the particular case of line bundles which admit a simple description in terms of their pullback to partial (and complete) desingularization. Consider $R\subset D$ a subdivisor of the singular divisor of the reduced curve $\ol{X}$, and let 
\begin{equation} \label{eq definition of nu_R}
\nu_R:\wt{X}_R\to \overline{X} 
\end{equation}
be the partial desingularization at $R$. Note that $\nu_D:\wt{X}^D\to \spec$ is just the normalization map $\nu$ that appeared in \eqref{eq normalization}. Denote by 
$$
\morph{\Jac(\ol{X})}{\Jac(\wt{X}_R)}{L}{\nu_R^*L}{}{\hat{\nu}_R}
$$
the pullback map. The fibres of this map are described in the following lemma due to Grothendieck \cite[Proposition 21.8.5]{EGA}, that we reproduce adapted to our notation.

\begin{lemma}[\cite{EGA}]\label{lm fibration picards}
For any subdivisor $R\subset D$ of the singular divisor of the reduced nodal curve $\ol{X}$, the pullback map $\hat{\nu}_R$ is a smooth fibration with fiber $(\CC^\times)^{|R|-n_R+1}$ where $n_R$ is the number of connected components of $\wt{X}_R$. 
\end{lemma}

One can give the following geometrical interpretation of Lemma \ref{lm fibration picards}: line bundles on reduced nodal curves can be described in terms of line bundles on each of the $n_R$ irreducible components of the (partial) desingularization, together with $|R|$ gluing data ({\it i.e.} an element of $\CC^\times$ identifying the two local components of the nodal point) for each of the intersection points, taking into account the identification given by scalar automorphisms on each of the components.

In the case of $R = D$, we have that $\wt{X}_R = \wt{X}_v$ is the normalization of the spectral curve and $\nu_R$ coincides with the normalization map $\nu$. One has the following description adapted to that case.

\begin{corollary} \label{co hat nu}
The pullback map 
\begin{equation} \label{eq hatnu}
\morph{\Jac(\ol{X}_v)}{\Jac(\wt{X}_v)}{L}{\nu^*L}{}{\hat{\nu}}
\end{equation}
is a smooth fibration with fiber $(\CC^\times)^{n(n-1)(g-1)-(n-1)}$. 
\end{corollary}

It can be checked that that the degree of a line bundle $L$ on a connected nodal curve $\ol{X}$ with irreducible components $X_i$ is given by the sum of the degrees of the line bundles obtained by restricting to each of the components, $\deg L=\sum_i\deg L|_{X_i}$. In view of this we refer to the {\it multidegree} of a line bundle $L$ on $\ol{X}$ as the the degree on each of the connected components of $\wt{X}$. In other words, the multidegree of $\hat{\nu}(L) =\nu^*L$ over the disconnected curve $\wt{X}$.

A rank one torsion free sheaf on $\spec$ is either a line bundle or a pushforward of a line bundle on a partial desingularization $\nu_R$ of $\ol{X}$ (see \cite{Seshadri} for instance). Consider $L \in \Jac(\wt{X}_R)$ be given by the line bundles $L_i$ on each connected component $\wt{X}_{R,i}$ of $\wt{X}_R$. Geometrically, the (rank one torsion free coherent) sheaf $\nu_{R,*}L$ on $\ol{X}$ is obtained by considering $n_R$-tuples of $L_i \to \wt{X}_{R,i}$, together with identifications at all points $x\in D\setminus R$. One can also check that 
\begin{equation} \label{eq deg under push-forward}
\deg (\nu_{R,*}L) = \deg(L)+|R|.
\end{equation}

%\subsubsection{Hitchin fibers for nodal curves}\label{sect Hitchin fibers nodal}

We now study in more detail the spectral curves parametrized by $\V^\nod$ and their corresponding Hitchin fibres. Let us first fix some notation. Recall that, for $v\in\V^{\nod}$ given by $(\alpha_1, \dots, \alpha_n)_{\sym_n}$ we denote the associated spectral curve by $\ol{X}_v$. After Lemmas \ref{lemma spectral curve red} and \ref{lm spectral curve nod}, $\ol{X}_v = \bigcup_{i = 1}^n X_i$, where $X_i = \alpha_i(X) \cong X$ and be the divisor of singularities $D$ has length $\delta$ and it is given by the union of the two-by-two intersection of the smooth irreducible components. For any subdivisor $R\subset D$ consider the partial desingularization along $R$,
\begin{equation}\label{eq big eq partial normaliz}
\xymatrix{\wt{X}_R \ar[r]^{\nu_R}\ar[dr]_{p_R}&\spec_v\ar[d]^\pi\\
&X.}
\end{equation}

Consider the decomposition $\wt{X}_R=\bigsqcup_{i=1}^{n_R}\wt{X}_{R,i}$ into connected components and denote $\ol{X}_{R,i} = \nu_R(\wt{X}_{R,i})$. Therefore, one has the decomposition $R=R_1\sqcup\dots\sqcup R_{n_R}\sqcup R_s$ such that
$$
\nu_{R,i}:\wt{X}_{R,i}\longrightarrow\spec_{R,i}
$$
is a partial desingularization of
$\spec_{R,i}$ along a non-separating divisor $R_i$, and $R_s$ is the separating divisor in $R$ ({\it i.e.} the divisor along which connected components are to appear after desingularization). Denote by $p_{R,i}$ the restriction of $p_R$ to the corresponding connected component. For each irreducible component $X_j = \alpha_j(X) \cong X$ of $\ol{X}_v$, and its corresponding connected component $\wt{X}_j \cong X_j \cong X$ of the normalization $\wt{X}_v$, consider the commuting diagram 
\begin{equation}\label{eq not curves R}
\xymatrix{
\wt{X}_{R,i}\ar[dddd]_{p_{R,i}}\ar[ddr]^{\nu_{R,i}} & & & \wt{X}_j\ar@{_(->}[lll]_{\widetilde{\iota}_j}\ar[dd]^{\nu_{R,i}^j}_{\cong}
\\
 & & &
\\
& \spec_{R,i}\ar[ddl]_\pi & & X_j \ar@{_(->}[ll]_{\iota_j}
\\
 & & &
\\
X.\ar[uurrr]_{\alpha_j}^{\cong} & & &
}
\end{equation}
We then see that $\wt{X}_j \cong X$ are the irreducible components of $\wt{X}_{R,i}$ and denote by $C_i$ the index set of these components, hence
$\wt{X}_{R,i}=\bigcup_{j\in C_i}{\wt{X}}_j$ has $|C_i|$ irreducible components. Write $\tilde D_i\subset\wt{X}_{R,i}$ for the singular divisor of $\wt{X}_{R,i}$ and observe that it coincides with the ramification divisor of $p_{R,i}:\wt{X}_{R,i}\longrightarrow X$. Observe as well that \begin{equation}\label{eq Di}
D_i:=\nu_{R,i}(\tilde D_i)=\sum_{j,k\in C_i}D_{jk}-R_i\subset D
\end{equation}
and
\begin{equation}\label{eq D from Di and R}
D=\sum_i(D_i+R_i)+R_s.
\end{equation}
\begin{figure}[H]\label{dibujo}
\includegraphics[scale=0.2]{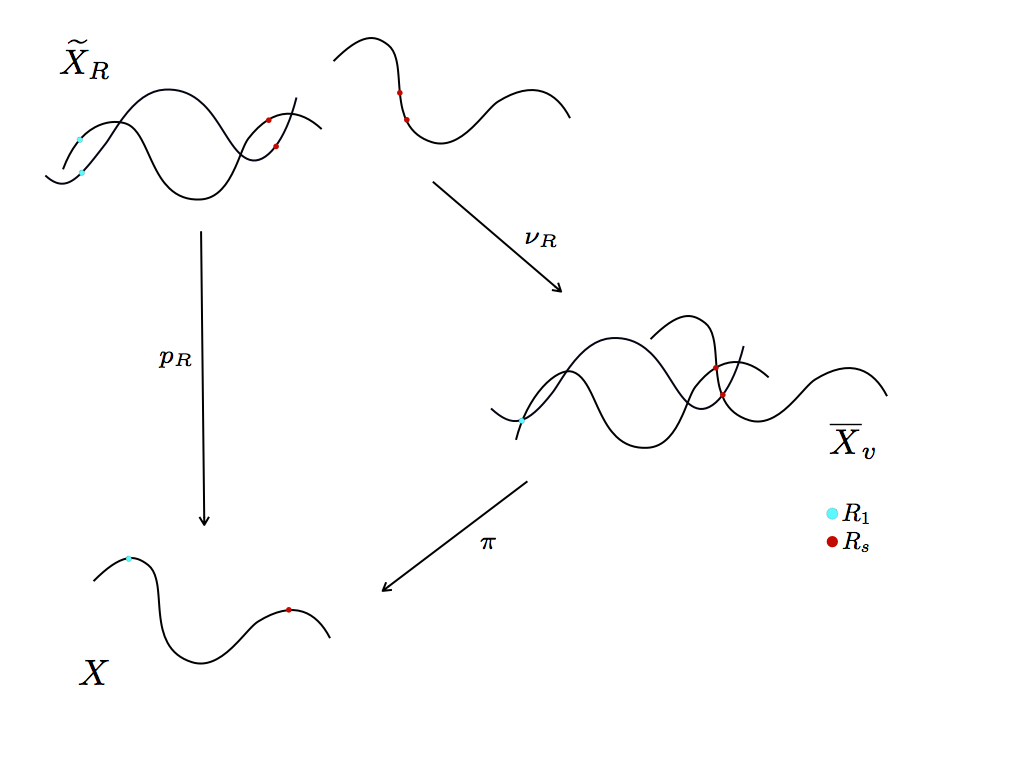}
\caption{Partial desingularization along $R$}
\end{figure}

We provide in the following lines a description of the Jacobians over $\wt{X}_R$. Choose an ordering $(\wt{X}_{R,1}, \dots, \wt{X}_{R,{n_R}})$ of the connected components of $\wt{X}_R$ and, with respect to it, denote
$$
\Jac^{\, \overline{\eta}}(\wt{X}_R) \cong \Jac^{\eta_1}(\wt{X}_{R,{1}})\times\dots\times\Jac^{\eta_{n_R}}(\wt{X}_{R,{n_R}})
$$
for each multidegree $\overline{\eta}$, and set $|\overline{\eta}| = \sum_{i=1}^{n_R}\eta_i$. Consider the decomposition  

\begin{equation}\label{eq picspecnR}
\Jac^{\, \eta}(\wt{X}_R) \cong\bigcup_{|\overline{\eta}|=\eta}  \Jac^{\, \overline{\eta}}(\wt{X}_R).
\end{equation}
Let also 
$$
 \Jac^{\eta_i}(\wt{X}_{R,{i}})=\bigcup_{\sum d_i^j=\eta_i}  \Jac^{(d_i^1,\dots, d_i^{|C_i|})}(\wt{X}_{R,{i}}),
$$
be the decomposition in terms of the multidegree associated to the irreducible components.

%%%%%%%%%%%%%%%%%%%

With the notation being settled, we now study push-forwards of line bundles under $\nu_R$. Recall that every rank one torsion-free sheaf on $\ol{X}_v$ is either of this form or a line bundle.

\begin{lemma}\label{lm spec Bor}
Let $v \in \V^\nod$. Only if
\begin{equation} \label{eq condition on eta}
\eta_i=\sum_{k=1}^{|C_i|} d_i^k=|D_i|,
\end{equation}
one has that the push-forward map
\begin{equation} \label{eq def check nu_R}
\morph{\Jac^{(d_1^1,\dots, d_1^{|C_1|})}(\wt{X}_{R,{1}})\times\dots\times \Jac^{(d_{n_R}^1,\dots, d_{n_R}^{|C_{n_R}|})}(\wt{X}_{R,n_R})}{\ol{\Jac}^{\, \delta}(\ol{X}_v)}{L}{\nu_{R,*}L,}{}{\check{\nu}_{R}}
\end{equation}
is well defined and an injection. Furthermore, when $R_s\neq\emptyset$, the Higgs bundles whose corresponding spectral data is in the image of $\check{\nu}_{R}$ are strictly polystable.
\end{lemma}

\begin{proof}
Assume first that $R_s=\emptyset$ hence $n_R = 1$ so $\wt{X}_R$ is connected. In that case $\nu_{R,*}L$ is stable. Otherwise, as any destabilizing subsheaf of $\nu_{R,*}L$ will come from a destabilizing subsheaf of $L$ and this would imply that $L$ is unstable. But $L$ is a line bundle so it is forcely stable. One also has that $\nu_{R,*}L \ncong \nu_{R,*}L'$ if $L \ncong L$ so it only remains to prove that the degree $\nu_{R,*}(L)$ is $\delta=|D|$. Note that this follows from \eqref{eq deg under push-forward} and \eqref{eq D from Di and R}, since \eqref{eq condition on eta} is equivalent to $\eta = |D_1|$ as $\wt{X}_R$ is connected.

Now, we study the case where $R_s\neq\emptyset$, so $\wt{X}_R$ has $n_R > 1$ connected components. Denote $\wt{\iota}_k^*L=L_k$, where the notation is as in \eqref{eq not curves R}. Note that 
\[
\pi_* \nu_{R,*}L=p_{R,*}L=\bigoplus_{i=1}^{n_R} p_{R,i,*} L_i,
\]
where the notation is as in \eqref{eq not curves R}. Note that the direct sum is invariant by the Higgs field, since the Higgs field is equivalent to a $\pi_*\Oo_{\spec_v}$ module structure on $\pi_* \nu_{R,*}L$, and the latter factors through a $\pi_* \nu_{R,*}\Oo_{\wt{X}_{R}}$-module structure. This proves that the Higgs bundle associated to $L$ is decomposable. Note that, as before, $\nu_{R, i,*}L_i$ is stable as $L_i$ is a line bundle, hence stable. Therefore, it must happen that 
\begin{equation} \label{eq condition for polystability}
\deg p_{R,i,*}L_i= \deg \pi_{i,*} \nu_{R,i,*}L_i =0
\end{equation}
for the Higgs bundle to be polystable. Note that we have used $p_{R,i}=\pi_i\circ\nu_{R,i}$. 

Given that $\spec_{R,i}$ is a totally reducible nodal spectral curve with $|C_i|$ irreducible components, arguing as in Lemma \ref{lemma spectral curve} (compare with \eqref{eq delta}) we find that \eqref{eq condition for polystability} is equivalent to
\[
\deg \nu_{R,i,*}L_i=(|C_i|^2-|C_i|)(g-1)=\left | \sum_{j,k\in C_i}D_{jk} \right |.
\]
Now, considering
$$
0\longrightarrow \nu_{R,i}^*\Oo_{\spec_{R,i}}\longrightarrow\Oo_{\wt{X}_{R,i}}\longrightarrow \Oo_{R_i}\longrightarrow 0,
$$
we have that
$$
\left |\sum_{j,k\in C_i}D_{jk} \right |=\deg \nu_{R,i,*}L_i=\deg L_i+|R_i|,
$$
which together with \eqref{eq Di} implies that \eqref{eq condition for polystability} is equivalent to \eqref{eq condition on eta}. In that case \eqref{eq def check nu_R} is well defined and it is injective since, as before, we have that $\nu_{R,i,*}L_i \ncong \nu_{R,i,*}L'_i$ whenever $L_i$ and $L'_i$ are not isomorphic.
\end{proof}

As a corollary of Lemma \ref{lm spec Bor}, one can derive the following well known fact when $R = D$. Hence, after Lemma \ref{lemma spectral curve red} the normalization $\wt{X}_v = \wt{X}_R$ of $\ol{X}$ decomposes into $n$ connected componets, each of them isomorphic to the base curve $X$. 

\begin{corollary} \label{co description check nu}
The push-forward map 
\begin{equation}\label{eq check nu}
\morph{\Jac^{(0,\dots,0)}(\wt{X}_v)}{\overline{\Jac}^\delta(\ol{X}_v)}{L}{\nu_*L,}{}{\check{\nu}}
\end{equation}
is well defined and an injection. Furthermore $\check{\nu} \left ( \Jac^{(0,\dots,0)}(\wt{X}_v) \right )$ classifies those strictly polystable Higgs bundles that decompose into direct sum of line Higgs bundles.
\end{corollary}

In Proposition \ref{pr description line bundles on V^nod} we provided a description of the dense open subset of the Hitchin fibre over $v \in \V^\nod$ corresponding to line bundles. Recalling that every torsion-free sheaf is given by the push-forward of a line bundle under a partial normalization $\nu_R$, we complete in the following lines the description initiated in Proposition \ref{pr description line bundles on V^nod} of Higgs bundles lying over $\V^\nod$.

%%%%%%%%%%%%%% 
\begin{proposition}\label{pr Bor 2}
Take any $v \in \V^\nod$ given by $v=(\alpha_1,\dots,\alpha_n)_{\sym_n}$ and suppose that the multidegree $\overline{d}$ satisfies \eqref{eq condition on eta}. One has the following,
\begin{enumerate}
\item \label{it iii} Assume $R_s\neq\emptyset$. Then, the Higgs bundles corresponding to spectral data in $\check{\nu}_R \left ( \Jac^{\ol{d}}(\wt{X}_R) \right )$ admit a reduction of their structure group to $\B_1\times\dots\times \B_{n_R}\subset \B$ 
where $\B_i$ is the Borel subgroup of $\GL(|C_i|,\CC)$.

\item \label{it v} Consider the Higgs bundle $(E,\varphi)=\bigoplus_{k=1}^{n_R}(E_k,\varphi_k)$ in $h^{-1}(v)\cap \check{\nu}_R \left ( \Jac^{\ol{d}}(\wt{X}_R) \right )$. Suppose that the spectral data of $(E,\Phi)$ is $\nu_{R,*}L$ where $L$ is a line bundle over $\wt{X}_R$. Then, for any ordering $J_k=(\alpha_{j_1},\dots, \alpha_{j_{|C_k|}})$ of $C_k$, one can chose canonically a filtration for $(E_k,\varphi_k)$, for all $k \in \{ 1, \dots, n_R \}$,
$$
(E_{{J_k}})_\bullet \, : \, 0 \subsetneq (E_{k,1}, \varphi_{1}) \subsetneq \dots \subsetneq (E_{k,|C_k|}, \varphi_{k,|C_k|} ) =(E_k,\varphi_k)
$$
such that
$$
(E_{k,i},\varphi_{k,i})/(E_{k,i-1},\varphi_{k,i-1})=\left (L|_{\wt{X}_{j_i}}\otimes \Oo \left( -\sum_{i'\geq i+1}{\wt{X}_{j_i}}\cap {\wt{X}_{j_{i'}}} \right ),\alpha_{j_i}\right )
$$ 
where we abuse notation by identifying the subdivisors ${\wt{X}_{j_i}}\cap {\wt{X}_{j_{i'}}}\subset D_{i'}$  \eqref{eq Di} and their images under $p_{i'}$, and $L|_{\wt{X}_{j_i}}$ with its pullback under $\alpha_{j_i}\circ(\nu_{R,k}^{j_i})^{-1}$.
\end{enumerate}
\end{proposition}

\begin{proof}
\eqref{it iii} Follows from Proposition \ref{pr Bor 1} and Lemma \ref{lm spec Bor}.

\eqref{it v} To simplify notation, take the orderings $((\alpha_1,\dots,\alpha_{|C_1|}), \dots (\alpha_{|C_{n_R-1}|},\dots,\alpha_n))$. The reasoning that follows adapts just the same way to any other choice of orderings. The statement is proven as Proposition \ref{pr description line bundles on V^nod}, taking the following remarks into account:

First note that the subscheme $Z_i\subset \ol{X}_v$ appearing in the proof of Proposition \ref{pr description line bundles on V^nod} is the image of its partial desingularization $\wt{Z}_i\subset\wt{X}_R$, on which the filtration will be given on each of the connected components. This restricts the proof to line bundles over connected curves $\wt{X}_R$.

By the previous remark we may assume $\wt{X}_R$ is connected and $J$ is an ordering for $\{\alpha_1,\dots, \alpha_n\}$. We obtain a full flag in the same way as in the proof of Proposition \ref{pr description line bundles on V^nod}, the difference with this case being that the ideal
$$
\mc{I}_{\wt{Z}_{i-1},\wt{Z}_{{i}}}\cong \Oo_{\wt{X}_{i}}(-\wt{X}_{i}\cap \wt{Z}_{i})
$$
depends on the ordering (and $R$) and so does 
$$
\wt{X}_{i}\cap \wt{Z}_{{i}}=\sum_{i'\geq i+1}\wt{X}_{i}\cap \wt{X}_{i'}.
$$
\end{proof}

\section{A $\BBB$-brane from the Cartan subgroup}\label{sect Cartan}

In this section we construct a $\BBB$-brane of $\M_n$, which is, by definition (cf.~\cite{kapustin&witten}), a pair $(\mathrm{N},(\Fff, \nabla_\Fff))$ given by:
\begin{itemize}
    \item A hyperholomorphic subvariety $\mathrm{N}\subset\M_n$, {\it i.e.} a subvariety which is holomorphic with respect to the three complex structures $\Gamma_1$, $\Gamma_2$ and $\Gamma_3$.
    
    \item A hyperholomorphic sheaf $(\Fff, \nabla_\Fff)$ supported on $\N$, {\it i.e.} a sheaf $\Fff$ equipped with a connection whose curvature $\nabla_\Fff$ is of type $(1,1)$ in the complex structures $\Gamma_1$, $\Gamma_2$ and $\Gamma_3$.
\end{itemize}

\begin{remark}
A flat connection is trivially of type $(1,1)$ in any complex structure. 
\end{remark}

%\subsection{A hyperholomorphic bundle on the Cartan locus}
%\label{sect construction BBB}

The embedding of the Cartan subgroup $\C\cong(\CC^\times)^n$ into $\GL(n,\CC)$ induces the {\it Cartan locus} of the moduli space of semistable Higgs bundles
$$
\Car = \left\{(E,\varphi)\in\M_n\ \left| \begin{array}{l}
\exists \, s \in H^0(X,E/\C),
\\
\varphi\in H^0(X,E_s(\mathfrak{c})\otimes K).
\end{array}\right. \right\},
$$
where $\mathbf{c}=\mathrm{Lie}(C)$ and $E_s$ is the principal $\C$-bundle on $X$ constructed from the section $s$. Observe that $\Car$ is the image of the injective morphism 
$$
c : \Sym^n \left (\M_1 \right ) \longrightarrow \M_n,
$$
which is hyperholomorphic, so $\Car$ is a hyperholomorphic subvariety.

%
%
%Note that both, the moduli space of topologically trivial rank $1$ Higgs bundles $(\M_1, \Gamma_1) \cong T^*\Jac^0(X)$ and the moduli space of rank $1$ flat connections $(\M_1, \Gamma_2) \cong \Loc_1(X)$ fiber algebraically over $\Jac^0(X)$. In fact, this projection
%$$
%\M_1 \longrightarrow \Jac^0(X)
%$$
%is hyperholomorphic. As a immediate consequence, the induced map
%$$
%p : \Sym^n \left (\M_1 \right ) \longrightarrow \Sym^n \left ( \Jac^0(X) %\right )
%$$
%is hyperholomorphic as well.
%
%These remarks imply the following proposition.
%
%\begin{proposition}\label{pr hyperholo}
%Suppose that $F$ is a vector bundle on $\Sym^n \left ( \Jac^0(X) \right )$, then $p^*F$ defines a hyperholomorphic vector bundle on the hyperholomorphic manifold $\Sym^n(\M_1)$.
%\end{proposition}
%
%\begin{proof}
%This follows from the fact that the Chern connection of a holomorphic bundle is of type $(1,1)$, that curvature commutes with pullbacks and that holomorphic maps respect types under pullback.
%\end{proof}
%
%In view of Proposition \ref{pr hyperholo}, one has that $\Fff := c_* p^* F$ is a hyperholomorphic vector bundle on $\Car$, the Cartan locus of the moduli space. The pair $(\Car, \Fff)$ constitutes a $\BBB$-brane on the Higgs moduli space $\M_n$.

Now we address the construction of the hyperholomorphic sheaf on $\Car$ for any topologically trivial line bundle $\Ll \to X$. Since a flat bundle is hyperholomorphic and the morphism $c$ is a hyperholomorphic morphism, it will suffice to construct a flat bundle on $\Sym^n(\M_1)$ and take its direct image under $c$.

After fixing a point $x_0 \in X$ we get an embedding $X \hookrightarrow \Jac^0(X)$. Consider our initial line bundle $\Ll \to X$, and let $\nabla_\Ll$ be a flat connection on it. Denote by $(\check{\Ll}, \check{\nabla}_\Ll)$ the unique flat line bundle in $\Jac^0(X)$ that restricts to $(\Ll, \nabla_\Ll)$. From a flat line bundle on $\Jac^0(X)$ one can define a flat line bundle on $\Sym^n(\Jac^0(X))$ as we explain in the following lemma.

\begin{lemma}\label{lm descent sym}
Let $(\check{\Ll}, \check{\nabla}_\Ll)$ be a flat line bundle on $\Jac^0(X)$. Consider 
$$
\pi_i: (\Jac^0(X))^{\times n}\to \Jac^0(X) 
$$
the projection onto the $i$-th factor. Let 
$$
\check{\Ll}^{\boxtimes n} :=\bigotimes_{i=1}^n\pi^*_i \check{\Ll}
$$ 
and 
$$
\check{\nabla}_\Ll^{\boxtimes n} :=\sum_{i=1}^n \pi^*_i \check{\nabla}_{\check{\Ll}} \otimes \bigotimes_{j \neq i}\id_{\pi^*_j\check{\Ll}}.
$$ 
Then $\left ( \check{\Ll}^{\boxtimes n}, \check{\nabla}_\Ll^{\boxtimes n} \right )$ is a flat bundle that descends to a flat bundle $\left (\check{\Ll}^{(n)}, \check{\nabla}_\Ll^{(n)} \right )$ on $\Sym^n(\Jac^0(X))$. 
\end{lemma}

\begin{proof}
The bundle $\check{\Ll}^{\boxtimes n}$ is invariant by the action of $\sym_n$ and moreover the natural linearization action derived from the one on the bundle $\oplus_{i=1}^n \check{\Ll}$ satisfies that over point $p\in (\Jac^0(X))^{\times n}$ with non trivial centraliser $Z_p\subset \sym$, the centraliser $Z_p$ acts trivially on $\check{\Ll}^{\boxtimes n}_p$. It follows from Kempf's descent lemma that $\check{\Ll}^{\boxtimes n}$ descends to a line bundle  $\check{\Ll}^{(n)}$ on $\Sym^n(\Jac^0(X))$
$$
\check{\Ll}^{(n)} := \left (q_* \check{\Ll}^{\boxtimes n} \right )^{\sym_n},
$$
where $q$ denotes the projection $\Jac^0(X)^{\times n} \to \Sym^n(\Jac^0(X))$.

Note that $\check{\nabla}_\Ll^{\boxtimes n}$ is flat since the $\pi_i^*\check{\nabla}_\Ll$ are flat and for any two $i \neq j$, one has that $\pi_i^*\check{\nabla}_\Ll$ and $\pi_j^*\check{\nabla}_\Ll$ commute. By equivariance with respect to the action of the symmetric group $\sym_n$, it descends to a flat connection $\check{\nabla}_\Ll^{(n)}$ on $\check{\Ll}^{(n)}$.
\end{proof}
%\begin{remark}\label{rk hyperholo}
%The same argument yields a hyperholomorphic bundle on $\Car$ for any choice of a flat vector bundle $(\Ff, \nabla) \to \Jac^0(X)$.
%\end{remark}

Recall that the moduli space of topologically trivial rank $1$ Higgs bundles fibres over the Jacobian, $\M_1 \longrightarrow \Jac^0(X)$. This fibration extends to the symmetric product 
$$
p: \Sym^n \left (\M_1 \right ) \longrightarrow \Sym^n \left (\Jac^0(X) \right ).
$$
Then, the flat line bundle $(\check{\Ll}^{(n)}, \check{\nabla}^{(n)})$ gives a flat line bundle $p^*\left (\check{\Ll}^{(n)}, \check{\nabla}_\Ll^{(n)} \right )$ on $\Sym^n(\M_1)$ and further a hyperholomorphic sheaf 
$$
(\Lll, \nabla_\Lll) = c_*p^*(\check{\Ll}^{(n)}, \check{\nabla}_\Ll^{(n)})
$$
on the Cartan locus $\Car$. Consider the pair 
$$
\Ccc(\Ll) := (\Car, (\Lll, \nabla_\Lll)).
$$ 
The above discussion implies the following.
\begin{proposition}\label{prop Car}
$\Ccc(\Ll)$ is a $\BBB$-brane on $\M_n$, which we call {\it Cartan $\BBB$-brane associated to the line bundle $\Ll \to X$}.
\end{proposition}

Note that the image of the Cartan locus under the Hitchin map coincides with the locus of totally reduced spectral curves,
\[
h(\Car) \cong \V.
\]

We finish this section with a description of the intersection of the Cartan locus with a generic Hitchin fibre associated to a nodal curve. Recall that the push-forward map $\check{\nu}$ is an injective morphism as we have seen in Lemma \ref{lm spec Bor}.

\begin{proposition} \label{pr intersection of the BBB with the Hitchin fibre}
For any $v \in \V^\nod$, one has
\[
h^{-1}(v) \cap \Car = \check{\nu}\left ( \Jac^{\ol{0}}(\wt{X}_v) \right ) \cong \Jac^{\ol{0}}(\wt{X}_v).
\]
Consider the isomorphism  
\begin{equation}\label{eq iso pic0specn}
m:\Jac^{\, \overline{0}}(\wt{X}_v) \cong \Jac^0(X)^{\times n}
\end{equation}
induced by the ordering $(X_1, X_2,\dots, X_n)$ of the connected components of $\wt{X}$. One has that under the isomorphism $m$:

\begin{enumerate}
\item \label{it description of spectral datum in Car} The spectral datum  $L\in  \check{\nu}\left ( \Jac^{\ol{0}}(\wt{X}_v) \right )$ corresponding to $\bigoplus_{i=1}^n (L_i,\alpha_i)\in \Car$ is taken to $(L_1,\dots, L_n)\in \Jac^{\ol{0}}(\wt{X})^{\times n}$. Namely, $L=\nu_*F=\bigoplus_j(\iota_j)_*L_j$ where $\iota_j$ is as in \eqref{eq not curves} and $F\in\Jac(\wt{X})$ restricts to $F|_{X_j}=L_j$.

\item \label{it restriction of BBB to Hitchin fibre} The restriction of $\Lll \to \Car$ to $h^{-1}(v) \cap \Car$ corresponds to $\check{\Ll}^{\boxtimes{n}} \to \Jac^0(X)^{\times n}$ defined in Lemma \ref{lm descent sym}.
\end{enumerate}
\end{proposition}

\begin{proof}
\eqref{it description of spectral datum in Car} By construction, a Higgs bundle in $\Car$ decomposes as a direct sum of line bundles,
$$
(E,\varphi) \cong \bigoplus_{i = 1}^n (L_i, \alpha_i).
$$
After Corollary \ref{co description check nu}, $\check{\nu}(\Jac^{\ol{0}})\subset 
h^{-1}(v) \cap \Car$. Now, let $L\in\ol{\Jac}^\delta(\spec_v)$ be the spectral datum corresponding to and element $(E,\varphi)\in h^{-1}(v) \cap \Car$. It is easy to see that the Higgs bundle is totally decomposable if and only if its 
 $\pi_*\mc{O}_{\spec_v}$-module structure
factors through a $\pi_*\nu_*\mc{O}_{\wt{X}}\cong \mc{O}_X^{\oplus n}$-module structure. Hence $L=\nu_*F$ for some $F\in \Jac(\wt{X})$. Corollary \ref{co description check nu} finishes the proof, as the only possible multidegree is $(0,\dots, 0)$.

\eqref{it restriction of BBB to Hitchin fibre} In order to prove the second statement, note that the isomorphism \eqref{eq iso pic0specn} is totally determined by a choice of an ordering of the connected components of $\wt{X}$, in this case $(X_1,\dots,X_n)$. Now, the choice of such an ordering induces an embedding $j:(\Jac^0(X))^{\times n}\plonge \Sym^n(\Jac^0(X))$ making the following diagram commute:
$$
\xymatrix{
\Jac^0(X))^{\times n}\ar@/_3pc/[dd]_q\ar@{^(->}[d]_j\ar[r]^{\quad m} & h^{-1}(u)\ar@{^(->}[d]^i
\\
\Sym^n(\M_1)\ar[d]_p\ar[r]^{\quad c} & \Car
\\
\Sym^n(\Jac^0(X)), &
}
$$
with $q=p\circ j$ being the usual quotient map. We need to check that
$$
m^*i^* \Lll \cong \check{\Ll}^{\boxtimes{n}}.
$$
But, since the above diagram commutes and $c$ is an injection, the LHS is equal to $j^*c^* \Lll =j^* c^* c_* p^*\check{\Ll}^{(n)} \cong j^*p^*\check{\Ll}^{(n)} \cong q^* \check{\Ll}^{(n)}$ and the statement follows by the construction of $\check{\Ll}^{(n)}$.
\end{proof}

%\subsection{Spectral data for the Cartan locus}
%\label{sc Spectral for Cartan}

%In this section we compute the fibers of the Hitchin map restricted to the hyperholomorphic subvariety $\Car \subset \M_n$.

\section{$\BAA$-branes from the unipotent radical of the Borel subgroup}
\label{sc BAA brane}

%\subsection{An isotropic subvariety}\label{sect iso var}
%\label{sc BAA brane construction}

Recall from Section \ref{sc Higgs moduli spaces} that $\M_n$ is a hyperk\"ahler with $((\Gamma_1, \omega_1), (\Gamma_2, \omega_2), (\Gamma_3,\omega_3))$ being its K\"ahler structures. After \cite{kapustin&witten}, a \emph{$\BAA$-brane} on $\M_n$ is a pair $(\W, (\Gg, \nabla_\Gg))$, with:
\begin{itemize}
\item $\W$ being a complex Lagrangian subvariety of $\M_n$ for the holomorphic symplectic form $\Omega_1 = \omega_2 + \imaginary \omega_3$.
\item $(\Gg,\nabla_\Gg)$ being a flat bundle supported on $\W$.
\end{itemize}

Starting from the line bundle $\Ll \to \Jac^0(X)$, we construct in this section a complex Lagrangian subvariety $\Uni(\Ll)$ of the moduli space of Higgs bundles, mapping to the Cartan locus $\V \subset \H$ of the Hitchin base. As we have seen, $\Uni(\Ll)$ is the support of a $\BAA$-brane after specifying a flat vector bundle on it.

Recall that we have fixed a point $x_0 \in X$. Denote by $\hat{\Ll}$ our topologically trivial line bundle $\Ll \to X$ tensored $\delta/n = (n-1)(g-1)$ times by $\Oo_X(x_0)$, 
\begin{equation} \label{eq definition of hatLl}
\hat{\Ll} := \Ll \otimes \Oo_X(x_0)^{(n-1)(g-1)}.
\end{equation}
Having in mind Proposition \ref{pr Bor 1}, we define the subvariety of $\M_n \times_{\H} \V$,
\begin{equation}\label{eq definition of Uni}
\Uni(\Ll)=\left\{(E,\varphi) \in \Bor \ \left| 
\begin{array}{l}
\exists \, \sigma \in H^0(X,E/\B),\\
\varphi\in H^0(X,E_\sigma(\mathfrak{b})\otimes K),\\
E_{\C}:= E_\sigma/\U \cong (\hat{\Ll}\otimes K^{\otimes 1-n}) \boxplus \dots \boxplus (\hat{\Ll}\otimes K^{-1}) \boxplus \hat{\Ll}.
\end{array}\right. \right\}.
\end{equation}

\begin{proposition}\label{pr Uni closed}
$\Uni(\Ll)$ is closed in $\M_n$.
\end{proposition}

\begin{proof}
Recall that we denoted by $\MMm_n$ the moduli stack of rank $n$ and degree $0$ Higgs bundles and its semistable locus by $\MMm^{\sst}_n\subset\MMm_n$. Recall as well that Theorem \ref{thm:stack_Higgs} (see also the discussion following it) states that $\M_n$ is a good moduli space for $\MMm_n^{\sst}$ and there is a morphism 
$$
\Psi: \MMm^{\sst}_n\longrightarrow\M_n
$$
which induces the quotient topology. 

Let us denote by $\Bbor$ the moduli stack of $\B$-Higgs bundles, that is, the moduli stack classifying pairs $(E_{\B}, \varphi_{\B})$ where $E_{\B}$ is a holomorphic $\B$-bundle and $\varphi_{\B}$ is an element of $H^0(X, E_{\B}(\lie{b}) \otimes K)$. By extension of structure group $\B \hookrightarrow \GL(n,\CC)$ one gets a morphism
\[
\mathbcal{i} : \Bbor \to \mathbcal{M}_n. 
\]
Recalling Theorem \ref{thm:stack_Higgs}, and the definition of $\Bor$, we see that the restriction of $\mathbcal{i}(\Bbor)$ to the semistable locus $\MMm^{\sst}_n$ of $\MMm_n$ surjects to $\Bor$. Also, one can construct the following projection
\[
\morph{\Bbor}{\Jac(X)^n}{(E_{\B}, \varphi_{\B})}{E_{\C} = E_{\B}/\U.}{}{\mathbcal{j}}
\]
Both $\mathbcal{i}$ and $\mathbcal{j}$ are algebraic morphisms hence smooth. Consider the substack of $\MMm_n$ given by 
\[
\mathbcal{U\! n\! i}(\Ll) := \mathbcal{i}\left ( \mathbcal{j}^{-1}((\hat{\Ll}\otimes K^{\otimes 1-n}) \boxplus \dots \boxplus (\hat{\Ll}\otimes K^{-1}) \boxplus \hat{\Ll}) \right ).
\]
Again, thanks to Theorem \ref{thm:stack_Higgs} and the construction of $\Uni(\Ll)$, we have that the restriction to the semistable locus, $\mathbcal{U\! n\! i}(\Ll)^{\sst}:=\mathbcal{U\! n\! i}(\Ll)\cap\MMm^{\sst}_n$, surjects to $\Uni(\Ll)$. Note that $\mathbcal{j}^{-1}((\hat{\Ll}\otimes K^{\otimes 1-n}) \boxplus \dots \boxplus (\hat{\Ll}\otimes K^{-1}) \boxplus \hat{\Ll})$ is a closed substack of $\Bbor$ as it is the preimage of a closed point, then $\mathbcal{U\! n\! i}(\Ll)$ is closed insidse $\mathbcal{i} \left ( \Bbor \right )$. We now observe that it is enough to prove that $\mathbcal{i}(\Bbor)$ is closed in $\MMm_n$ as this would imply that $\mathbcal{U\! n\! i}(\Ll)$ is closed in $\MMm_n$. Now, by Theorem \ref{thm:stack_Higgs} the previous discussion implies that $\mathbcal{U\! n\! i}(\Ll)^{\sst}$ is closed inside $\MMm_n^{\sst}$, and thus maps onto a closed subset, proving the statement.

 %by Proposition \ref{pr Bor 1}, $\Bor$ is closed in $\M_n$. So if $\mathbcal{i}(\Bbor)$ is also closed, then so is

Now, universal closedness of $\mathbcal{i}(\Bbor)$ follows from the valuative criterion, as the image of $\Bbor$
has a universal bundle $(\mathbcal{E},\mathbf{\Phi})$ admitting a reduction of the structure group to $\B$. Given a discrete valuation ring $R$ with fraction field $k$, properness of $\GL_n(\CC)/\B$ ensures that the existence of a reduction of the structure group over $\Spec(k)$ extends uniquely to $\Spec(R)$. This proves the valuative criterion for the bundle. Now, assume that the universal Higgs field defines a $\B$-equivariant morphism
$$
\phi:\mathbcal{E}_{\B}|_{\Spec(k)}\longrightarrow \mathfrak{b}\otimes K,
$$
where $\mathbcal{E}_{\B}$ denotes the universal bundle together with a reduction to $\B$. Since $\phi$ extends to $\phi':\mathbcal{E}|_{\Spec(R)}\longrightarrow \mathfrak{gl}(n,\CC)\otimes K$, closedness of $\mathfrak{b}\subset\mathfrak{gl}(n,\CC)$ \'etale local  triviality of $\mathbcal{E}|_{\Spec(R)}$ do the rest.

%. Recalling Proposition \ref{pr Bor 1} and the subsequent Remark \ref{rm Bor stack}, it follows that $\mathbcal{U\! n\! i}(\Ll)$ is a closed substack of $\MMm_n$. Therefore $\Uni(\Ll)$ is closed.
\end{proof}

In order to prove that $\Uni(\Ll)$ is an isotropic submanifold of $(\M_n,\Omega_1)$ we first give a description of it in gauge theoretic terms. Let $\EE$ denote the topologically trivial rank $n$ vector bundle; choose a reduction of the structure group to $\B$ (which always exists), and let $\EE_{\B}$ be the corresponding principal $\B$-bundle, so that $\EE\cong\EE_{\B}(\GL(n,\CC))$. Define $\EE_{\C}=\EE_{\B}/\U$. It follows from \eqref{eq definition of Uni} that 
\begin{equation}\label{eq dual brane}
\Uni(\Ll)=\left\{
(\ol{\partial}_A,\varphi)\in \M_n\  \left| \begin{array}{l}
\exists g\in \mc{G} \ {\rm satisfying } \\
\ 1) \ g\cdot\ol{\partial}=\ol{\partial}_{\C}+N, {\rm where }\\ 
\phantom{XX}N\in \Omega^{0,1}(X,\EE_{\B}(\mathfrak{n})),
\\
\phantom{XX}
(\EE_{\C},\ol{\partial}_{\C}) = (\hat{\Ll}\otimes K^{\otimes 1-n}) \boxplus \dots \boxplus (\hat{\Ll}\otimes K^{-1}) \boxplus \hat{\Ll};
\\
\ 2)\ g\cdot\varphi\in \Omega^0(X,\EE_{\B}(\mathfrak{b})\otimes K).
\end{array}\right.\right\}.
\end{equation}
\begin{remark}
Both $\Car$ and $\Uni(\Ll)$ are subvarieties of $\M_n \times_{\H} \V$, but they do not intersect, as the elements of $\Car\cap\Uni(\Ll)$ would have underlying bundle of the form $E_{\C}$ in \eqref{eq definition of Uni}, which is unstable, and totally decomposable Higgs field, conditions which yield unstable Higgs bundles. 
\end{remark}

\begin{proposition} \label{pr uni isotropic}
The complex subvariety $\Uni(\Ll)$ of $\M_n$ is isotropic with respect to the symplectic form $\Omega_1$ defined in \eqref{eq Omega1}.
\end{proposition}

\begin{proof}
It is enough to prove the statement for open subset of stable points in $\Uni(\Ll)$. We will check this subset is non empty in Proposition \ref{pr filtration for L-normalized}. 

So let $(E,\varphi)\in\Uni(\Ll)$ be a stable point. By \eqref{eq dual brane}, a vector $(\dot{A},\dot{\varphi})\in T_{(E,\varphi)}\M_n$ satisfies that, up to the adjoint action of the gauge Lie algebra,  
$$
(\dot{A},\dot{\varphi})\in\Omega^{0,1}\left(X,\EE_{\B}(\mathfrak{n})\right)\times \Omega^0(X,\EE_{\B}(\mathfrak{b})\otimes K).
$$
The result follows from gauge invariance of the symplectic form $\Omega_1$ and the fact that 
$\mathfrak{n}\subset\mathfrak{b}^{\perp}$, 
where orthogonality is taken with respect to the Killing form.
\end{proof}

We now give a description of the spectral data of the Higgs bundles corresponding to the points of $\Uni(\Ll)$. We will focus on the open subset of those Higgs bundles whose spectral data is a line bundle. This will allow us to show that this subvariety is mid-dimensional, and, after Proposition \ref{pr uni isotropic}, Lagrangian. 

\begin{proposition} \label{pr spectral data of Uni}
Let $\hat{\Ll}$ be defined as in \eqref{eq definition of hatLl}. For every $v \in \V^\nod$, one has the following identification inside $h^{-1}(v)$,
\begin{equation} \label{eq description of the Hitchin fibre of Uni}
\Uni(\Ll) \cap \Jac^\delta(\ol{X}_v) = \left \{ L \in{\Jac}^{\, \delta}(\ol{X}_v) \textnormal{ such that } \nu^*L =p^*\hat{\Ll}\cong\left ( \hat{\Ll}, \dots, \hat{\Ll} \right ) \right \}.
\end{equation}
Furthermore, Higgs bundles described in \eqref{eq description of the Hitchin fibre of Uni} are stable.
\end{proposition}

\begin{proof}
Thanks to Proposition \ref{pr description line bundles on V^nod}, we have that the spectral datum $L$ of any $(E,\varphi)\in \Uni(\Ll)\cap\Jac^\delta(\spec_v)$ satisfies
 $$
 \hat{\Ll}=\alpha_{i}^*\iota_{i}^*L.
 $$
Now, since any line bundle on $\wt{X}_v$ is totally determined by its restriction to all the connected components, 
it is enough to check that
$j_i^*p^*\hat{\Ll}=j_i^*\nu^*L$, which follows from commutativity of the arrows in \eqref{eq not curves} and the fact that $\alpha_i:X\to X_i$ is an isomorphism. This concludes the proof.% of \eqref{eq description of the Hitchin fibre of Uni}.

%All there is left to check is that all line bundles are in the RHS of \eqref{eq description of the Hitchin fibre of Uni} are stable, as by Proposition \ref{pr Bor 2}, this implies automatically that they are inside $\Uni(\Ll)$.  
 
%Let $L \in \Jac^\delta(\ol{X}_v)$, which by Theorem \ref{thm spectral corresp} is the spectral data of $(E,\varphi)$, semistable or unstable. We will check that all such line bundles satisfy the strict inequality in \eqref{eq Schaub}. First of all, any subscheme pure of dimension $1$ of $\spec$ is of the form
%$$
%Z_I=\bigcup_{i\in I} X_i,\qquad Y_I=\bigcup_{i\in I^c} X_i.
%$$
%for some $I\subset \{1,\dots, n\}$.

%Now, $L$ being a line bundle, it follows that any rank one torsion free quotient of $L|_{Z_I}$ must be isomorphic to $L|_{Z_I}$, so it is enough to check that $L|_{Z_I}$ satisfies the strict inequalities in \eqref{eq Schaub}, and therefore is stable. This is an easy computation.
\end{proof}

%We begin in Section \ref{sect Hitchin fibers nodal} by studying the spectral data over the singular locus $\V^\nod$. We recall that this is the subset of $\V$ whose corresponding spectral curves $\ol{X}_v$ are nodal curves. In Section \ref{sect Hitchin fibers Uni} we specify the latter to $\Uni(\Ll)$. 

%\subsubsection{Generic spectral data for $\Uni(\Ll)$}\label{sect Hitchin fibers Uni}

The description of the spectral data given in Proposition \ref{pr spectral data of Uni} allows us to study the dimension of $\Uni(\Ll)$, which turns up to be one half of $\dim \M_n$.

\begin{proposition} \label{pr mid-dimensionality of Uni}
The complex subvariety $\Uni(\Ll)$ of $\M_n$ has dimension
$$
\dim \Uni(\Ll) = n^2 (g-1) + 1 = \frac{1}{2} \dim \M_n.
$$
\end{proposition}

\begin{proof}
First, we observe that $\Uni(\Ll)$ is a fibration over $\V$ and recall that $\dim \V = n g$. By Proposition \eqref{pr filtration for L-normalized}, over the dense open subset $\V^\nod \subset \V$, the fibre of $\Uni(\Ll)|_{\V^\nod} \to \V^\nod$ at $v$ has a dense open subset 
$$
\hat{\nu}^{-1}(\hat{\Ll}, \dots, \hat{\Ll}) \subset \ol{\Jac}^{\, \delta}(\ol{X}_v) \cong h^{-1}(v),
$$
where we recall the pull-back map described in \eqref{eq hatnu}. Now, by Corollary \ref{co hat nu}, 
$$
\hat{\nu}^{-1}(\hat{\Ll}, \dots, \hat{\Ll})\cong \left(\CC^\times\right)^{\delta-n+1}.
$$
By smoothness of the point, the Hitchin fiber is transverse to the (local) Hitchin section, so 
\begin{align*}
\dim \Uni(\Ll)|_{\V^{\nod}} = & \dim \V^\nod + \dim \hat{\nu}^{-1} \left( \hat{\Ll}, \dots, \hat{\Ll}  \right)
\\
= & ng + \delta - n + 1
\\
= & ng + (n^2 - n)(g - 1) - n + 1
\\
= & n^2(g - 1) + 1.
\end{align*}
which is half of the dimension of $\M_n$, as we recall from \eqref{eq dim M_n}. This finishes the proof since by Proposition \ref{pr uni isotropic}, $\Uni(\Ll)$ is isotropic, so its dimension can not be greater than $\frac{1}{2}\dim \M_n$.
\end{proof}

Finally, we can state the main result of the section.

\begin{theorem} \label{tm uni Lagrangian}
The complex subvariety $\Uni(\Ll)$ of $\M_n$ is a closed complex Lagrangian with respect to $\Omega_1$. 
\end{theorem}

\begin{proof}
This is clear after Propositions \ref{pr Uni closed}, \ref{pr uni isotropic} and \ref{pr mid-dimensionality of Uni}.
\end{proof}

%By Theorem \ref{tm uni Lagrangian} we can define the {\it Borel unipotent $\BAA$-brane associated to $\Ll$} as the triple given by
%$$
%\UUni(\Ll) := \left (\Uni(\Ll), \Oo_{\Uni(\Ll)}, \nabla \right ), 
%$$
%where $\nabla$ is the trivial connection on $\Oo_{\Uni(\Ll)}$.

Thanks to Proposition \ref{pr Bor 2} we have at hand a description of every point in the Hitchin fibers over $\V^\nod$. Hence we can study the intersection of these fibres with $\Uni(\Ll)$ as we will do in the remaining of the section. Before stating the result we need some extra definitions. Let $v = (\alpha_1, \dots, \alpha_n)_{\sym_n}$ in $\V^\nod$ giving the spectral curve $\ol{X}_v$ with singular divisor $D\subset \ol{X}_v$, and let $R\subset D$ be a subdivisor. We have seen that $\ol{X}_v$ has $n$ irreducible components $X_i = \alpha_i(X)$ and recall that we have set $D_{ij} = X_i \cap X_j$. For each ordering $J = (\alpha_{j_1}, \dots, \alpha_{j_n})$ of the set $\{\alpha_1,\dots,\alpha_n\}$, define the divisors
\begin{equation}\label{eq BJi}
B_{J,i}:=\sum_{i'\geq i+1}D_{j_i j_{i'}}\cap R.
\end{equation}
Set also
\[
b_{J,i} := |B_{J,i}|. 
\]

\begin{proposition} \label{pr filtration for L-normalized}
Let $\hat{\Ll}$ be defined as in \eqref{eq definition of hatLl} and let $v \in \V^\nod$ with spectral curve $\ol{X}_v$ and divisor of singularities $D$. Chose $R \subset D$ %with $R_s \neq \emptyset$ 
and consider the associated desingularization $\wt{X}_R$ of $\ol{X}_v$. Then, for any $n$-tuple of integers $\ol{d} = (d_1, \dots, d_n)$, we have the following identifications inside $h^{-1}(v)$,
\begin{equation} \label{eq description of the Hitchin fibre of Uni bis}
\Uni(\Ll) \cap \check{\nu}_R \left ( \Jac^{\ol{d}}(\wt{X}_R)\right ) = \left \{
L \in \Jac^{\ol{d}}(\wt{X}_R) \ 
 \left| \begin{array}{l}
\exists \, J = (\alpha_{j_1}, \dots, \alpha_{j_n}) \textnormal{ ordering of } \{\alpha_1, \dots, \alpha_n \}  \\
\textnormal{such that, for all } 1 \leq i \leq n, \textnormal{ we have: }\\
\, a) \, d_i = \delta - b_{J,i} \textnormal{ and}\\
\, b) \, L|_{X_{j_i}} \cong \hat{\Ll} \otimes \Oo(B_{J,i}).
\end{array}\right.
\right \},
\end{equation}
when $R_s = \emptyset$ and $\ol{d}$ satisfies $b)$ for some ordering $J$, and 
\[
\Uni(\Ll) \cap \check{\nu}_R \left ( \Jac^{\ol{d}}(\wt{X}_R)\right ) = \emptyset,
\]
in contrary case.
\end{proposition}

\begin{proof}
Recall the notation of Proposition \ref{pr Bor 1}. Take $(E,\varphi) \in h^{-1}(v)$ where $v \in \V^\nod$ is given by $(\alpha_1, \dots, \alpha_n)_{\sym_n}$. Note that $(E,\varphi)\in\Uni(\Ll)$ if and only there exists an ordering $J = (\alpha_{j_1}, \dots, \alpha_{j_n})$ and a filtration 
$$
0=(E_0,\varphi_0)\subsetneq (E_1,\varphi_1)\subsetneq\cdots\subsetneq (E_n,\varphi_n)=(E,\varphi)
$$
such that
$$
\quotient{(E_i,\varphi_i)}{(E_{i-1},\varphi_{i-1})}\cong (\hat{\Ll}\otimes K^{i-1},\alpha_{j_i}).
$$
The statement then follows from Proposition \ref{pr Bor 2}, noting that 
$$
\nu_R^*\left ( K^{i-n} \otimes \Oo(-B_{J,i}) \right )=\Oo \left ( \sum_{i'\geq i+1}\wt{X}_{j_i}\cap\wt{X}_{j_{i'}} \right).
$$
\end{proof}

\section{Duality}
\label{sc duality}

In this section we discuss about the duality under mirror symmetry of the $\BBB$-brane $\CCar(\Ll)$, and a $\BAA$-brane supported on $\Uni(\Ll)$. Ideally, we would like to transform them under a Fourier--Mukai transform between coarse compactified Jacobians of reducible curves. Since such a tool is unavailable, we will make use of the integral functor $\Phi$ between the corresponding moduli stacks. Since the Cartan locus $\Car$ and the Jacobian $\Jac^{\delta}(\overline{X})$ are both fine moduli spaces, we will restrict the Poincar\'e sheaf $\overline{\PPp}$ to $\Car$ on one side and $\Jac^0(\overline{X})$ on the other, obtaining an integral functor $\Phi^{\Car}$ between their derived categories of sheaves. As we will see in this section, $\Phi^{\Car}$ sends our $\BBB$-brane $\CCar(\Ll)$ to the trivial sheaf supported on $\Uni(\Ll)$ what provides evidence of a duality statement between them. A note of warning should be added here: ongoing work by Arinkin and Pantev \cite{pantev} shows that the integral functor $\Phi$ on the stack of Higgs bundles over totally reducible spectral curves need not preserve semistability \cite{pantev}. We do not see this phenomenon occuring here, as we pick the target of $\Phi^\Car$ to be the Jacobian, although this should be taken into account when studying the transform of $\CCar(\Ll)$ under the whole integral functor $\Phi$. 

Recall that in our case, the normalization $\wt{X}_v$ is the disjoint union $\bigsqcup_i X_i$ of copies of the base curve $X$, which is smooth. Then, the direct product of Jacobians $\prod_i \Jac^0(X_i)$ is the moduli space classifying line bundles of multidegree $\overline{0}$, which is a fine moduli space with universal line bundle $\wt{\Uu}$. The restriction of each $X_i$ is a line bundle over an irreducible smooth curve, hence simple. It then follows that the associated moduli stack is
\[
\Jjac^{\overline{0}}(\wt{X}_v) \cong \left [ \quotient{\Jac^{\ol{0}}(\wt{X}_v)}{(\CC^*)^{\times n}}  \right ] \cong \prod_{i = 1}^n \left [ \quotient{\Jac^{\ol{0}}(X_i)}{\CC^*}  \right ],
\]
where each $\CC^*$ acts trivially. Recall also that the restriction of the Cartan locus $\Car$ to the Hitchin fibre associated to $\overline{X}_v$ is $\check{\nu} \left ( \Jac^{\ol{0}}(\wt{X}_v) \right )$. Note that this is a fine moduli space with universal sheaf
\[
\Uu^{\Car} := (\nu \times \check{\nu})_* \wt{\Uu} \longrightarrow \ol{X}_v \times \check{\nu} \left ( \Jac^{\ol{0}}(\wt{X}_v) \right ).
\]
We consider the substack $\check{\nu} \left ( \Jjac^{\overline{0}}(\wt{X}_v) \right )$ of $\Jjac^{\, \delta}(\overline{X}_v)$. By all of the above, we have that 
\[
\check{\nu} \left ( \Jjac^{\ol{0}}(\wt{X}_v) \right ) \cong \left [ \quotient{\check{\nu} \left (\Jac^{\ol{0}}(\wt{X}_v) \right )}{(\CC^*)^{\times n}}  \right ], 
\]
and the restriction of the universal sheaf $\UUu|_{\overline{X}_v \times \check{\nu} \left ( \Jac^{\ol{0}}(\wt{X}_v) \right )}$ pulls-back to $\Uu^\Car$ under the obvious projection
\begin{equation} \label{eq stack surjection for nu Jac}
\check{\nu} \left (\Jac^{\ol{0}}(\wt{X}_v) \right ) \longrightarrow \left [ \quotient{\check{\nu} \left (\Jac^{\ol{0}}(\wt{X}_v) \right )}{(\CC^*)^{\times n}}  \right ].
\end{equation}

It follows from a result of Mumford (see for instance \cite[Theorem 2, Section 8.2]{neron}) that the Jacobian of degree $\delta$ line bundles over a reduced curve $\overline{X}_v$ is a fine moduli space $\Jac^\delta(\overline{X}_v)$ with universal line bundle $\Uu^0 \to \ol{X}_v \times \Jac^{\delta}(\ol{X}_v)$. Since line bundles are simple, one the has that the corresponding moduli stack is the quotient stack
\[
\Jjac^{\, \delta}(\overline{X}_v) \cong \left [ \quotient{\Jac^{\, \delta}(\overline{X}_v)}{\CC^*}  \right ],
\]
for the trivial action of $\CC^*$. One trivially has that $\Uu^0$ is the pull-back of $\UUu^0$ under the projection
\begin{equation} \label{eq stack surjection for Jac}
\Jac^{\, \delta}(\overline{X}_v) \longrightarrow \left [ \quotient{\Jac^{\, \delta}(\overline{X}_v)}{\CC^*}  \right ].
\end{equation}

With $\Uu^0$ and $\Uu^\Car$ we already have all the ingredients for the following definition, analogous to \eqref{eq definition Poincare bundle}, of a Poincar\'e bundle over $\check{\nu} \left ( \Jac^{\ol{0}}(\wt{X}_v) \right ) \times \Jac^{\delta}(\ol{X}_v)$,
\begin{equation} \label{eq definition Cartan Poincare sheaf}
\Pp^\Car := \Dd_{f_{23}} \left ( f_{12}^*\Uu^\Car \otimes f_{13}^* \Uu^0  \right )^{-1} \otimes \Dd_{f_{23}} \left ( f_{13}^* \Uu^0  \right ) \otimes \Dd_{f_{23}} \left ( f_{12}^* \Uu^\Car  \right ),
\end{equation}
where the $f_{ij}$ are the corresponding projections from $\ol{X}_v \times \check{\nu} \left ( \Jac^{\ol{0}}(\wt{X}_v) \right ) \times \Jac^{\delta}(\ol{X}_v)$ to the product of the $i$-th and $j$-th factors. 

We can see that $\Pp^\Car$ is obtained from the restriction of the Poincar\'e sheaf $\overline{\PPp}$ to the Cartan locus and the Jacobian of $\overline{X}_v$.

\begin{proposition} \label{pr meaning of Pp^Car}
The sheaf $\Pp^\Car$ is the pull-back of $\overline{\PPp}|_{\check{\nu} \left ( \Jjac^{\ol{0}}(\wt{X}_v) \right ) \times \Jjac^\delta(\overline{X})}$ under the product of morphisms \eqref{eq stack surjection for nu Jac} and \eqref{eq stack surjection for Jac}. 
\end{proposition}

\begin{proof}
Since $\overline{\PPp}$ extends $\PPp \to \overline{\Jjac}^{\, \delta}(\overline{X}_v) \times \Jjac^{\, \delta}(\overline{X}_v)$, we have from \eqref{eq definition Poincare bundle} that
\begin{align*}
\overline{\PPp}|_{\check{\nu} \left ( \Jjac^{\ol{0}}(\wt{X}_v) \right ) \times \Jjac^\delta(\overline{X})} & \cong \PPp|_{\check{\nu} \left ( \Jjac^{\ol{0}}(\wt{X}_v) \right ) \times \Jjac^\delta(\overline{X})}
\\
& \cong \Dd_{f_{23}} \left ( f_{12}^*\UUu|_{\overline{X}_v \times \check{\nu} \left ( \Jjac^{\ol{0}}(\wt{X}_v) \right )} \otimes f_{13}^* \UUu^0  \right ) \otimes \Dd_{f_{23}} \left ( f_{13}^* \UUu^0  \right )^{-1}
\\
& \qquad \otimes \Dd_{f_{23}} \left ( f_{12}^* \UUu|_{\overline{X}_v \times \check{\nu} \left ( \Jjac^{\ol{0}}(\wt{X}_v) \right )}  \right )^{-1}.
\end{align*}
Then, the result follows from the observation that $\Uu^\Car$ is the pull-back of $\check{\nu} \left ( \Jjac^{\ol{0}}(\wt{X}_v) \right )$ under \eqref{eq stack surjection for nu Jac}, and $\Uu^0$ is the pull-back of $\UUu^0$ under \eqref{eq stack surjection for Jac}.
\end{proof}

Let us consider the integral functor associated to $\Pp^\Car$, 
\begin{equation} \label{eq Cartan FM restricted to Pic}
\morph{D^b \left ( \check{\nu} \left ( \Jac^{\ol{0}}(\wt{X}_v) \right ) \right )}{D^b \left ( \Jac^\delta(\ol{X}_v) \right )}{\Ee^\bullet}{R \pi_{2,*}(\pi_1^*\Ee^\bullet \otimes \Pp^\Car),}{}{\Phi^\Car}
\end{equation}
where $\pi_1$ and $\pi_2$ to be, respectively, the projection from $\check{\nu} \left ( \Jac^{\ol{0}}(\wt{X}_v) \right ) \times \Jac^\delta(\ol{X}_v)$ to the first and second factors.

Recall that our $\BBB$-brane $\CCar(\Ll)$ is given by the hyperholomorphic bundle $\Lll$ supported on $\Car$. By Proposition \ref{pr intersection of the BBB with the Hitchin fibre}, over the dense open subset $\V^\nod$ of the Cartan locus of the Hitchin base $\V = h(\Car) \subset \H$, the hyperholomorphic sheaf $\Lll$ restricted to a certain Hitchin fibre $\overline{\Jac}^{\, \delta}(\ol{X}_v)$ is $\check{\nu}_*\check{\Ll}^{\boxtimes n}$, supported on $\check{\nu}(\Jac^{\overline{\delta}}(\wt{X}_v))$. The main result of this section is the study of the behaviour of $\check{\nu}_*\check{\Ll}^{\boxtimes n}$ under $\varphi^{\Car}$, but first we need some technical results.

Fix $x_0$ and take the line bundle $\Oo(x_0)^{(n-1)(g-1)}$. Denote 
$$
\tau : \Jac^{\ol{0}}(\wt{X}) \stackrel{\cong}{\longrightarrow} \Jac^{\ol{\delta}}(\wt{X})
$$
the isomorphism given, on each of the components, by tensorization by the previous line bundle. We can define a Poincar\'e bundle $\wt{\Pp} \to \Jac^{\overline{0}}(\wt{X}_v) \times \Jac^{\overline{\delta}}(\wt{X}_v)$.

Consider the projections to the first and second factors
$$
\xymatrix{
 &  \Jac^{\overline{0}}(\wt{X}_v) \times \Jac^{\overline{\delta}}(\wt{X}_v) \ar[ld]_{\widetilde{\pi}_1} \ar[rd]^{\widetilde{\pi}_2} &
\\
\Jac^{\overline{0}}(\wt{X}_v) & & \Jac^{\overline{\delta}}(\wt{X}_v),
}
$$
and, using $\wt{\Pp}$, one can construct another Fourier--Mukai integral functor
$$
\morph{D^b(\Jac^{\overline{0}}(\wt{X}_v))}{D^b(\Jac^{\overline{\delta}}(\wt{X}_v))}{\Ee^\bullet}{R\widetilde{\pi}_{2,*} (\widetilde{\pi}_1^*\Ee^\bullet \otimes \wt{\Pp}).}{}{\wt{\Phi}}
$$
Note that $\wt{\Phi}$ is governed by the usual Fourier--Mukai transform on each of the $\Jac^0(X_i)$. We need the following lemma in order to describe the interplay between $\Phi^\Car$ and $\wt{\Phi}$. 

\begin{lemma} \label{lm relation of Poincares}
One has that
$$
\left ( \check{\nu} \times \id_{\Jac}  \right )^*\Pp^\Car \cong (\id_{\widetilde{\Jac}} \times \hat{\nu})^*\wt{\Pp}. 
$$
\end{lemma}

\begin{proof}
Note that $\left ( \check{\nu} \times \id_{\Jac}  \right )^*\Pp^\Car$ is a family of line bundles over $\Jac^{\overline{0}}(\wt{X})$ para\-me\-trized by $\Jac^{\delta}(\ol{X}_v)$. Since $\wt{\Pp} \to \Jac^{\overline{0}}(\wt{X}_v) \times \Jac^{\overline{0}}(\wt{X}_v)$ is a universal family for these objects, there exists a map
$$
t : \Jac^{\delta}(\ol{X}_v) \longrightarrow \Jac^{\overline{0}}(\wt{X}_v),
$$
such that 
$$
\left ( \check{\nu} \times \id_{\Jac}  \right )^*\Pp^\Car \cong (\id_{\widetilde{\Jac}} \times t)^*\wt{\Pp}. 
$$
Recall the description of $\Pp_J$ given in \eqref{eq description of Pp_J} for each $J \in \Jac^{\delta}(\ol{X}_v)$. Recall as well the projections $f_1 : \ol{X}_v \times \overline{\Jac}^{\, \delta}(\ol{X}_v) \to \ol{X}_v$ and $f_2 : \ol{X}_v \times \overline{\Jac}^{\, \delta}(\ol{X}_v) \to \overline{\Jac}^{\, \delta}(\ol{X}_v)$, and consider the following commuting Cartesian diagram
$$
\xymatrix{
\ol{X}_v \times \Jac^{\overline{0}}(\wt{X}_v) \ar[rr]^{\id_{\wt{X}} \times \check{\nu}} \ar[d]_{f'_2} & & \ol{X}_v \times \overline{\Jac}^{\, \delta}(\ol{X}_v) \ar[d]^{f_2}
\\
\Jac^{\overline{0}}(\wt{X}_v) \ar[rr]^{\check{\nu}} & & \overline{\Jac}^{\, \delta}(\ol{X}_v)
}
$$
We know from \cite[Proposition 44 (1)]{esteves} that the determinant of cohomology commutes with base change, {\it i.e.} 
\begin{equation}\label{eq:basechange}
\check{\nu}^* \Dd_{f_2} = \Dd_{f'_2} (\id_{\ol{X}} \times  \check{\nu})^*. 
\end{equation}

Consider the obvious projection $\wt{f}_2 : X_\gamma \times \Jac^{\overline{0}}(\wt{X}_v) \to \Jac^{\overline{0}}(\wt{X}_v)$. Since the following diagram commutes,
\[
\xymatrix{
\wt{X}_v \times \Jac^{\overline{0}}(\wt{X}_v) \ar[rr]^{\nu \times \id_{\widetilde{\Jac}}} \ar[rrd]_{\widetilde{f}_2} & & \ol{X}_v \times \Jac^{\overline{0}}(\wt{X}_v) \ar[d]^{f'_2}
\\
& & \Jac^{\overline{0}}(\wt{X}_v),
}
\]
the definition of the determinant of cohomology ensures that
\begin{equation} \label{eq determinant of coh and pushforward}
\Dd_{f'_2} (\nu \times \id_{\widetilde{\Jac}})_* \cong \Dd_{\widetilde{f}_2}.
\end{equation}

One also has that the following diagrams commute
\begin{equation} \label{eq comm diagram for f}
\xymatrix{
\ol{X}_v \times \Jac^0(X_\gamma) \ar[rr]^{\id_{\ol{X}} \times \check{\nu}} \ar[rrd]_{f'_1} & & \ol{X}_v \times \overline{\Jac}^\delta(\ol{X}_v) \ar[d]^{f_1}
\\
 & & \ol{X}_v   ,}
\end{equation}
and
\[
\xymatrix{
 & & & \wt{X}_v \times \Jac^{\overline{0}}(\wt{X}_v) \ar[llld]_{\nu \circ \widetilde{f}_1} \ar[rrrd]^{\widetilde{f}_2} \ar[d]^{(\nu \times \id_{\widetilde{\Jac}})} & & &
\\
\ol{X}_v & & & \ol{X}_v \times \Jac^{\overline{0}}(\wt{X}_v) \ar[lll]^{f'_1} \ar[rrr]_{f'_2} & & & \Jac^{\overline{0}}(\wt{X}_v).
}\]
As a consequence, one has that $f'_1((f'_2)^{-1}(U)) = \nu \widetilde{f}_1 (\widetilde{f}_2^{-1}(U))$ for every open subset $U \subset \Jac^{\overline{0}}(\wt{X}_v)$. It then follows from the definition of pull-back and pusforward that, for any $J \in \Jac^\delta(\ol{X}_v)$, 
\begin{align*}
(f'_2)_* (f'_1)^* J(U) = & \lim_{\tiny{W \supseteq f'_1((f'_2)^{-1}(U))}} J(W)
\\
= & \lim_{\tiny{W \supseteq f'_1((f'_2)^{-1}(U))}} J(W) 
\\
= & (\widetilde{f}_2)_*(\nu \circ \widetilde{f}^1)^* J(U), 
\end{align*}
so $(f'_2)_* (f'_1)^* = (\widetilde{f}_2)_*(\nu \circ \widetilde{f}_1)^*$ and therefore, 
\begin{equation} \label{eq relation of det coh for M}
\Dd_{f'_2} (f'_1)^* \cong \Dd_{\widetilde{f}_2} \widetilde{f}_1^* \nu^*.
\end{equation}

Recalling the definition of $\Uu^\Car$ as $(\nu \times \check{\nu})_*\wt{\Uu}$, we observe that
\begin{equation} \label{eq relation between Uu_b and Uu_gamma}
(\id_{\ol{X}} \times \check{\nu})^* \Uu^\Car \cong (\nu \times \id_{\widetilde{\Jac}})_* \wt{\Uu}.
\end{equation}

Using the projection formula and \eqref{eq:basechange}--\eqref{eq relation between Uu_b and Uu_gamma}, we have that, for any $J \in \Jac^\delta(\ol{X}_v)$, 
\begin{align*}
\wt{\Pp}_{t(J)} \cong \, & \check{\nu}^*\Pp_J^\Car
\\
\cong \, &  \check{\nu}^* \left ( \Dd_{f_2} \left (\Uu^\Car \otimes f_1^*J \right )^{-1} \otimes \Dd_{f_2}(f_1^*J) \otimes \Dd_{f_2}\left (\Uu^\Car \right ) \right )
\\
\cong \, &  \check{\nu}^* \Dd_{f_2} \left (\Uu^\Car \otimes f_1^*J \right )^{-1} \otimes \check{\nu}^* \Dd_{f_2}(f_1^*J) \otimes \check{\nu}^* \Dd_{f_2}\left (\Uu^\Car \right )
\\
\cong \, &  \Dd_{f'_2} \left ( (\id_{\ol{X}} \times  \check{\nu})^* \left (\Uu^\Car \otimes f_1^*J \right) \right )^{-1} \otimes \Dd_{f'_2} \left ( (\id_{\ol{X}} \times  \check{\nu})^*(f_1^*J) \right ) \otimes \Dd_{f'_2} \left ( (\id_{\ol{X}} \times  \check{\nu})^* \Uu^\Car \right )
\\
\cong \, &  \Dd_{f'_2} \left ( (\id_{\ol{X}} \times  \check{\nu})^* \Uu^\Car \otimes (f'_1)^*J \right )^{-1} \otimes \Dd_{f'_2} \left ( (f'_1)^*J) \right ) \otimes \Dd_{f'_2} \left ( (\id_{\ol{X}} \times  \check{\nu})^* \Uu^\Car \right )
\\
\cong \, &  \Dd_{f'_2} \left ( (\nu \times  \id_{\widetilde{\Jac}})_* \wt{\Uu} \otimes (f'_1)^*J \right )^{-1} \otimes \Dd_{f'_2} \left ( (f'_1)^*J) \right ) \otimes \Dd_{f'_2} \left ( (\nu \times  \id_{\widetilde{\Jac}})_* \wt{\Uu} \right )
\\
\cong \, &  \Dd_{f'_2} \left ( (\nu \times  \id_{\widetilde{\Jac}})_*  \left (\wt{\Uu} \otimes \widetilde{f}_1^* \nu^*J \right ) \right )^{-1} \otimes \Dd_{f'_2} \left ( (f'_1)^*J) \right ) \otimes \Dd_{f'_2} \left ( (\nu \times  \id_{\widetilde{\Jac}})_* \wt{\Uu} \right )
\\ 
\cong \, &  \Dd_{\widetilde{f}_2} (\wt{\Uu} \otimes \wt{f}_1^*\nu^*J)^{-1}  \otimes \Dd_{\widetilde{f}_2} \left (\wt{f}_1^*\nu^*J \right ) \otimes \Dd_{\widetilde{f}_2} (\wt{\Uu})
\\ 
\cong \, & \wt{\Pp}_{\nu^*J}
\\ 
\cong \, & \wt{\Pp}_{\hat{\nu}(J)}.
\end{align*}
This implies that $t = \hat{\nu}$, thus completing the proof.

\end{proof}

We can now study the image of $\check{\nu}_*(\check{\Ll}^{\boxtimes n})$ under \eqref{eq Cartan FM restricted to Pic}.  

\begin{proposition}
	One has the isomorphism
	$$
	\Phi^{\Car}(\check{\nu}_*(\check{\Ll}^{\boxtimes n})) \cong \hat{\nu}^* \wt{\Phi}(\check{\Ll}^{\boxtimes n}),
	$$
	and furthermore, $\hat{\nu}^* \wt{\Phi}(\check{\Ll}^{\boxtimes n})$ is a complex supported on degree $g$ given by $\hat{\nu}^* \Oo_{(\check{\Ll}^{\boxtimes n})}$.
\end{proposition}

\begin{proof}
Let us also consider the following maps
$$
\xymatrix{
 & \Jac^{\overline{0}}(\wt{X}_v) \times \Jac^{\delta}(\ol{X}_v) \ar[ld]_{\pi'_1} \ar[rd]^{\pi'_2}   &
\\
\Jac^{\overline{0}}(\wt{X}_v) &   & \Jac^{\delta}(\ol{X}_v),
}
$$
and observe that 
\begin{itemize}
 \item $\pi_2' = \pi_2 \circ (\check{\nu} \times \id_{\Jac})$, 
 \item $\pi_1' = \widetilde{\pi}_1 \circ (\id_{\widetilde{\Jac}} \times \hat{\nu})$, 
 \item $\pi_1 \circ (\check{\nu} \times \id_{\Jac}) = \check{\nu} \circ \pi_1'$, and
 \item $\widetilde{\pi}_2 \circ (\id_{\widetilde{\Jac}} \times \hat{\nu}) = \hat{\nu} \circ \pi_2'$.
\end{itemize}
Recalling Lemma \ref{lm relation of Poincares}, that $\check{\nu}$ is an injection and that $\hat{\nu}$ is flat by Lemma \ref{lm fibration picards}, one has the following, 
\begin{align*}
\Phi^{\Car}(\check{\nu}_*(\check{\Ll}^{\boxtimes n})) = & R\pi_{2,*} \left ( \pi_1^*\check{\nu}_* ( \check{\Ll}^{\boxtimes n}) \otimes \Pp^\Car \right )
\\
\cong & R\pi_{2,*} \left ( R(\check{\nu} \times \id_{\Jac})_* (\pi'_1)^*(\check{\Ll}^{\boxtimes n}) \otimes \Pp^\Car \right )
\\
\cong & R\pi_{2,*}R \left ( \check{\nu} \times \id_{\Jac})_* ((\pi'_1)^*(\check{\Ll}^{\boxtimes n}) \otimes (\check{\nu} \times \id_{\Jac})^*\Pp^\Car \right )
\\
\cong & R\pi_{2,*}R(\check{\nu} \times \id_{\Jac})_*  \left( (\pi'_1)^* (\check{\Ll}^{\boxtimes n}) \otimes (\id_{\widetilde{\Jac}} \times \hat{\nu})^*\wt{\Pp} \right )
\\
\cong & R\pi'_{2,*} \left ((\pi'_1)^*(\check{\Ll}^{\boxtimes n}) \otimes (\id_{\widetilde{\Jac}} \times \hat{\nu})^*\wt{\Pp} \right )
\\
\cong & R\pi'_{2,*} \left ( (\id_{\widetilde{\Jac}} \times \hat{\nu})^* \widetilde{\pi}_1^*(\check{\Ll}^{\boxtimes n}) \otimes (\id_{\widetilde{\Jac}} \times \hat{\nu})^* \wt{\Pp} \right )
\\
\cong & R\pi'_{2,*} (\id_{\widetilde{\Jac}} \times \hat{\nu})^* \left ( \widetilde{\pi}_1^*(\check{\Ll}^{\boxtimes n}) \otimes \wt{\Pp} \right )
\\
\cong & \hat{\nu}^* R\widetilde{\pi}_{2,*} \left (\widetilde{\pi}_1^*(\check{\Ll}^{\boxtimes n}) \otimes \wt{\Pp} \right )
\\
\cong & \hat{\nu}^* \wt{\Phi}(\check{\Ll}^{\boxtimes n}).
\end{align*}
Finally, recalling that the usual Fourier--Mukai transform on $\Jac^0(X) \times \Jac^{\delta/n}(X)$ sends the line bundle $\check{\Ll}$ to the (complex supported on degree $g$ given by) sky-scraper sheaf $\Oo_{\hat{\Ll}}$, we have that $\Phi^\Car(\check{\nu}_*\check{\Ll}^{\boxtimes n})$ is (the complex supported on degree $g$ given by)  
$$
%\begin{equation}\label{eq dual sheaves}
\hat{\nu}^* \wt{\Phi}(\check{\Ll}^{\boxtimes n}) \cong \hat{\nu}^* \Oo_{(\hat{\Ll}^{\boxtimes n})},
%\end{equation}
$$
and the proof is complete.
\end{proof}

Recalling Proposition \ref{pr filtration for L-normalized}, we arrive to the main result of the section, which shows that our $\BBB$-brane $\CCar(\Ll)$ and our $\BAA$-brane $\UUni(\Ll)$ are related under the Fourier--Mukai integral functor $\Phi^\Car$.

\begin{corollary} \label{co support of the dual brane in the Hitchin fibre}
For every $v \in \V^\nod$, the support of the image under $\Phi^\Car$ of the $\BBB$-brane $\CCar(\Ll)$ restricted to a Hitchin fibre $h^{-1}(v)$, is the support of our $\BAA$-brane $\UUni(\Ll)$ restricted to the open subset of the (dual) Hitchin fibre given by the locus of invertible sheaves,
$$
\supp \left ( \Phi^\Car \left ( \check{\nu}_*(\check{\Ll}^{\boxtimes n}) \right ) \right ) = \Uni(\Ll) \cap \Jac^\delta(\ol{X}_v).
$$
\end{corollary}

\begin{remark} \label{rk torsion bundles} 
Corollary \ref{co support of the dual brane in the Hitchin fibre} points at a duality between $\CCar(\Ll)$ and $\UUni(\Ll)$.
The piece of work \cite{TorsionBundles} has provided evidence for this fact via a Fourier--Mukai transform. Indeed, when $X$ is an unramified cover of a smooth curve $Y$, there exist submanifolds of $\CCar(\Ll)$ and (unions of) $\UUni(\Ll)$ covering two Fourier--Mukai dual branes on the moduli space of Higgs bundles on $Y$. 
\end{remark}

\section{Parabolic subgroups and branes on the singular locus}
\label{sc singular locus branes}

Cartan branes are the simplest example of branes supported on the singular locus $\M^{\sing}_n$ of the moduli space of Higgs bundles. In this section we first study the other hyperholomorphic subvarieties covering the singular locus, and, in second place, we construct Lagrangian subvarieties paired to them.

\subsection{Levi subgroups and the singular locus}
\label{sc Levi BBB}

Consider the $n$-tuple of positive integers 
\[
\ol{r} = \left ( r_1, \stackrel{m_1}{\dots}, r_1, \dots, r_s, \stackrel{m_s}{\dots}, r_s  \right )
\]
where $0<r_1 < \cdots < r_s$ and set $|\ol{r}| = \sum_{\ell=1}^s m_\ell r_\ell$ and $m_{\ol{r}} = \sum_{\ell = 1}^s m_\ell$. Any maximal rank reductive subgroup of $\GL(n,\CC)$ is conjugate to
\[
\L_{\ol{r}}:=\GL(r_1,\CC)\times \stackrel{m_1}{\dots} \times \GL(r_1,\CC) \times \cdots\times \GL(r_s,\CC)\times \stackrel{m_s}{\dots} \times \GL(r_s,\CC),
\]
where $|\ol{r}| = n$. Denote by $\M_{\ol{r}}\subset \M_{n}$ the image of the moduli space $\M_{\L_{\ol{r}}}$ of $\L_{\ol{r}}$-Higgs bundles. Note that $\M_{\ol{r}}$ is the image of the injective morphism,
\[
c_{\, \ol{r}} : \Sym^{m_1}(\M_{r_1}) \times \dots \times \Sym^{m_s}(\M_{r_s}) \longrightarrow \M_n.
\]

\begin{remark}
In particular, $\Car = \M_{(1,\stackrel{n}{\dots}, 1)}$ for $\ol{r} = (1,\stackrel{n}{\dots}, 1)$.
\end{remark} 

The same arguments as in the case of Cartan subgroups show that this is a complex subscheme in all three complex structures of $\M_n$.

\begin{proposition}
Fix $\ol{r}$ with $|\ol{r}| = n$, and consider $\M_{\ol{r}}\subset\M_{n}$. This subvariety is complex in all three complex structures $\Gamma_1,\Gamma_2,\Gamma_3$, and therefore hyperholomorphic.
\end{proposition}

The union of these subvarieties covers the singular locus of the moduli space of Higgs bundles. 

\begin{proposition}[\cite{simpson2}, Section 11]
The singular locus is the locus of strictly polystable bundles,  
$$
\M^{\sing}_n = \bigcup_{|\ol{r}|=n}\M_{\ol{r}}.
$$
\end{proposition}
	
%Now, if $(r_1,\dots,r_s)$ and $(l_1,\dots, l_m)$ are such that for all $j=1,\dots, m$ there is an $n_j$ such that $\sum_{i=1}^{n_j} r_i=\sum_{k=1}^j l_j$, then $\M_{\ol{r}}\subset \M_{\ol{l}}$. In particular
%$$
%\M^{\sing}_n \subset \bigcup_{\sum r_i=n}\M_{\ol{r}}= \bigcup_{r_1\leq r_2, r_1+r_2=n} \M_{(r_1,r_2)}.
%$$
%%%%%%%%%%%%%%%%%%%%%%%%%%%%%%%%%%%%%%%%%%%%%%%%%%%%%%%%%%%%%%%%

Denote
\[
\H_{\ol{r}}:=\Sym^{m_1}(\H_{r_1})\times\dots\times \Sym^{m_s}(\H_{r_s})
\]
and, relating the invariant polynomials of $\L_{\ol{r}}$ with those of $\GL(n,\CC)$, construct an injective morphism
\[%\begin{equation} \label{eq H_r to H_n}
\begin{array}{ccc}
\H_{\ol{r}} &\longrightarrow& \H_n.
\end{array}
\]%\end{equation}
Note that the image $h(\M_{\ol{r}})$ under the Hitchin map of $\M_{\ol{r}}$
coincides with the image of $\H_{\ol{r}}$ under this morphism.
%where $\beta = (\beta^1,\dots,\beta^s) \in \H_{\ol{r}}$, being $\beta^\ell \in \Sym^{m_\ell}(\H_{r_\ell})$ given by $\beta^\ell = (b^\ell_{1}, \dots, b^\ell_{m_\ell})_{\sym}$ with $b^\ell_{i}=(b^\ell_{i1},\dots,b^\ell_{i r_i})$ and $b^\ell_{ij}\in H^0(X,K^j)$, is sent to the point in $\H_n$
%$$
%\left (\sum_{\ell = 1, i=1}^{\ell = s, i = m_\ell}b^\ell_{i1},\dots,\sum_{{\sum_k j_k=i}}\prod_{k\in I}b^\ell_{kj_k},\dots, \prod_i b^\ell_{i,r_i} \right ) \in \H_n.
%$$
Write $\H_r^\sm$ for the locus of smooth spectral curves in the Hitchin base and set
\[
\V_{\ol{r}} := \Sym^{m_1}(\H^\sm_{r_1}) \times \dots \times \Sym^{m_s}(\H^\sm_{r_s}).
\]
Every point $\beta \in \V_{\ol{r}}$ is of the form $\beta  = (\beta^1,\dots,\beta^s)$, being $\beta^\ell \in \Sym^{m_\ell}(\H_{r_\ell})$ given by $\beta^\ell = (b^\ell_{1}, \dots, b^\ell_{m_\ell})_{\sym}$ with $b^\ell_{i}=(b^\ell_{i1},\dots,b^\ell_{i r_i})$ and $b^\ell_{ij}\in H^0(X,K^j)$. 

Denote by $\Delta_r$ the big diagonal of $\Sym^r(\H_r)$ and set
\[
\V_{\ol{r}}^\red := \left ( \Sym^{m_1}(\H^\sm_{r_1})  \setminus \Delta_{r_1} \right ) \times \dots \times \left ( \Sym^{m_s}(\H^\sm_{r_s}) \setminus \Delta_{r_s} \right ).
\]
Proceeding as in Lemmas \ref{lemma spectral curve} and \ref{lemma spectral curve red}, one can prove that, for every $\beta \in \V^\red_{\ol{r}}$, the corresponding spectral curve $\spec_\beta$ is reduced  with $m_{\ol{r}}$ irreducible components $\ol{X}_{b^1_1},\dots, \ol{X}_{b^1_{m_1}}, \dots ,\ol{X}_{b^s_1},\dots, \ol{X}_{b^s_{m_s}}$, which are in turn spectral curves for $b^\ell_i\in\H_{r_i}$. Observe that the corresponding $r_i$-to-$1$ spectral covers $\pi^\ell_i : \ol{X}_{b^\ell_i} \to X$ coincide with the restriction of $\pi : \ol{X}_\beta \to X$ to each of the irreducible components, so that 
$$
\xymatrix{
	\spec_{b^\ell_i}\ar@{^(->}[r]^{\iota^\ell_i}\ar[dr]_{\pi^\ell_{i}}&\spec_\beta 
	\ar[d]^{\pi}\\
&	X
	}
$$
commutes. 
We consider the nodal locus $\V^{\nod}_{\ol{r}}\subset \V_{\ol{r}}$, consisting of  spectral curves with smooth irreducible components intersecting only in nodal points. Note that $\V_{\ol{r}}^\nod$ is dense within $\V_{\ol{r}}$ and the latter is dense in $\H_{\ol{r}}$. 
\begin{lemma}\label{lm intersection irred comp}
Let $\beta \in \V_{\ol{r}}$. Then $D^{\ell,\ell'}_{i,i'}=\spec_{b^\ell_i}\cap\spec_{b^{\ell'}_{i'}}$ is a divisor linearly equivalent to $K^{r_ir_{i'}}$, thus of length
$2r_ir_{i'}(g-1)$.

Moreover, if $\beta \in \V_{\ol{r}}^\nod$, then the divisor of singularities of $\ol{X}_\beta$ has simple points, and is given by the union $D = \bigcup_{\ell < \ell',i<i'} D^{\ell,\ell'}_{i,i'}$, and the normalization is $\nu_\beta:\wt{X}_\beta=\ol{X}_{b^1_1} \sqcup \dots \sqcup \ol{X}_{b^1_{m_1}} \sqcup \dots \sqcup \ol{X}_{b^s_1} \sqcup \dots \sqcup \ol{X}_{b^s_{m_s}} \to \ol{X}_\beta$. %If further $\beta \in \V_{\ol{r}}^\nod$, the divisor of singularities is the union $D = \bigcup_{\ell<\ell', i<i'} D^{\ell,\ell'}_{i,i'}$ and consists of simple points.
\end{lemma}

\begin{proof}
To see the first statement, deform the plane curve $\spec_{b_i}$ to $\lambda^{r_i}=0$. Then, the intersection with $\spec_{b_{i'}}$ is the vanishing locus of a section of $\pi^*K^{r_{i'}}$ along $X$ with multiplicity $r_{i}$. The second and third statements are obvious.
\end{proof}

The following proposition is proved as Proposition \ref{pr intersection of the BBB with the Hitchin fibre}.

\begin{proposition}\label{prop spectral data levi}
Let $\beta \in \V_{\ol{r}}^{\nod}$, and let $\delta_i=(r_i^2-r_i)(g-1)$. Then 
$$
h^{-1}(\beta) \cap \M_{\ol{r}} =\check{\nu} \left ( \Jac^{\ol{\delta}}(\wt{X}_\beta)\right ),
$$
where $\ol{\delta}=(\delta_1,\stackrel{m_1}{\dots},\delta_1,\dots,\delta_s, \stackrel{m_s}{\dots}, \delta_s)$.
\end{proposition}

\subsection{Parabolic subgroups and complex Lagrangian subvarieties}\label{sect BAA para}

Let $\P_{\ol{r}}$ be the parabolic subgroup whose Levi subgroup is $\L_{\ol{r}}$. Recall that the corresponding unipotent radical is $\U_{\ol{r}}=[\P_{\ol{r}},\P_{\ol{r}}]$, and one has the identificaltion $\P_{\ol{r}} = \L_{\ol{r}} \ltimes \U_{\ol{r}}$. In this section we construct Lagrangian subvarieties associated to the choice of the a parabolic subgroup of the form $\P_{\ol{r}}$.

%define another submanifold $\Uni^{\ol{r}}(\ol{E}) \subset \M_n$ associated to the choice of a a parabolic subgroup $\P_{\ol{r}}$ whose Levi subgroup is $\L_{\ol{r}}$, as well as some vector bundles $\ol{E} = (E_1,\dots,E_s)$ on $X$ with $\rk\  E_i=r_i$. 

Denote the locus of those Higgs bundles reducing its structure group to $\P_{\ol{r}}$ by
\[%\begin{equation}\label{eq Par}
\Par_{\, \ol{r}} = \left\{(E,\varphi)\in\M_n\ \left| \begin{array}{l}
\exists \, \sigma \in H^0(X,E/\P_{\ol{r}}),\\
\varphi\in H^0(X,E_\sigma(\mathfrak{p}_{\ol{r}})\otimes K).
\end{array}\right. \right\}.
\]%\end{equation}

Proceeding as in Proposition \ref{pr Bor 1}, one can prove that $\Par_{\ol{r}}$ coincides with the preimage of $\H_{\ol{r}}$ under the Hitchin map. 

\begin{proposition} \label{pr Par is the preimage of H_r}
One has the following,
\[%\begin{equation}\label{eq Par if preimage of Hr}
\M_n\times_{\H} \H_{\ol{r}} = \Par_{\, \ol{r}}.
\]%\end{equation}
\end{proposition}

For $\ol{r}=(r_1, \stackrel{m_1}{\dots}, r_1, \dots, r_s, \stackrel{m_s}{\dots}, r_s)$ fixed, we say that $J$ is an ordering of $\ol{r}$ if it is an ordering of the positive integers $\{ r_1, \stackrel{m_1}{\dots}, r_1, \dots, r_s, \stackrel{m_s}{\dots}, r_s \}$. Let us denote by $\Ord_{\, \ol{r}}$ the set of orderings of $\ol{r}$. Given $\beta \in \V_{\ol{r}}$ one can consider an ordering $J_\beta = \left ( \ol{X}_1, \dots, \ol{X}_m \right )$ of the irreducible components of $\ol{X}_\beta$ where the $j$-th element is the irreducible component indexed by $b^{\ell_j}_{i_j}$. Accordingly with $J_\beta$ denote by $\pi_j$ the restriction to the irreducible component $\ol{X}_j$ of the projection $\pi : \ol{X}_\beta \to X$ and abbreviate by $r_j := r_{\ell_j}$ the degree of the covering of $X$ associated to $\ol{X}_j \stackrel{r_j:1}{\to} X$. We say that the ordering $J_\beta$ respects $J$ if we obtain $J$ out of $J_\beta$ by setting at the $j$-th position, the rank $r_j$ of the corresponding irreducible component $\ol{X}_j$.

In order to state the equivalent to Proposition \ref{pr description line bundles on V^nod} some extra care is needed, as the fact that the integers $r_i$ are different, breaks the symmetry we have in the case of Borel groups, so that orderings of the indices need to be taken into account.
\begin{proposition}\label{prop Par}
Let $\beta \in \V_{\ol{r}}$ be associated to a spectral curve $\ol{X}_\beta$ has $m = m_{\ol{r}}$ irreducible components $\ol{X}_{b^1_1},\dots, \ol{X}_{b^1_{m_1}}, \dots ,\ol{X}_{b^s_1},\dots, \ol{X}_{b^s_{m_s}}$. Let $(E,\varphi)$ be a Higgs bundle whose spectral data consists of a line bundle $L$ over $\ol{X}_\beta$.  For any ordering of $\ol{r}$, $J \in \Ord_{\ol{r}}$, and any ordering $J_\beta$ of the irreducible components of $\ol{X}_\beta$ respecting $J$, one can choose canonically a filtration
$$
(E_{J_\beta})_\bullet \, : \, 0 \subsetneq (E_{1},\varphi_{1}) \subsetneq \dots \subsetneq (E_{m},\varphi_{m}) = (E,\varphi),
$$
such that
$$
(E_{j},\varphi_{j})/(E_{j-1},\varphi_{j-1}) = (\pi_{j,*}L|_{\ol{X}_j}\otimes K^{-R^{J}_j},\varphi_{j}/\varphi_{j-1}),
$$
where $R^{J}_j= \sum_{k\geq j+1}r_{k}r_{j}$ depends only on $J$ and $\varphi_{j}/\varphi_{j-1}$ is determined by $\ol{X}_j$ as explained in \eqref{eq Higgs field}. Note that in the expression of $R^J_j$ $r_k$ may be equal to $r_j$.
\end{proposition}

%Due to the monodromy of the Hitchin fibration, one can not choose coherently an ordering for each of the spectral curves in $\V_{\ol{r}}$. We can however fix once and for all an ordering $I$ of $\{ r_1, \dots, r_s \}$ appearing in $\ol{r}$. Let us say that an ordering $J$ of the irreducible components of $\ol{X}_\beta$ respects $I$ if all the irreducible components with a given rank $r_i$ are consecutive in $J$. Note that every ordering $J$ respecting $I$ gives the same values for the $R_j$, so one can say that they only depend on $I$ in that case.

Given a line bundle of trivial degree $\Ll \in \Jac^0(X)$ and a point $x_0$, we define for every $r$,
\[
\hat{\Ll}_r := \Ll \otimes \Oo(x_0)^{(r-1)(g-1)}.
\]
Recall from \eqref{eq Hitchin section} the description of the Hitchin section of $h : \M_r \to \H_r$ associated to a line bundle of degree $(r-1)(g-1)$ over $X$. Observe that one has
\[
\Sigma_{\hat{\Ll}_{r}} : \H_r \to \M_r.
\]

For a given $\Ll \in \Jac^0(X)$, we define the subvariety of $\Par_{\, \ol{r}}$
\begin{equation}\label{eq Uni par Hitchin section}
\Uni_{\, \ol{r}}(\Ll) := \left\{
(E,\varphi) \in \Par_{\, \ol{r}} \ 
\left|\ 
\begin{array}{l}
\exists \, \sigma\in H^0(X,E/\P_{\ol{r}}), \, \textnormal{and} \, J \in \Ord_{\, \ol{r}} :\\
\varphi\in H^0(X,E_\sigma(\mathfrak{p}_{\ol{r}})\otimes K);\\
(E_\sigma, \varphi) / \U_{\ol{r}} := \Sigma_{\hat{\Ll}_{r_1}}(\beta) \otimes K^{-R^J_1} \boxplus \dots \boxplus \Sigma_{\hat{\Ll}_{r_m}}(\beta) \otimes K^{-R^J_m}.
\end{array}
\right.
\right\}.
\end{equation}

Using Proposition \ref{prop Par}, we can study the spectral data of the Higgs bundles contained in $\Uni_{\, \ol{r}}(\Ll)$.

\begin{proposition}\label{prop Uni_r and Hitchin fibration}
One has the following,
\begin{enumerate}
\item \label{it Uni_r surjects to V_r} The restriction of $\Uni_{\, \ol{r}}(\Ll)$ to $\V_{\ol{r}}^\nod$ is surjective.

\item \label{it fibers Uni_r} Let $\beta \in \V_{\ol{r}}^{\nod}$, we have that
$$
\Uni_{\, \ol{r}}(\Ll) \cap h^{-1}(\beta) \cap \Jac(\spec_\beta) = \hat{\nu}^{-1}(\hat{\Ll}_{r_1}, \stackrel{m_1}{\dots}, \hat{\Ll}_{r_1}, \dots, \hat{\Ll}_{r_s}, \stackrel{m_s}{\dots}, \hat{\Ll}_{r_s}).
$$
\end{enumerate}
\end{proposition}

We are now in a position to prove that $\Uni_{\, \ol{r}}(\Ll)$ is Lagrangian, hence a suitable choice for the support of a $\BAA$-brane.

\begin{theorem}\label{thm Uni r L lagrangian}
The subscheme $\Uni_{\, \ol{r}}(\Ll)$ is Lagrangian.
\end{theorem}

\begin{proof}
It is enough to prove that the open subset $\Uni_{\ol{r}}(\Ll)^\nod$ given by the the restriction of $\Uni_{\ol{r}}(\Ll)$ to $\V_{\ol{r}}^\nod$, is Lagrangian. 

Fix $\beta \in \V_{\ol{r}}^{\nod}$. By  Proposition \ref{prop Uni_r and Hitchin fibration} \eqref{it fibers Uni_r} the intersection of $\Uni_{\ol{r}}(\Ll) \cap h^{-1}(\beta)$ with $\Jac(\ol{X}_\beta)$ is non-empty, so there are Higgs bundles $(E,\varphi)$ which have a line bundle as spectral data. Those $(E,\varphi)$ are stable hence are smooth points in $\Uni_{\, \ol{r}}(\Ll)$. With all this, we  prove isotropicity as we did in Proposition \ref{pr uni isotropic}.

By Lemma \ref{lm intersection irred comp} and Lemma \ref{lm fibration picards}, there is an exact sequence
$$
0\longrightarrow (\CC^\times)^{\delta_{\ol{r}}- s + 1} \longrightarrow \Jac(\spec_b) \stackrel{\hat{\nu}}{\longrightarrow}\Jac(\wt{X}_b)\longrightarrow 0
$$
where $\delta_{\ol{r}}=\sum_{1\leq i<j\leq s}2r_ir_j(g-1)$. It then follows by Proposition \ref{prop Uni_r and Hitchin fibration} \eqref{it fibers Uni_r}  that $$\dim\Uni_{\ol{r}}(\Ll) \cap h^{-1}(\beta)=\dim \Jac(\spec_\beta)=\delta_{\ol{r}}- s + 1.$$ 
By
 Proposition \ref{prop Uni_r and Hitchin fibration}  \eqref{it Uni_r surjects to V_r}, one has that
$$
\V_{\ol{r}}^{\nod}\subset h(\Uni_{\, \ol{r}}(\Ll)),
$$
and recall that $\V^\nod_{\ol{r}}$ is dense in $\H_{\ol{r}}$, so they both have the same dimension. Since there are smooth points in $\Uni_{\, \ol{r}}(\Ll)$, it follows that the dimension is 
\begin{align*}
\dim \Uni^{\ol{r}}(\Ll) & =\delta_{\ol{r}}-s+1+\dim \H_{\ol{r}}=\delta_{\ol{r}}-s+1+\sum_{i}(r_i^2(g-1)+1)
\\
& = n^2(g-1)+1,
\end{align*} 
which is half of the dimension of $\M_n$.
\end{proof}

\begin{remark}\label{rk assumption reasonable}
After endowing $\M_{\ol{r}}$ and $\Uni_{\, \ol{r}}(\Ll)$ with a suitable hyperk\"ahler and flat bundle respectively, one obtains a pair of branes, the first of type $\BBB$ and the later of type $\BAA$. Propositions \ref{prop spectral data levi} and \ref{prop Uni_r and Hitchin fibration} indicate that there exists a duality between these branes similar to the one we envisaged  in Section \ref{sc duality} for the case of Borel subgroups ({\it i.e.} $\ol{r} = (1, \dots, 1)$).
\end{remark}

%%%%%%%%%%%%%%%%%%%%%%%%%%%%%%%%%%%%%%%%%%%%%%%%%%%%%%%%%%%%%%%%%%%%%%%%%

Now, it is also possible to construct more general unitary Lagragian submanifolds, even in the absence of Hitchin sections. The key is to use very stable bundles to produce Lagrangian multisections of the Hitchin map. Given a vector bundle $E$ we say, after Drinfeld \cite{drinfeld, laumon}, that $E$ is {\it very stable} if it has no non-zero nilpotent Higgs fields. This implies that  $E$ is stable \cite[Proposition 3.5]{laumon} (provided $g \geq 2$). Furthermore, very stable bundles are dense within the moduli space of vector bundles \cite[Proposition 3.5]{laumon}. Gathering the results of Pauly and the second author (see \cite[Theorem 1.1 and Corollary 1.2]{verystable}) with  the remark \cite[Corollary 7.3]{TorsionBundles}, one gets

\begin{theorem}\label{tm very stable}
Let $E$ be a stable bundle. Then, $E$ is very stable if and only if the Lagrangian subvariety given by the embedding
\[
\map{H^0(X,\End(E) \otimes K)}{\M_n}{\phi}{(E,\phi),}{}
\]
is provides a Lagrangian multisection of the Hitchin fibration ({\it i.e.} the restriction of the Hitchin fibration to $H^0(X,\End(E) \otimes K) \hookrightarrow \M_n$ is finite and surjective).  
\end{theorem}

Set $m = m_{\ol{r}}$. Given an ordering $J \in \Ord_{\ol{r}}$ consider an $m$-tuple of very stable vector bundles over $X$, $\ol{E} = (E_1,\dots,E_{m})$, whose $i$-th element has $\rk\  E_i=r_i$ given by the $i$-th position of $J$. Denote $\deg E_i=e_i$ and %$d_i=e_i+(r_i^2-r_i)(g-1)$ and 
$$
f_i^J=e_i+(r_i^2-r_i)(g-1) + 2R_i^J(g-1)
$$
where $R_{i}^J=\sum_{k\geq i+1}r_{j_i}r_{j_k}$ are defined as in Proposition \ref{prop Par}. From now on, we shall assume that the choice of $J$ and $\ol{E}$ is done under the following numerical condition on the degrees $e_i$.

\begin{assumption}\label{assumption}
Let $\ol{e} = (e_1, \dots, e_m)$ be an $m$-tuple of integers and pick $J \in \Ord_{\, \ol{r}}$. Suppose that, for all subset $I\subset \{1,\dots,m\}$, there  are inequalities
\begin{equation}\label{eq dis}
\sum_{i\in I} f_i^J>(r_I^2-r_I)(g-1),
\end{equation}
where $r_I=\sum_{i\in I}r_{j_i}$, and when $I = \{1,\dots,m\}$ one has the equality
\[
\sum_{i = 1}^m f_i^J = (n^2-n)(g-1).
\]
\end{assumption}

Given a $m$-tuple of very stable bundles $\ol{E}$ whose degrees $\ol{e}$ satisfy Assumption \ref{assumption}, we define the following subvariety of $\Par_{\, \ol{r}}$,

\begin{equation}\label{eq Uni par}
\Uni_{\ol{r}}(\ol{E}) := \left\{
(E,\varphi)\ 
\left|\ 
\begin{array}{l}
\exists \, \sigma\in H^0(X,E/\P_{\ol{r}}):\\ 
\varphi\in H^0(X,E_\sigma(\mathfrak{p}_{\ol{r}})\otimes K);\\
E_\sigma / \U_{\ol{r}} := E_{\L_{\ol{r}}}\cong\bigoplus_{i=1}^m E_i.
\end{array}
\right.
\right\}.
\end{equation}

In what follows we prove that $\Uni^{\ol{r}}(\ol{E})$ is a Lagrangian submanifold. As in the case of $\Uni_{\, \ol{r}}(\Ll)$, this is proven through the study the associated spectral data.

Consider restriction of the Hitchin map $h$ to $\Uni_{\, \ol{r}}(\ol{E})$. After Proposition \ref{pr Par is the preimage of H_r} one has that the image is contained in $\H_{\ol{r}}$,
$$
h : \Uni_{\ol{r}}(\ol{E}) \longrightarrow \Hr.
$$

Before we can give the analogous to Proposition \ref{pr filtration for L-normalized}, we need an intermediate result.

%Assume that $\ol{E}$ satisfies Assumption \ref{assumption}. Let $b\in \H_{\ol{r}}^{\nod}$, and let $\hat{\Ll}=(\hat{\Ll}_{b_1},\dots,\hat{\Ll}_{b_s})$ be as in Assumption \ref{assumption} $ii)$. Given an ordering $J$, set  
\begin{proposition}\label{prop fibers dual brane}
Let $\beta \in \V^\nod$. Assume that $\ol{E}$ satisfies Assumption \ref{assumption} and denote by $S_{i,\beta}$ the finite set of Higgs bundles over $\beta$ admitting $E_i$ as underlying vector bundle. Let $\Ss_{i,\beta}$ the associated set of spectral data over $\ol{X}_\beta$ associated to each of the Higgs bundles in $S_{i,\beta}$. For each $J \in \Ord_{\ol{r}}$, pick
\begin{equation}\label{eq orderings}\hat{\Ll}^{J}_{\ol{E}, \beta}= ( \Ll_1 \otimes\pi^*_1 K^{R^{J}_1},\dots, \Ll_m\otimes \pi^*_m K^{R_m^J}),
\end{equation}
where $\Ll_i \in \Ss_{i, \beta}$. Let us denote by $\Ss^J_\beta$ the set of all tuples of the form \eqref{eq orderings}. 

Assume that $\ol{E}$ satisfies Assumption \ref{assumption} $i)$. Let $b\in \H_{\ol{r}}^{\nod}$, and let $\mathcal{O}rd_s$ denote the set of orderings of $\{1,\dots,s\}$. For each $J\in\mathcal{O}rd_s$, let $\hat{\Ll}^J$ be as in \eqref{eq orderings}. Then, $\Uni_{\, \ol{r}}(\ol{E}) \cap h^{-1}(b) \cap \Jac^{\ol{d}}(\spec_\beta)$ is either empty or
$$
\Uni_{\, \ol{r}}(\ol{E}) \cap h^{-1}(b) \cap \Jac^{\ol{d}}(\spec_\beta) =\bigcup_{\hat{\Ll}^{J}_{\ol{E}, \beta}\in\Ss^J_{\beta}} \hat{\nu}^{-1}(\hat{\Ll}^{J}_{\ol{E}, \beta})
$$
where we identify $\Jac^{\ol{d}}(\spec_b)$ with an open subset of $h^{-1}(b)$ and define
 $$
 \hat{\nu}:\Jac^{\ol{d}}(\spec_b)\longrightarrow\Jac^{\ol{d}}(\wt{X}_b)
 $$
 to be the pullback map.
\end{proposition}

\begin{proof}
After checking that \eqref{eq dis} ensures the stability of the points of $\Uni_{\ol{r}}(\ol{E}) \cap h^{-1}(b)$, the proof follows as in Proposition \ref{pr filtration for L-normalized}.
\end{proof}

Continuing the parallelism with $\Uni(\Ll)$, we next prove Lagrangianity of the submanifold $\Uni_{\ol{r}}(\ol{E})$.

\begin{theorem}\label{thm Unir lagrangian}
Under Assumption \ref{assumption}, the subscheme $\Uni_{\ol{r}}(\ol{E})$ is Lagrangian.
\end{theorem}

\begin{proof}
The proof is analogous to that of Theorem \ref{thm Uni r L lagrangian}.
\end{proof}

\begin{remark}\label{rk other degrees}
For the sake of clarity, we have chosen to work with the moduli space of degree $0$ Higgs bundles. Note however that the subvarieties $\Mr$ and $\Uni^{\ol{r}}$ make sense in a larger context. Indeed, consider the moduli space of rank $n$, degree
$d$ Higgs bundles $\M_X(n, d)$ with $(n,d)\neq 1$. Then, $\M_X(n, d)^{sing}\neq\emptyset$, and so there will exist partitions $\ol{r}$ of $n$ for which $\Mr\neq\emptyset$. Note that in that case the (semi)stability condition for torsion free sheaves should then be modified accordingly.
%$$
%\deg_Z \mathscr{F}_Z \, > \, dn_Z-(n^2n_Z-n_Z^2n )(g - 1) \quad \textnormal{(resp. $\geq$),} 
%$$
%where $n_Z = \rk(\pi_* \mathcal{O}_Z)$.
\end{remark}
%\begin{remark}\label{rk integers hiperholo}
%The bundles $\hat{\Ll}_{b_i}$ are built in a similar way to $\hat{\Ll}$ from \eqref{eq definition of hatLl}. Indeed, the restriction of $\Oo$ to $\spec_{b_i}$ is again $\Oo$, which we twist by $\pi_{b_i}^*\Oo(m_i)$. The fact that $\hat{\Ll}_{b_i}$ is a pullback from a line bundle on $X$ ensures isotropicity of $\Mrr(\ol{m_i})$, as the ${\pi}_{b_i,*}(\hat{\Ll}_{b_i})$ has underlying vector bundle independent of $b_i$. On the other hand, the fact that there is a spectral datum over each point of the Hitchin base, ensures Lagrangianity.
%\end{remark}


\begin{thebibliography}{ZZZZZ}

\bibitem[Al]{Alper} J. Alper, {\it  Good moduli spaces for Artin stacks}, Annales de l’Institut Fourier, 63(6):2349–2402, 2013.

\bibitem[AHH]{AHH} J. Alper, D. Halpern-Leistner and J. Heinloth, {\it Existence of moduli spaces for algebraic stacks}. Preprint: arXiv:1812.01128 [math.AG]


\bibitem[AIK]{altman&iarrobino&kleiman}
A. Altman, A. Iarrobino, and S. Kleiman, 
{\it Irreducibility of the compactified Jacobian.} In Real
and complex singularities. 
Proc. Ninth Nordic Summer School/NAVF Sympos. Math., Oslo,
1976), pages 1--12. Sijthoff and Noordhoff, Alphen aan den Rijn, 1977.

\bibitem[AK]{altman&kleiman}
A. Altman and S. Kleiman, 
{\it Compactifying the Picard scheme, I.} 
Adv. Math. 35 (1980), 50--112.

\bibitem[Ar1]{arinkin1} 
D. Arinkin,
{\it Cohomology of line bundles on compactified Jacobians.}  
Math. Res. Lett. 18 (2011), no. 06, 1215--1226.

\bibitem[Ar2]{arinkin2} 
D. Arinkin,
{\it Autoduality of compactified Jacobians for curves with plane singularities.} J. Algebraic Geometry 22 (2013), 363--388.

\bibitem[BS1]{BS2}
D. Baraglia and L. P. Schaposnik,
{\it Real structures on moduli spaces of Higgs bundles.}
Adv. Theo. Math. Phys. {\bf 20} (2016), 525--551.

\bibitem[BS2]{BS3}
D. Baraglia and L. P. Schaposnik,
{\it Cayley and Langlands type correspondences for orthogonal Higgs bundles.} 
Trans. Amer. Math. Soc. 371 (2019), 7451--7492. 

\bibitem[BNR]{BNR} 
A. Beauville, M. S. Narasimhan and S. Ramanan,
{\it Spectral curves and the generalised theta divisor.}
J. Reigne Angew. Math {\bf 398} (1989), 169--179.

\bibitem[BCFG]{biswas&calvo&franco&garciaP} 
I. Biswas, L.A. Calvo, E. Franco and O. Garci\'ia-Prada,
{\it Involutions of the moduli spaces of $G$-Higgs bundles over elliptic curves.} J. Geom. Phys. 142
(DOI: 10.1016/j.geomphys.2019.03.014).

\bibitem[BG]{BGP}
I. Biswas and O. Garc\'{i}a-Prada,
{\it Anti-holomorphic involutions of the moduli spaces of Higgs bundles.}
J. l'\'{E}c. Polytech. Math. {\bf 2} (2015), 35--54.

\bibitem[BLR]{neron} 
S. Bosch, W. Lutkebohmert and M. Raynaud, 
{\it N\'eron models.}
{Springer--Verlag} (1980).

\bibitem[BBS]{BBS} 
S. Bradlow, L. Branco and L. P. Schaposnik. 
{\it Orthogonal Higgs bundles with singular spectral curves.}
\texttt{arXiv:1909.03994[math.AG]}.

\bibitem[B]{B}
L. Branco. 
{\it Higgs bundles, Lagrangians and mirror symmetry.}
{DPhil Thesis}, University of Oxford, 2017.

\bibitem[CL]{CL} 
P-H. Chaudouard and G. Laumon. 
{\it Un th\'eor\`eme du support pour la fibration de Hitchin.}
Ann. Inst. Fourier (Grenoble) {\bf 66} (2016), no. 2, 711–727.

\bibitem[C]{corlette} 
K. Corlette. 
{\it  Flat G-bundles with canonical metrics.} 
J. Diff. Geom., {\bf 28}(3) (1988) 361--382.

\bibitem[dC]{dC} 
M.A.A. de Cataldo. 
{\it A support theorem for the Hitchin fibration: the case of $\SL_n$.}
Compositio Math. 153 (6) (2017), 1316--1347.



\bibitem[DG]{donagi&gaitsgory}
R. Donagi and D. Gaitsgory,
{\it The gerbe of Higgs bundles.}
Transform. Groups {\bf 7} (2002), 109--153.

\bibitem[DP]{donagi&pantev}
R. Donagi and T. Pantev,
{\it Langlands duality for Hitchin systems.}
Invent. Math. {\bf 189} (2012), 653--735.

\bibitem[Do]{donaldson} 
S.K. Donaldson. 
{\it Twisted harmonic maps and the self-duality equations.}
Proc. London Math. Soc. (3), {\bf 55}(1) (1987), 127--131.

\bibitem[Dr]{drinfeld} V.G .Drinfeld,  {\it Letter to P.Deligne}, 22nd June 1981.

\bibitem[Es]{esteves} 
E. Esteves. 
{\it Compactifying the relative Jacobian over families of reduced curves.}
Trans. Am. Math. Soc. {\bf 353} (2001) no 8, 3045--3095.   


\bibitem[EGK]{EGK} 
E. Esteves, M. Gagn\'e, and S. Kleiman. 
{\it Autoduality of the compactified Jacobian.} 
J. London Math. Soc. (2), 65(3) (2002), 591--610.

\bibitem[FGOP]{TorsionBundles} 
E. Franco and P. Gothen and A. Oliveira and A. Pe\'on-Nieto.
{\it Unramidied covers and branes on the Hitchin system.} \texttt{arxiv:1802.05237[math.AG]}

\bibitem[FJ]{franco&jardim} 
E. Franco and M. Jardim.
{\it Mirror symmetry for Nahm branes.}
\texttt{arXiv:1709.01314[math.AG]}.

\bibitem[Ga]{gaiotto}
D. Gaiotto,
{\it S-duality of boundary conditions and the Geometric Langlands program.}
Proc. Symp. Pure Math. 98 (2018) 139--180.


\bibitem[GR]{garciaprada-ramanan}
O. Garcia-Prada, S. Ramanan. 
{\it Involutions and higher order automorphisms of Higgs moduli spaces.} To appear in Proc. London Math. Soc.

\bibitem[Gr]{EGA} 
A. Grothendieck. 
{\it EGA IV, Quatri\`eme partie.}
Publ. Mat. de IHES,  {\bf 32} (1967), 5--361. 

\bibitem[H-L]{HalpernInst} D. Halpern-Leistner, {\it On the structure of instability in moduli theory}.Preprint, arXiv: 1411.0627.

%\bibitem[Ha]{haiman} 
%M. Haiman.
%Hilbert schemes, polygraphs and the Macdonald positivity conjecture.  
%J. Amer. Math. Soc., 14(4):941--1006, 2001.
 
\bibitem[HMP]{HMP} T. Hausel, A. Mellit, and D. Pei.
{\it Mirror symmetry with branes by equivariant Verlinde formula.}
In, {Geometry and Physics: Volume I: A Festschrift in honour of Nigel Hitchin}.
Oxford University Press (2018).

\bibitem[HT]{HT} 
T. Hausel and M. Thaddeus.
{\it Mirror symmetry, Langlands duality, and the Hitchin system.}
Invent. Math. {\bf 153} (2003), 197--229.

\bibitem[He]{Heinloth}J. Heinloth, {\it Hilbert-Mumford stability on algebraic stacks and applications to G-bundles on curves}, \'EPIGA, Volume 1 (2017), Article Nr. 11.

\bibitem[HS]{heller&schaposnik} 
S. Heller and L.P. Schaposnik. 
{\it Branes through finite group actions.}
J. Geom. Phys. 129 (2018), 279--293.

\bibitem[H1]{hitchin-self} 
N.J. Hitchin.
{\it The self-duality equations on a Riemann surface.}
Proc. London Math. Soc. (3), {\bf 55}(1) (1987) 59--126.

\bibitem[H2]{hitchin_duke} 
N.J. Hitchin. 
{\it Stable bundles and integrable systems.}
Duke Math. J. {\bf 54}, no 1 (1987), 91--114. 

\bibitem[H3]{HitTeich} 
N.J. Hitchin. 
{\it Lie groups and Teichm\"uller space.}
Topology {\bf 31}(3) (1992) 449--473.

\bibitem[H4]{hitchin_char}
N.J. Hitchin, 
{\it Higgs bundles and characteristic classes.}
Arbeitstagung Bonn 2013, {Progr. Math.}, \textbf{319}, Birkhäuser/Springer, Cham, 2016, 247--264. 

\bibitem[H5]{hitchin_spinors}
N.J. Hitchin,
{\it Spinors, Lagrangians and rank 2 Higgs bundles.}
Proc. London Math. Soc. {\bf 115} (2017), 33--54.

\bibitem[HL]{huybrechts&lehn}
D. Huybrechts, and M. Lehn,
{\it The geometry of moduli spaces of sheaves.}
\textit{Cambridge University Press} (2010).

\bibitem[KW]{kapustin&witten}
A. Kapustin and E. Witten,
{\it Electric-magnetic duality and the geometric Langlands program.}
Commun. Number Theory Phys. {\bf 1} (2007), 1--236.

\bibitem[Ka]{kass}
J. L. Kass, 
{\it Autoduality holds for a degenerating abelian variety.}
Res Math Sci (2017) 4:27.

\bibitem[KM]{knudsen&mumford} 
F. Knudsen and D. Mumford. 
{\it The projectivity of the moduli space of stable curves I: Preliminaries on “det” and “Div”.}
Math. Scand. {\bf 39} (1976), 19--55. MR 55:10465   


\bibitem[La]{laumon} G. Laumon.
{\it Un analogue global du c\^{o}ne nilpotent.}
Duke Math. J. \textbf{57}, 647--671 (1988).


\bibitem[ML1]{maoli1} 
M. Li,
{\it Construction of the Poincar\'e sheaf on the stack of rank two Higgs bundles of $\PP^1$}, 
\texttt{arXiv:1709.05292[math.AG]}.

\bibitem[ML2]{maoli2}
M. Li, 
{\it Construction of the Poincar\'e sheaf for higher genus curves},
\texttt{arXiv:1801.02993[math.AG]}.

\bibitem[MRV1]{melo0} 
M. Melo, A. Rapagnetta and F. Viviani, 
{\it Fine compactified Jacobians of reduced curves.}
Trans. Amer. Math. Soc. 369 (2017), no. 8, 5341--5402.


\bibitem[MRV2]{melo1} 
M. Melo, A. Rapagnetta and F. Viviani, {\it  Fourier--Mukai and autoduality for compactified Jacobians. I.}
J. Reigne Angew. Math. 755 (2019) 1--65- 

\bibitem[MRV3]{melo2} 
M. Melo, A. Rapagnetta and F. Viviani.
{\it Fourier--Mukai and autoduality for compactified Jacobians. II.}
Geom. Topol. 23 (5) (2019), 2335--2395.

\bibitem[Mu]{mukai} 
S. Mukai. 
{\it Duality between $\Dd(X)$ and $\Dd(\hat{X})$ with its application to Picard sheaves.} 
Nagoya Math. J.,  \textbf{81} (1981), 153--175.

\bibitem[Ne]{Newstead}
P. Newstead,{\it Introduction to Moduli Problems and Orbit Spaces}, Narosa Publishing House, 1978 (reprint
2012).

\bibitem[N]{Nitin} 
N. Nitsure.
{\it Moduli space of semistable pairs on a curve.} 
Proc. London Math. Soc. (3) {\bf 62} (1991), no. 2, 275–300. 

\bibitem[Pa]{pantev} T. Pantev, private communication.


\bibitem[PP]{verystable} C. Pauly and A. Pe\'on-Nieto. 
{\it Very stable bundles and properness of the Hitchin map.}
Geom. Dedicata \textbf{198} (1) (2019) 143--148.


\bibitem[Sch]{Schaub} 
D. Schaub. 
{\it Courbes spectrales et compactifications de Jacobiennes.} 
Mathematische Zeitschrift {\bf 227} (2) (1998) 295--312.
\bibitem[Se1]{SeshadriQuot} C. S. Seshadri,{\it Quotient spaces modulo reductive algebraic groups}, Ann. of Math. (2) 95 (1972), 511–556; errata, ibid. (2) 96 (1972), 599.

\bibitem[Se2]{Seshadri} 
C.S. Seshadri. 
{\it Fibr\'es vectoriels sur les courbes alg\'ebriques. }
Ast\'erisque {\bf 96}, (1982).

\bibitem[Si]{Si} 
C.T. Simpson. 
{\it Higgs bundles and local systems.} 
Publ. Math., Inst. Hautes Etudes Sci. {\bf 75} (1992), 5--95.

\bibitem[Si1]{simpson1} 
C.T. Simpson. 
{\it Moduli of representations of the fundamental group of a smooth projective variety I.} 
Publ. Math., Inst. Hautes Etud. Sci. {\bf 79} (1994), 47--129.

\bibitem[Si2]{simpson2} 
C.T. Simpson.
{\it Moduli of representations of the fundamental group of a smooth projective variety II.}
Publ. Math., Inst. Hautes Etud. Sci. {\bf 80} (1995), 5--79.

\bibitem[SYZ]{SYZ}
A. Strominger, S.T. Yau and E. Zaslow.
{\it Mirror Symmetry is $T$-duality.}
Nucl. Phys. B {\bf 479} (1996), 243--259.
    







%\bibitem[BS1]{BS1}
%D. Baraglia and L. P. Schaposnik,
%Higgs bundles and $(A,B,A)$-branes.
%Comm. Math. Phys. {\bf 331} (2014), 1271--1300.

%\bibitem[BGH1]{biswas&garciaP&hurtubise1} 
%I. Biswas, O. Garc\'ia-Prada and J. Hurtubise, 
%Pseudo-real principal G-bundles over a real curve.
%J. Lond. Math. Soc. {\bf 93} (2016), 47--64.

%\bibitem[BGH2]{biswas&garciaP&hurtubise2} 
%I. Biswas, O. Garc\'ia-Prada and J. Hurtubise, 
%Branes and Langlands duality. 
%Preprint arXiv: 1707.00392.


%\bibitem[BtD]{broker&tomDieck} 
%T. Br\"ocker and T. tom Dieck, 
%Representations of Compact Lie Groups.
%\textit{Springer--Verlag} (1985).

%\bibitem[D]{DDec} 
%R. Donagi. 
%Decomposition of spectral covers. 
%Journ{\'e}es de g{\'e}om{\'e}trie alg{\'e}brique d'Orsay, 1992, \textit{Ast\'erisque} 218, 145--175 (1993).

%\bibitem[D]{D} R. Donagi.
%Spectral covers. 
%\textit{MSRI Series} 28, (1995).

%\bibitem[GW]{garciaP&wilkins}  
%O. Garc\'ia-Prada and G. Wilkin, 
%Action of the mapping class group on character varieties and Higgs bundles. 
%Preprint arXiv:1612.02508.


%\bibitem[Il]{ilusie} 
%L. Illusie. 
%Conditions de finitude relatives. 
%Lecture Notes in Mathematics, vol. 225, \textit{Springer-Verlag} (1971), 222--273.

%\bibitem[KR]{KR71} 
%B. Kostant and S. Rallis. 
%Orbits and representations associated with symmetric spaces.
%Am. J. of Math. {\bf 93} (1971), 753--809.

%\bibitem[T]{Thaddeus par} 
%M. Thaddeus.
%Variation of moduli of parabolic Higgs bundles.
%J. Reigne Angew. Math, {\it 547} (2002), 1--14.

\end{thebibliography}
\end{document}